\documentclass[12pt,a4paper]{amsart}
\usepackage{geometry}           
\usepackage{ucs} 
\usepackage[utf8x]{inputenc}
\usepackage{xcolor}

\usepackage{graphicx}
\usepackage{amsmath,amssymb}
\usepackage{bbm}

\usepackage{mathtools}

\theoremstyle{plain}
\newtheorem{theorem}{Theorem}
\newtheorem{corollary}[theorem]{Corollary}
\newtheorem{proposition}[theorem]{Proposition}
\newtheorem{lemma}[theorem]{Lemma}
\theoremstyle{remark}

\newtheorem{remark}[theorem]{Remark}

\newcommand{\var}{\mathop{\rm Var}}
\newcommand{\cov}{\mathop{\rm Cov}}
\newcommand{\jij}{J_{i\leftarrow j}}

\newcommand{\red}[1]{{\color{black} #1}}

\title[Universality of mean-field  equations]{Universality of the mean-field equations of  networks of Hopfield-like  neurons}

\author[Faugeras]{Olivier Faugeras$^*$}
\address{$^*$Université Côte D'Azur, Inria}
\email{Olivier.Faugeras@inria.fr}
\author[Tanr\'e]{Etienne Tanr\'e$^\dagger$}
\address{$^\dagger$Université Côte D'Azur, Inria, CNRS, LJAD}
\email{Etienne.Tanre@inria.fr}

\date{\today}
\usepackage[foot]{amsaddr}

\begin{document}
	\maketitle

\tableofcontents
\newpage
\begin{abstract}
We revisit the problem of characterising the mean-field limit of a network of Hopfield-like neurons. Building on the previous works 
\red{of Ben Arous and Guionnet}
%of \cite{guionnet:95,ben-arous-guionnet:95,guionnet:97} and  \cite{dembo_universality_2021} 
we establish for a large class of 
networks of Hopfield-like neurons, i.e. rate neurons, the mean-field equations on a time interval $[0,\,T]$, $T>0$, of the thermodynamic
limit of these networks, i.e. the limit when the number of neurons goes to infinity. 
%Unlike all previous work, except  \cite{dembo_universality_2021}, 
\red{Here,}
we do not assume that the synaptic weights describing the connections between the  neurons are i.i.d. as zero-mean Gaussians. 
The limit equations are stochastic and very simply described in terms of two functions, a ``correlation'' function noted $K_Q(t,\,s)$ 
and a ``mean'' function noted $m_Q(t)$. The ``noise'' part of the equations is a linear function of the Brownian motion, which is obtained 
by solving a Volterra equation of the second kind whose resolving kernel is expressed as a function of $K_Q$. We give a constructive proof 
of the uniqueness of the limit equations. We use the corresponding algorithm for an effective computation of the functions $K_Q$ and $m_Q$, 
given the weights distribution. Several numerical experiments are reported.
\end{abstract}
\textit{MSC2020 AMS classification: 60F10, 92B20, 60B10}

\section{Introduction}\label{Sec:Intro}
We revisit the problem of characterizing the limit of a network of 
Hopfield neurons. Hopfield \cite{hopfield:82} defined a large class of 
neuronal networks and 
\red{
	analysed   their ability to 
	perform computation  \cite{hopfield:84,hopfield-tank:86}.}
%
%characterized some of their computational
%properties \cite{hopfield:84,hopfield-tank:86}, i.e. their ability to 
%perform computations. 

Inspired by his work, Sompolinsky and 
co-workers studied the thermodynamic limit of these networks when 
the interaction term is linear \cite{crisanti-sompolinsky:87} using the 
dynamic mean-field theory developed in 
\cite{sompolinsky-zippelius:82} for symmetric spin glasses. The 
method they use is a functional integral formalism developed in particle 
physics and produces the self-consistent mean-field equations of the 
network. This was later extended to the case of a nonlinear interaction 
term, the nonlinearity being an odd sigmoidal function 
\cite{sompolinsky-crisanti-etal:88}.  Using the same formalism the 
authors established the self-consistent mean-field equations of 
the network and the dynamics of its solutions which featured a chaotic 
behavior for some values of the network parameters. 
More recently, a derivation of the equations using the path integral formalism was provided by the same authors \cite{crisanti-sompolinsky:18}.

A little later the 
problem was picked up again independently by mathematicians, although in a different context. Ben Arous and 
Guionnet applied large deviation techniques to study the
thermodynamic limit of a network of spins interacting linearly through  
i.i.d. centered Gaussian weights. The intrinsic spin dynamics 
(without interactions) is a stochastic differential equation, unlike in the case of the work of Sompolinsky and collaborators where the only source of uncertainty lies in the interaction weights. 
Ben Arous and Guionnet proved
that the annealed (averaged) law of the empirical measure satisfies a 
large deviation principle and that the good rate function of this large 
deviation principle achieves its minimum value at a unique non 
Markovian measure~\cite{guionnet:95,ben-arous-guionnet:95,guionnet:97}. 
They also proved averaged propagation of chaos results. 
The advance provided by the work of Ben Arous and Guionnet is, beyond the crucial fact that they set the problem on firm mathematical ground, 
that they provide a description of the limit dynamics in a much more precise way than all previous authors.

Moynot and 
Samuelides \cite{moynot-samuelides:02} adapted their work to the case 
of a network of Hopfield neurons with a nonlinear interaction term, the 
nonlinearity being a sigmoidal function, and proved similar results in the 
case of discrete time. The intrinsic neural dynamics is the gradient of a 
quadratic potential. They did not provide a description of the limit dynamics akin to the one of Ben Arous and Guionnet.

Even more recently, van Meegen, K\"uhn and Helias, in \cite{meegen-kuhn-etal:21} nicely unified
 the Large Deviation approach of Ben Arous and Guionnet with the field-theoretical 
one of Sompolinsky and collaborators by showing that the Large Deviation rate function is in effect equal to the effective action arising in the field theory. 
They did not deal correctly with the limit dynamics though in the sense that they did not state in the limit equations the dependency of the effective interaction term with the Brownian motion.

Adapting the work of Ben Arous and Guionnet, Cabana and Touboul, in a series of papers \cite{cabana-touboul:13,cabana-touboul:18,cabana-touboul:18b}, have proposed a number of generalisations 
to more complicated setups. They applied the theory of Large Deviations and showed the existence of a unique probability measure where the rate function is zero, they 
identified the corresponding 
measure but did not provide a pathwise description of this probability measure. In effect their description, like that of all previous authors,  is based upon a certain process called the effective 
interaction process which is dependent of the Brownian motion under the limit law but this dependency is not analysed. 

In \cite{dembo_universality_2021} the authors proved that the result in \cite{guionnet:95,ben-arous-guionnet:95,guionnet:97} is in fact universal, i.e. that their limit probability 
measure is independent of the distribution of the weights under some mild assumptions.

In this article we show how to go beyond the ``self-consistent'' approach and provide a complete description of the effective interaction term as a function of the Brownian motion under the limit law. Thanks to the result in \cite{dembo_universality_2021} our results also hold for a much larger class of synaptic weights distributions than the Gaussian case.

The paper is organised as follows. In Section~\ref{sect:models} we present the classes of spin and neuron models that we analyse and the network equations 
that govern their activity. We then state the problem at hand,
namely the characterisation of these network equations when the network size grows 
arbitrarily large, the  thermodynamic limit. In Section~\ref{sect:previous} we recall the two main results which inspired our work
and
briefly state our three main results. In Section~\ref{sect:nonzero} we state and prove our first two main results, the stochastic 
equations of the dynamics of the thermodynamic limit of the networks and the characterisation of the diffusion part of these equations. Section~\ref{sect:f1} is a 
pause to analyse the special case when the interaction between the neurons is constant, albeit random, not depending on the states of the neurons. 
Section~\ref{sect:uniqueness} contains the proof of our third main result, the uniqueness of the thermodynamic limit. Being constructive this proof opens 
up the door to numerical experiments many of them presented in Section~\ref{sect:numerics}.

%%%%%%%%%%%%%%%%%
%%%%%%%%%%%%%%%%%
%%%%%%%%%%%%%%%%%
\section{The models}\label{sect:models}
We consider a filtered probability space $(\Omega, \mathcal{G},(\mathcal{G}_t),p)$ equipped with a 
sequence of one-dimensional Brownian motions $(B^i_t)_{t \geq 0, i\in\mathbb{N}}$.

The object of interest is a fully connected network of Hopfield-like neurons, in short Hopfield neurons. A Hopfield neuron is an abstraction of a real neuron 
where the state of the neuron is described by the values of its membrane potential. 
The dynamics of an isolated Hopfield neuron is driven by a certain  stochastic differential equation where the ``noise'' is assumed to account for the probabilistic 
effects present in various biological actors within the neuron, such as the ion channels. Our view of a Hopfield neuron is rather large since it encompasses 
the spin-glass model of Ben Arous and Guionnet ~\cite{guionnet:95,ben-arous-guionnet:95,guionnet:97}. When we will need to differentiate between the 
two models we will refer to the H-model for the biological one and to the S-model for the spin-glass model. This is described in Section~\ref{subsec:single}. 
Hopfield neurons can interact, one says that they are connected. The biological support of these interactions is the notion of synaptic connection. 
In Hopfield neurons, the synaptic connection between a neuron $j$ and a neuron $i$ is extremely simplified and boils down to the product of two numbers, 
the synaptic weight $\jij$ that represents the maximum strength of the influence of $j$ over $i$, and a function of the membrane potential $X^j$ of neuron $j$ 
which represents its activity, its firing rate.  This is described in Section~\ref{subsec:network}.
%%%%%%%%%%%%%%%%%
%%%%%%%%%%%%%%%%%
\subsection{Single neuron}\label{subsec:single}
Consider a Hopfield neuron with membrane potential $Y_t$. Its dynamics is described by a stochastic differential equation
\begin{equation}\label{eq:isolated}
\begin{cases}
dY_t  &=  g(Y_t)dt+\lambda dB_t\\
\text{Law of } Y_0  &=  \mu_0,
\end{cases}
\end{equation}
where $\lambda$ is a nonzero real parameter.
The function $g$ is chosen such that \eqref{eq:isolated} has a unique strong solution which 
does not explode in finite time. In the cases studied in this paper we have $g(y)=-\nabla U(y)$.
\begin{remark}\label{rem:g}
\noindent
\begin{description}
	\item[H-model] \textbf{$U(y) = \frac{\alpha}{2} y^2$}, i.e $g(y)=-\alpha y$, with $\alpha > 0$. This is the H-model where the membrane potential is an 
	Orstein-Uhlenbeck process
\[
Y_t = Y_0\exp(-\alpha t) + \lambda \int_0^t \exp(-\alpha (t - s)) dB_s.
\]
	\item[S-model] \textbf{$U$ is a confining potential}. This is the S-model studied by Ben Arous and Guionnet 
	\cite{ben-arous-guionnet:95}. 
	The probability law $\mu_0$ has a compact support $[-A,\,A]$ for some $A > 1$ which does not put any mass on the boundary $\{-A,\, A\}$.
	The function $U$ is $C^2$ on 
	the interval $]-A,A[$, and $U$ tends to infinity when $|x| \to A$ sufficiently fast to ensure that
\[
\lim_{\underset{<}{|x| \to A}} k_U(x) = +\infty,
\]
where
\begin{equation}\label{eq:kU}
k_U(x) = 2 \int_0^x \exp 2U(y)\left(\int_0^y \exp -2U(z) \,dz \right)\,dy.
\end{equation}
An example of such a function $U$ is 
\begin{equation}
U(y)=-\log(A^2-y^2).
\label{eq:U}
\end{equation}
It follows that
\begin{equation}
\label{eq:g}
g(y) = - \frac{2y}{A^2-y^2}.
\end{equation}
Assuming that the support of the initial condition \(\mu_0\) is included in \((-A,A)\),
it is recalled in \cite{ben-arous-guionnet:95} that the solution does 
not explode in finite time i.e. that if $T_\varepsilon = \inf \{t \, | \, |Y_t| \geq A - \varepsilon \}$, 
then for all $T$, $p(T_\varepsilon \leq T) \leq \frac{\exp T}{1+k_U(A-\varepsilon)}$ so that 
$p(\lim_{\varepsilon \downarrow 0} T_\varepsilon = +\infty)=1$. More details about explosion times can be found in the classical books by Rogers and Williams \cite{rogers_diffusions_2000}, Borodin \cite{borodin_stochastic_2017}, and Peskir and Shiryaev \cite{peskir_optimal_2006}.

Figure~\ref{fig:Ug} shows the functions $U(x)$ and $g(x)$ when $A=2$.
\begin{figure}[h]
\centerline{
\includegraphics[width=0.4\textwidth]{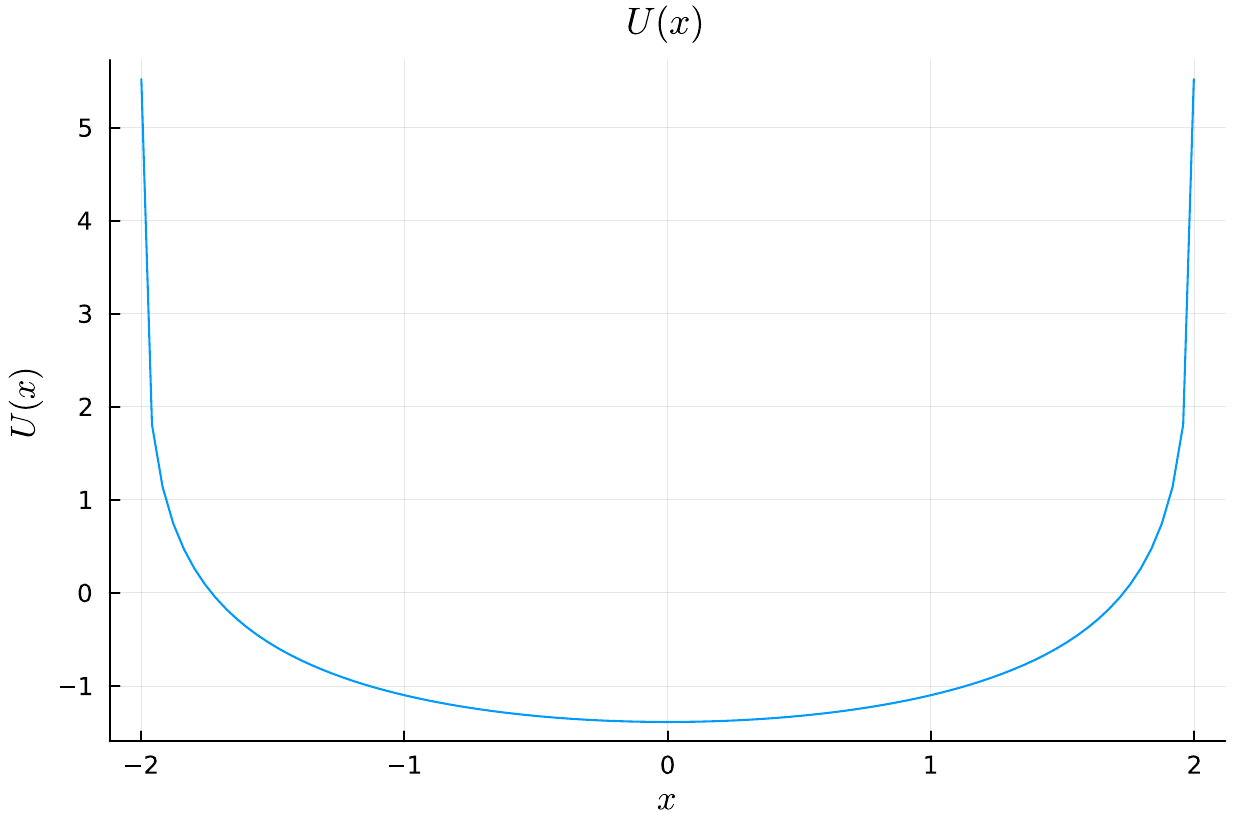}\hspace{0.5cm}\includegraphics[width=0.4\textwidth]{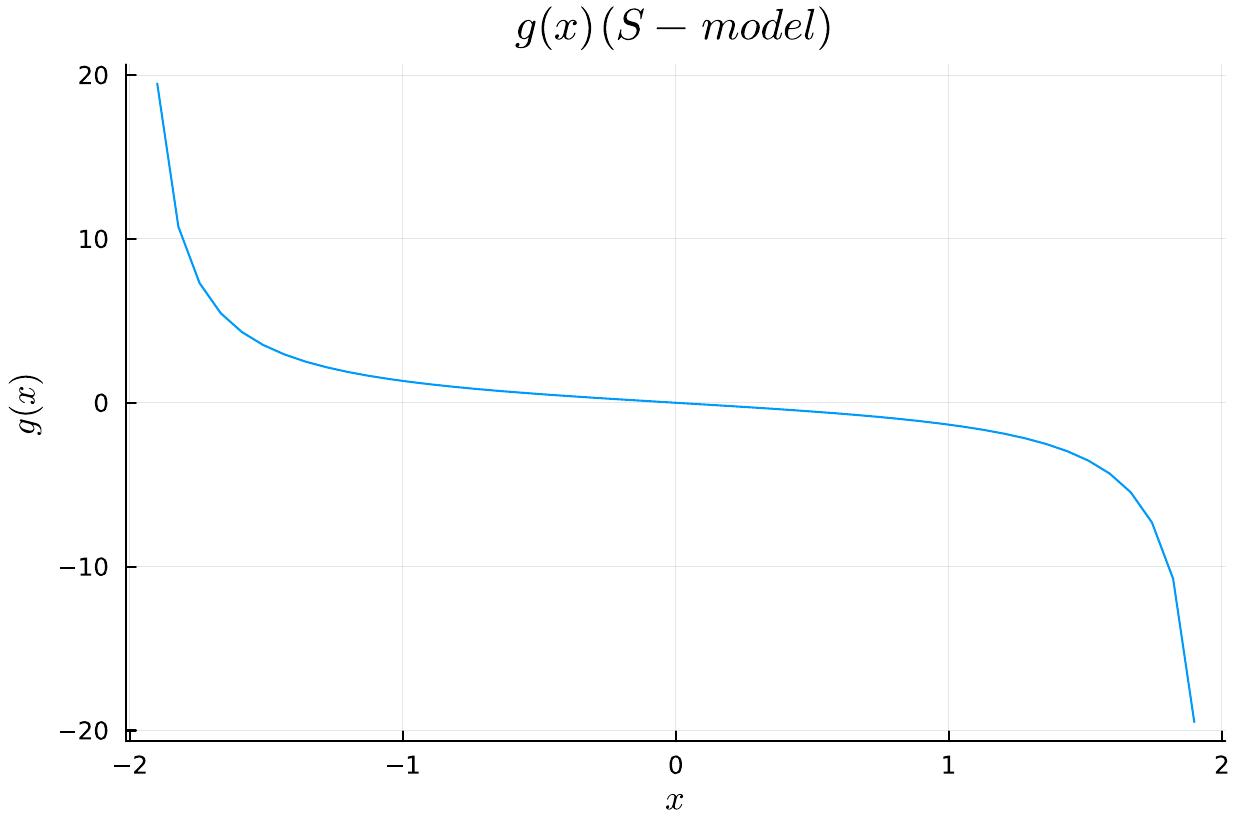}
}
\caption{S-model ($A=2$). Left: the function $U(x)$ on the interval $(-2,\,2)$. Right: the function $g(x)=-\nabla U(x)$}
\label{fig:Ug}
\end{figure}
\end{description}
\end{remark}
Let us note $P$ the law of the solution to \eqref{eq:isolated}. The previous remark shows that in the
 case $g(y)=-\alpha y$, $P$ is a probability measure on the set $\mathcal{C}_T$ of continuous functions on $[0,T]$ with values in $\mathbb{R}$ while in the second case it is a probability measure on the set $\mathcal{C}_T^A$ of continuous functions on $[0,T]$ with values in $[-A,\,A]$. In both cases we note $\mathcal{F}_t$ the $\sigma$-algebra generated by $(Y_s, 0 \leq s \leq t)$, 
  and $\mathbb{F}$ the corresponding filtration. In the sequel we note $\mathcal{C}$ either $\mathcal{C}_T$ or $\mathcal{C}_T^A$ and $P$ is a probability measure on $(\mathcal{C},  \mathcal{F}_T, \mathbb{F})$.
%%%%%%%%%%%%%%%%%%%%%%%%%%%
%%%%%%%%%%%%%%%%%%%%%%%%%%%
\subsection{Network of neurons}\label{subsec:network}
We consider $N$ Hopfield neurons and connect each one of them to all others (and to itself). The strengths of the connections are random, described by the synaptic weights. We introduce a second probability space $(\Theta,\mathcal{A},\gamma)$ and i.i.d. 
random variables $(\jij)_{i,\,j=1,\cdots,N}$ on $\Theta$ which are, under $\gamma$, 
i.i.d. as 
\begin{equation}\label{eq:synweights}
\jij \simeq \mathcal{N}\left(\frac{J}{N},\frac{\sigma^2}{N}\right),\ i,\,j=1,\cdots,N.
\end{equation} 
We note $\mathbf{J}$ the $N \times N$ matrix of the $\jij$s and $\mathcal{E}$ the expected value
with respect to the 
probability $\gamma$.
\begin{remark}
The hypothesis that the weights distribution is Gaussian is lifted in Section~\ref{subsec:nongaussian}.
\end{remark}
We consider a continuous function $f: \mathbb{R} \to \mathbb{R}$ which converts the values of the membrane potential into a quantity called the ``activity'' of the neuron. A classical example of such 
\textit{activity} is the firing rate,
i.e. the number of action potentials produced per second. 

The network equations are
\begin{equation}\label{eq:network}
\begin{cases}
dX^i_t & =   \left( g(X^i_t)+\sum_ {j =1}^N \jij f(X^j_t) 
\right)\,dt+\lambda dB^i_t \quad \text{ for } i =  1, \cdots, N\\
\text{Law of } X_0 & =  \mu_0^{\otimes N},
\end{cases}
\end{equation}
where the $(B_t^i)_{t \geq 0}$ are  $N$ independent  Brownian motions on $(\Gamma, \mathcal{G},(\mathcal{G}_t),p)$. 

The S-model studied by Ben Arous and Guionnet features $f(y) = y$ while the $\jij$s are i.i.d. as $\mathcal{N}(0,\frac{\sigma^2}{N})$, the so-called ``soft spin glass'' case. 
We extend this model to the case with non centered weights \(\jij\) in Section~\ref{sect:nonzero}.
 
 The H-model corresponds to the case $g(y) =-\alpha y$, $f$ a sigmoidal function and the $\jij$s are not necessarily centered. For future reference we assume that in this case
 \begin{equation}\label{eq:sigbound}
 | f(x) | \leq C\quad \forall x \in \mathbb{R} 
 \end{equation}
 
Key to the approach in this paper is to relate the solution to $N$ independent copies of \eqref{eq:isolated} with that of \eqref{eq:network}. This relation is obtained through the application of the  Cameron-Martin \cite{cameron-martin:44} and Girsanov \cite{girsanov:60} Theorems which we recall here for future reference.
\begin{theorem}[Cameron-Martin]\label{theo:cameron-martin}
Let $\mathbf {B}$ be an $N$-dimensional Brownian motion on $(\Omega, \mathcal{G}, (\mathcal{G}_t),p)$. Let $\mathbf{\Phi} = (\Phi^1,\cdots,\Phi^N)$ be in $\mathcal{L}^2([0,T];\mathbb{R}^N)$. Define
\[
\zeta_0^T(\mathbf{\Phi},\omega) = \int_0^T \mathbf{\Phi}(t)\, d\mathbf{B}_t - \frac{1}{2} \int_0^T \| \mathbf{\Phi}(t,\omega) \|^2\,dt,
\]
\[
\tilde{\mathbf{B}}_t(\omega) = \mathbf{B}_t(\omega) - \int_0^t \mathbf{\Phi}(s,\omega)\,ds,
\]
\[
d\tilde{p}(\omega) = \exp	\left\{ \zeta_0^T(\mathbf{\Phi},\omega) \right\} dp(\omega)
\]
Then, if
\begin{equation}\label{eq:proba}
\tilde{p}(\Omega) =1
\end{equation}
$\tilde{\mathbf{B}}$ is a new $N$-dimensional Brownian motion defined on $(\Omega,\mathcal{G},\tilde{p})$ w.r.t. to the same filtration $(\mathcal{G}_t)$. Moreover a sufficient condition for \eqref{eq:proba} to hold is
\[
\mathbb{E}_p\left[ \exp\left({\int_0^T  \| \mathbf{\Phi}(t,\omega) \|^2\,dt}\right)\right] < \infty 
\]
\end{theorem}
\begin{remark}
We recall that the quantity $\frac{d\tilde{p}}{dp}$ is called the Radon-Nikodym derivative of the probability measure $\tilde{p}$ w.r.t. the probability measure $p$.
\end{remark}
We also recall that $\mathcal{L}^2([0,T];\mathbb{R}^N)$ denotes the set of $\mathbb{R}^N$-valued $\mathcal{G}_t$-adapted processes $(X_t)_{0\leq t\leq T}$ such that $\int_0^T \| X_t \|^2\,dt < \infty$ a.s..
\begin{theorem}[Girsanov]\label{theo:girsanov}
Consider  $\mathbf{B}$  an $N$-dimensional Brownian motion on \((\Omega, \mathcal{G}, (\mathcal{G}_t),p)\). 
Let $(X_t)_{0 \leq t \leq T}$ be the $N$-dimensional It\^o process given by
\[
{\mathbf X}_t = {\mathbf X}_0 + \int_0^t \mathbf{a}(s)\,d\mathbf{B}_s + \int_0^t \mathbf{b}(s)\,ds
\]
where $\mathbf{a},\,\mathbf{b} \in \mathcal{L}^2([0,T];\mathbb{R}^N)$ 
Let $\mathbf{\Phi} = (\Phi^1,\cdots,\Phi^N)$ be in $\mathcal{L}^2([0,T];\mathbb{R}^N)$.  
Let $\tilde{\mathbf{B}}$ and $\tilde{p}$ be defined as in Theorem~\ref{theo:cameron-martin}. 
If $\tilde{p}(\Omega)=1$ then $\mathbf{X}_t$ is still an It\^o process on the probability space $(\Omega,\mathcal{G},\tilde{p})$ w.r.t.  the same filtration $(\mathcal{G}_t)$ and satisfies
\[
\mathbf{X}_t = \mathbf{X}_0 + \int_0^t \mathbf{a}(s)\,d\tilde{\mathbf{B}}_s + \int_0^t \left(\mathbf{b}(s) + \mathbf{a}(s) \cdot \mathbf{\Phi}(s)\right)\,ds.
\]
\end{theorem}

 As mentioned above, we consider a network of $N$ non-interacting neurons, the state of each one being described by \eqref{eq:isolated}:
 \begin{equation}\label{eq:networkisolated}
\begin{cases}
dY_t^i & =  g(Y_t^i)dt+\lambda dB_t^i
\quad 1 \leq i \leq N\\
\text{Law of } Y_0 & =  \mu_0^{\otimes N}.
\end{cases}
\end{equation}
The law of its solution is the probability measure $P^{\otimes N}$ on $\mathcal{C}^N$.

We relate \eqref{eq:networkisolated} and  \eqref{eq:network} through Theorems~\ref{theo:cameron-martin} and ~\ref{theo:girsanov}.
\begin{lemma}\label{lem:cameronetal}
The system \eqref{eq:network} has a unique weak solution. Its law, noted $P^N(\mathbf{J})$, is 
absolutely continuous with respect 
to the law $P^{\otimes N}$ of the uncoupled system. 
The Radon-Nikodym derivative is
\begin{equation}\label{eq:girsanov1}
\frac{d P^N(\mathbf{J)}}{d P^{\otimes N}} = \exp \left\{ \sum_i \frac{1}{\lambda} \int_0^T 
\sum_{j} \jij f(Y^j_t) 
dB^i_t  - 
\frac{1}{2\lambda^2} \int_0^T \left( \sum_{j} \jij f(Y^j_t) \right)^2\,dt \right\},
\end{equation}
where the $Y^j$s are solutions to \eqref{eq:networkisolated}.
\end{lemma}
\begin{proof}
We apply Theorems~\ref{theo:cameron-martin} and~\ref{theo:girsanov}.

Let $\Phi \in \mathcal{L}^2([0,T];\mathcal{C}^N)$ and define 
\[
\zeta_0^T = \sum_{i=1}^N \int_0^T \Phi^i(t)\,dB^i_t-\frac{1}{2} \int_0^T \Phi^i(t)^2\,dt
\]
The Cameron-Martin Theorem~\ref{theo:cameron-martin} states that if 
\begin{equation}\label{eq:novikov}
\mathbb{E}_{P^{\otimes N}}\left[ e^{\int_0^T \| \Phi(t) \|^2\,dt} \right] < \infty
,  
\end{equation}
we can define for all $t \in [0,T]$
\[
\tilde{B}^i_t = B^i_t -\int_0^t \Phi^i(s) \,ds,
\]
and 
\[
d P^N(\mathbf{J}) = \exp\left\{ \zeta_0^T\right\})dP^{\otimes N},
\]
which are such that $P^N(\mathbf{J})$ is a probability measure on $\mathcal{C}^N$ and $(\tilde{B^i})_i$ an $N$-dimensional Brownian motion under $P^N(\mathbf{J})$. 
Moreover the Girsanov Theorem states that the $Y^i$s satisfy also
\[
Y^i_t = Y_0 + \int_0^t \left( g(Y^i_s)+ \lambda \Phi^i(s) \right) \,ds    + \lambda \tilde{B}^i_t,\, i=1,\cdots,N.
\]
The choice of
\[
\Phi^i(t) = \frac{1}{\lambda} \sum_ {j =1}^N \jij f(Y^j_t),\ i=1,\cdots,N
\]
shows that, if \eqref{eq:novikov} is satisfied, the $Y^i$s, which we rewrite $X^i$s,  are solutions to \eqref{eq:network} with the Brownian motion $\tilde{B}$ under $P^N(\mathbf{J})$ and that the Radon-Nikodym derivative of $P^N(\mathbf{J})$ w.r.t. to $P^{\otimes N}$ is given by \eqref{eq:girsanov1}.

Condition \eqref{eq:novikov} is clearly satisfied for the S-model since the $Y^i$s take values in $]-A,A[$ and for the H-model since $f$  
is assumed to be bounded, see \eqref{eq:sigbound}. 
This proves that \eqref{eq:network} has a unique weak solution whose law is $P^N(\mathbf{J})$. 
 \end{proof}
%%%%%%%%%%%%
%%%%%%%%%%%%
\subsection{Statement of the problem}
\red{The task at hand is to describe the asymptotic behaviour of the system \eqref{eq:network} as \(N\) increases.
 Unlike in previous work we do not assume the weights to be centred Gaussian random variables.
In order to do this, we introduce the following classical tools:}

Let \(X^i\) be the solution of \eqref{eq:network}, we denote \(\hat{\mu}_N\) the associated  empirical measure 
\[
\hat{\mu}_N = \frac{1}{N} \sum_{i=1}^N \delta_{X^i}.
\]
The measure $\hat{\mu}_N$ is a probability measure on the set of trajectories $\mathcal{C}$.  Let us give an example of
how it acts on test functions: Consider the function \(\varphi\) defined  by 
\[
\forall Y \in \mathcal{C}, \quad \varphi(Y) = f(Y_t) f(Y_s).
\]
The integral of \(\varphi\) w.r.t. $\hat{\mu}_N$ yields
\begin{equation}\label{eq:intwrtemp}
\left<\varphi, \hat{\mu}_N\right> = \int \varphi(Y)\,d\hat{\mu}_N(Y) = \frac{1}{N} \sum_{i=1}^N f(X^i_t)f(X^i_s).
\end{equation}

%%%%%%%%%%%%%%%%%%%%%%%%%%%%%
%%%%%%%%%%%%%%%%%%%%%%%%%%%%%
%%%%%%%%%%%%%%%%%%%%%%%%%%%%%
\section{Results}\label{sect:previous}

%\subsection{The results of Ben Arous and Guionnet\texorpdfstring{ \cite{ben-arous-guionnet:95,guionnet:97}}{}}
\subsection{The results of Ben Arous and Guionnet~\cite{ben-arous-guionnet:95,guionnet:97}}
This is the case of the S-model where the weights are centered i.e. $J=0$ in \eqref{eq:synweights}, and $\lambda=1$. Note that our parameter $\sigma$, the standard deviation of the synaptic weights, is the same as their parameter $\beta$, the inverse temperature.
Key to their approach is the averaging of the law $P^N(\mathbf{J})$ with respect to 
$\mathbf{J}$ and replacing the initial problem by that of studying the limit behaviour of the empirical measure under the averaged law $Q^N$ given by 
\[
Q^N = \int_\Theta P^N(\mathbf{J}(\theta))\,d\gamma(\theta) = \mathcal{E}[P^N(\mathbf{J})]
\]
A summary of the essential  results is as follows
\begin{description}
\item[Existence and uniqueness of an averaged (annealed) limit] There exists a unique probability law $Q$ on $\mathcal{C}$ such that the law of the empirical measure $\hat{\mu}_N$ under $Q^N$ converges to $\delta_Q$ (Theorem 2.1 in \cite{guionnet:97}).
\item[Averaged (annealed) propagation of chaos] $Q^N$ propagates chaos, or is $Q$-chaotic, i.e. given $m \geq 1$ and $m$ continuous bounded functions $(\varphi_j)_{j=1, \cdots, m}$ defined on $\mathcal{C}$
\[
\lim_{N \to \infty} \mathbb{E}_{Q^N} \bigg[\prod_{j=1}^m \varphi_j(X^j)\bigg] = \prod_{j=1}^m \mathbb{E}_Q\bigg[ \varphi_j(X)\bigg].
\]
This is not very useful in practice because 
	the use of $Q^N$ in the left hand side requires to average over the weights $\mathbf{J}$ but they also prove
\item[Existence and uniqueness of a quenched limit] 
The law of the empirical measure of the quenched system, i.e. the law of $\hat{\mu}_N$ under $P^N(\mathbf{J})$, converges to $\delta_Q$ for almost all $\mathbf{J}$s (Theorem 2.7 in \cite{ben-arous-guionnet:95}). This of course does not imply a quenched propagation of chaos since the neurons are not exchangeable for almost all weights but 
	they also prove
\item[Quenched propagation of chaos]
If the function $g$ in \eqref{eq:isolated} and \eqref{eq:network} is the gradient of an even function\footnote{Note that this condition is satisfied in the S- and H-models.}  and  the initial law $\mu_0$ is symmetric, then for any set of $m$ continuous bounded functions $(\varphi_j)_{j=1,\cdots,m}$ defined on $\mathcal{C}$
\[
\mathbb{E}_{P^N(\mathbf{J})} \bigg[ \prod_{j=1}^m \varphi_j(X^j)\bigg] 
\xrightarrow[N\to\infty]{p}
\prod_{j=1}^m \mathbb{E}_Q\bigg[\varphi_j(X)\bigg].
\]
This means that for all $\varepsilon > 0$
\[
\lim_{N \to \infty} \gamma\left( \theta \, :  \left| \mathbb{E}_{P^N(\mathbf{J})} \bigg[ \prod_{j=1}^m \varphi_j(X^j)\bigg] - \prod_{j=1}^m \mathbb{E}_Q\bigg[\varphi_j(X)\bigg]\right| > \varepsilon  \right) = 0.
\]
\item[Characterisation of the limit law] The limit law $Q$ is 
characterised by its Radon-Nikodym derivative w.r.t. to the law $P$
\begin{equation}\label{eq:limitRN}
\frac{dQ}{dP} = \int_{\red{\Theta}} \exp \left\{ \int_0^T G_t^Q\red{(\theta)} \,dB_t - \frac{1}{2} \int_0^T (G_t^Q\red{(\theta)})^2\,dt      \right\}\,d\gamma_Q\red{(\theta)},
\end{equation}
where $G^Q$ is under $\gamma_Q$ a centered Gaussian process with covariance function $K_Q(t,s) = \sigma^2 \mathbb{E}_Q[f(X_t) f(X_s)]$. 
Moreover this limit law is unique and solution of the following stochastic system
\begin{equation}\label{eq:limit}
\begin{cases}
X_t  & =  X_0+\int_0^t g(X_s)\,ds +  B_t\\
B_t &=  W_t+  \int_0^t\int_0^s \tilde{K}^s_Q(s,u)\,dB_u\,ds \\
\text{Law of }X & = Q, \quad Q_{\mathcal{F}_0} = \mu_0\\
K_Q(t,s) & =  \sigma^2\mathbb{E}_Q\big[ f(X_t) f(X_s)\big],
\end{cases}
\end{equation}
where $W_t$ is a Brownian motion under $Q$. 

The deterministic function $\tilde{K}^t_Q$ is defined as follows: Let  $(G^Q_t)_{t \in [0,T]}$ be a centered Gaussian process with covariance $K_Q$, one defines
\[
\tilde{K}_Q^t(s,u) = \frac{\int_\Theta \exp\left\{ -\frac{1}{2} \int_0^t (G^Q_v(\theta))^2 \,dv \right\}\,G^Q_s(\theta) G^Q_u(\theta) \,d\gamma_Q(\theta)}{ \int_\Theta \exp \left\{ -\frac{1}{2} \int_0^t (G^Q_v(\theta))^2 \,dv \right\} \,d\gamma_Q(\theta) }  \quad \forall\, 0 \leq s, u\, \leq t,
\]
where $G^Q$ is as before, under $\gamma_Q$, a centered Gaussian process with covariance $K_Q$.
Note that this is equivalent to changing $\gamma_Q$ into $\tilde{\gamma}_Q$ defined by
\[
d\tilde{\gamma}_Q^t(\theta) = \dfrac{\exp\left( -\frac{1}{2} \int_0^t (G^Q_v(\theta))^2 \,dv \right)}{\int_\Theta \exp \left( -\frac{1}{2} \int_0^t (G^Q_v(\theta))^2 \,dv \right)
d\gamma_Q(\theta)}\,d\gamma_Q(\theta),
\]
so that
\[
\tilde{K}_Q^t(s,u) = \int_\Theta G^Q_s(\theta) G^Q_u(\theta) \,d\tilde{\gamma}_Q^t(\theta).
\]
\end{description}

Equations \eqref{eq:limit} look a bit mysterious and our next goals are a) to generalise them to the non-zero mean weights case, i.e. $J \neq 0$ in \eqref{eq:synweights}, and b) to explore the relationship between the Brownian motions $B_t$ and $W_t$ as well as that between $K_Q$ and $\tilde{K}_Q^{t}$.
%%%%%%%%%%%%%%%%%%%%%%%%%
%%%%%%%%%%%%%%%%%%%%%%%%%
%\subsection{The result of Dembo, Lubetzky and Zeitouni \texorpdfstring{\cite{dembo_universality_2021}}{}}\label{subsec:nongaussian}
\subsection{The result of Dembo, Lubetzky and Zeitouni \cite{dembo_universality_2021}}\label{subsec:nongaussian}
This is also the case of the S-model in which the weights are still assumed zero mean but \textbf{without} the assumption that their distribution is Gaussian. They prove that under the following mild assumptions
\begin{equation}\label{eq:dembo-conds}
\begin{cases}
\mathcal{E}[\jij]=0 ,\ \mathcal{E}[\jij^2] = \frac{\sigma^2}{N}, & \\
\exists \varepsilon_0 > 0 \text{ \rm such that } \sup_{i,\,j,\,N} \left\{\mathcal{E}[e^{\varepsilon_0 \sqrt{N} |\jij|}] \right\}< \infty,
\end{cases}
\end{equation}
the empirical measure $\hat{\mu}_N$ converges weakly almost surely in the interactions $\mathbf J$ toward the {\textbf{same} measure $\delta_Q$ as in the Gaussian case.

\subsection{Statement of the results}\label{sect:ourresults}
We have obtained three main results. The first one is a characterization of the dynamics of the trajectories of the limit measure $Q$, this is Theorem~\ref{theo:dynmean}. 
The equations \eqref{eq:dynmean} of this dynamics contain the Brownian motion under $P$ and the Brownian motion under $Q$ which are related by an integral equation. 
Our second contribution is to solve this equation to obtain a simple relation between these two Brownian motions featuring functions of the covariance 
$K_Q$, this is Proposition~\ref{prop:BtWt}. The third contribution is a proof of the uniqueness of the limit measure $Q$, this is Theorem~\ref{theo:uniqueness}. 
It is a constructive Picard-like method that can be used to effectively construct $Q$ and simulate its trajectories. Our proof is different from the one in 
\cite{ben-arous-guionnet:95} which is valid for the S-model for zero-mean weights ($J=0$) and carries over to the nonzero-mean case. Note that, 
because of \cite{dembo_universality_2021}, our results are independent of the distribution of the $\jij$s, provided they satisfy the assumptions \eqref{eq:dembo-conds}.
%%%%%%%%%%%%%%%%%%%%%%%%%
%%%%%%%%%%%%%%%%%%%%%%%%%
\section{Non zero-mean weights}\label{sect:nonzero}
It is more relevant to neuroscience to consider the case where the weights are distributed as \eqref{eq:synweights} with $J \neq 0$. This complicates quite a bit the limit dynamics as we show now.

If $Q^N$ is the average of the law $P^N(\mathbf{J})$ given by \red{\eqref{eq:girsanov1}} w.r.t. the weights  and $P$ the law of the uncoupled neuron one finds, using  the independence property of the 
weights that
\begin{multline}\label{eq:girsanov}
\frac{dQ^N}{dP^{\otimes N}} = 
\int \prod_{i=1}^N \exp \left\{ \int_0^T \left( \frac{1}{\lambda}\sum_{j=1}^N \jij(\theta) f(Y^j_t) \right)\,dB^i_t \right.\\
- 
\left.\frac{1}{2}  \int_0^T \left( \frac{1}{\lambda}\sum_{j=1}^N \jij(\theta) f(Y^j_t) \right)^2\,dt  \right\}\,d\gamma(\theta),
\end{multline}
where the $Y^j$s in the above formula are the solutions of the uncoupled system
\begin{equation}\label{eq:uncoupledsystem}
dY^j_t = g(Y^j_t) dt + \lambda dB^j_t \quad \quad j=1,\cdots,N
\end{equation}
and are therefore not functions of the $\jij$s.

If one defines
\[
G^i_t(\theta) =\sum_{j=1}^N \jij(\theta) f(Y^j_t),
\]
one observes that, conditionally on \(\mathcal{F}^{B^1, \cdots, B^N}_t\) and \(Y^1_0, \cdots, Y^N_0\), the $G^i$s, $i=1,\cdots, N$ are Gaussian processes whose mean and covariance are easily calculated using equations similar to \eqref{eq:intwrtemp}:
\[
\mathcal{E} [G_t^i] = \frac{J}{N } \sum_{j=1}^N f(Y^j_t) = J \int f(Y_t) \, d\hat{\mu}_N(Y),
\]
and
\[
\mathcal{E} [G_t^i G_s^i ] = \sigma^2 \int f(Y_t) f(Y_s) \, d\hat{\mu}_N(Y) + J^2 \int f(Y_t) \, d\hat{\mu}_N(Y) \times  \int f(Y_s) \, d\hat{\mu}_N(Y),
\]
where $\hat{\mu}_N(Y)$ is the empirical measure
\[
\hat{\mu}_N(Y) = \frac{1}{N} \sum_{j=1}^N \delta_{Y^j}.
\]
Note that because the weights $\jij$ are independent, the processes $(G^i_{t})$ and $(G^j_{t})$, $i \neq j$, are conditionally independent given
 \(\mathcal{F}_t^{B^1,\cdots,B^N}\) and  \(Y^1_0, \cdots, Y^N_0\).

Under $\gamma$, the law of each $G^i_t$ is $\gamma_{\hat{\mu}_N}$, meaning that $G^i_t$, conditionally  on 
\(\mathcal{F}_t^{B^1,\cdots,B^N}\) and  \(Y^1_0, \cdots, Y^N_0\), is a Gaussian process with mean $ m_{\hat{\mu}_N}(t) = J\int f(Y_t)\, d\hat{\mu}_N(Y) = \frac{J}{N} \sum_j f(Y^j_t)$, and covariance 
\[
K_{\hat{\mu}_N}(t,s) = 
\sigma^2 \int f(Y_t) f(Y_s) \, d\hat{\mu}_N(Y) = 
\frac{\sigma^2}{N} \sum_{j=1}^N f(Y^j_t)f(Y^j_s).
\]
Having noticed this, it is possible to generalise the approach in \cite{ben-arous-guionnet:95,guionnet:97} and obtain that the limit law $Q$ is characterised  by the following modification of \eqref{eq:limitRN}
\begin{equation}\label{eq:RN}
\frac{dQ}{dP} = \int \exp \left\{ \frac{1}{\lambda}\int_0^T G_t^Q dB_t -  \frac{1}{2\lambda^2}  \int_0^T (G_t^Q)^2 dt  \right\} \,d\gamma_Q,
\end{equation}
where \(B\) is a \(P\) Brownian motion and, conditionally on \(\mathcal{F}^B_t\)  and $\mu_0$, 
	$G_t^Q$ is under $\gamma_Q$ a Gaussian process with mean $m_Q(t) := J \mathbb{E}^Q[f(Y_t)]$ and covariance $K_Q(t,s) := \sigma^2 \mathbb{E}^Q[f(Y_t) f(Y_s)]$. 

The main result of this section is the mean field dynamics:
	\begin{theorem}\label{theo:dynmean}
		The law $Q$ is that of the solutions of the following system
		\begin{equation}\label{eq:dynmean}
			\begin{cases}
				X_t & = X_0 + \displaystyle\int_0^t g(X_s) \, ds +  J \displaystyle\int_0^t \mathbb{E}^Q[f(X_s)]\,ds +  C_t\\
				C_t & =  \lambda W_t + \displaystyle \int_0^t \int_0^s \dfrac{1}{\lambda^2}\tilde{K}_Q^{s,\lambda}(s,u) \,dC_u\,ds \\
				K^{t,\lambda}_Q(t,s) &= \sigma^2 \mathbb{E}^Q\big[ f(X_t)f(X_s)\big]\\
				\text{{\rm Law of }} X & =  Q, \quad Q_{|\mathcal{F}_0} = \mu_0
			\end{cases}
		\end{equation}
		where 
		$W$ is a $Q$-Brownian motion and $\tilde{K}_Q^{t,\lambda}$ is given by \eqref{eq:KQtildeopb}.
	\end{theorem}

\begin{remark}\label{rem:Pbrownian}
The proof of Theorem~\ref{theo:dynmean} shows that the $P$-Brownian motion $B$ is obtained by
\begin{equation}\label{eq:Pbrownian}
B_t =  \dfrac{1}{\lambda}\left[\int_0^t m_Q(s)\,ds + C_t\right]
\end{equation}
\end{remark}

\begin{remark}\label{rem:lambda_changement_variable}
It is sufficient to prove Theorem~\ref{theo:dynmean} in the particular case where \(\lambda=1\). 
Indeed, for any \(\lambda\), we set \(\bar{X}^i_t = \frac{1}{\lambda}X^i_t\)  and the particle system \(\bar{X}^i\)
evolves with the dynamics \eqref{eq:network} with \(\lambda = 1\), \(\bar{f}(x) = \frac{1}{\lambda}f(\lambda x)\)
and \(\bar{g}(x) = \frac{1}{\lambda}g(\lambda x)\).
Thus, in the following, we state the intermediate results for a general \(\lambda\)
but for the sake of clarity we write the details only for \(\lambda = 1\). We also simplify the
notation \(K^{t,\lambda}\) in \(K^{t}\).
\end{remark}

Before we can give the proof we need to take a short detour and discuss the interpretation of covariance functions such as $K_Q$ as linear operators acting on functions in $L^2([0,t]),\,0\leq t \leq T$ and the solutions to Volterra equations of the second kind.
%%%%%%%%%%
%%%%%%%%%%
\subsection{Covariance functions and linear operators}\label{sect:hstc}
We consider the set of correlation functions $K$ (symmetric positive kernels on $[0,T] \times [0,T]$), such that
\begin{equation}\label{eq:tcnorm}
\int_0^T K(s,s) \,ds < \infty.
\end{equation}
Note that this is true for the correlation functions arising from any probability measure on $\mathcal{C}$ through the fourth equation in \eqref{eq:limit} because for the S-model where $f(y) = y$ the $X_t$s are bounded by $\pm A$ and for the H-model $f$ is bounded, see \eqref{eq:sigbound}.

Given such a correlation function $K$ we associate to it the linear operator $\bar{K}^t$ on $L^2([0,t])$ by
\begin{equation}\label{eq:barK}
(\bar{K}^t \cdot \psi)(s) = \int_0^t K(s,u) \psi(u) \,du\quad 0 \leq s \leq t,\,\forall \psi \in L^2([0,t]).
\end{equation}
In this context $K$ is called the kernel of the operator $\bar{K}^t$.

Such an operator is called a Trace-Class operator (because of \eqref{eq:tcnorm}) and the corresponding norm is noted
\begin{equation}\label{eq:tcnormdef}
\Vert \bar{K}^t \Vert_{t,1} = \int_0^t |K(s,s)| \,ds,\  0 < t \leq T.
\end{equation}
It is known, e.g. \cite{reed_functional_1981}, that the set of Trace-Class operators on $L^2([0,t])$ is a subset of the set of Hilbert-Schmidt operators equipped with the norm
\begin{equation}\label{eq:hsnormdef}
\Vert \bar{K}^t \Vert_{t,2} = \left( \int_0^t \int_0^t K(s,u)^2 \,ds\,du\right)^{1/2},\  0 < t \leq T
\end{equation}
and that the norms satisfy
\begin{equation}\label{eq:tchs}
\Vert \bar{K}^t \Vert_{t,2}  \leq \Vert \bar{K}^t \Vert_{t,1} ,\  0 < t \leq T.
\end{equation}
Given two such operators $\bar{K}_1^t$ and $\bar{K}_2^t$ their composition $\bar{K}_{12}^t=\bar{K}_1^t \circ \bar{K}_2^t$ is defined by
\begin{equation}\label{eq:iterate1}
K_{12}^t(s,u) = \int_0^t K_1^t(s,v)K_2^t(v,u)\,dv\quad \forall \, s,\,u\, \in [0,t]
\end{equation}

	Given such a correlation operator $\bar{K}^t$ we can define on $L^2([0,t])$ another operator, noted $\bar{\tilde{K}}^{t,\lambda}$ by 
	\begin{equation}\label{eq:KQtildeopb}
		\bar{\tilde{K}}^{t,\lambda} = \bar{K}^t \circ \left({\rm Id}+ \dfrac{1}{\lambda^2}\bar{K}^t \right)^{-1} = \left({\rm Id}+ \dfrac{1}{\lambda^2}\bar{K}^t\right)^{-1} \circ \bar{K}^t,
	\end{equation}
$\lambda$ being a non zero real.

\begin{remark}\label{rem:Ktildetspecial}
The reader may be puzzled by the presence of the upper index $t$ in $\bar{K}^t$ and $\bar{\tilde{K}}^{t, \lambda}$, This is meant to indicate that these operators act on functions of $L^2([0,t])$. This would be unnecessary if the operators acting on functions of $L^2([0,s])$, $0 < s < t$ were the restrictions of those acting on functions of $L^2([0,t])$. This is true for $\bar{K}$ but it is \emph{not} true in general for $\bar{\tilde{K}}$.
\end{remark}

	\begin{remark}\label{rem:contraction}
		Note that because $K$ is a correlation function, $\bar{K}^t$ has positive eigenvalues, hence ${\rm Id}+\frac{1}{\lambda^2}\bar{K}^t$ is invertible, and its inverse $({\rm Id}+\frac{1}{\lambda^2}\bar{K}^t)^{-1}$ is contracting.
	\end{remark}

	The relation between $\bar{K}^t$ and $\bar{\tilde{K}}^{t,\lambda}$ is given by the following Lemma (see Lemma A1 in \cite{ben-arous-guionnet:95}):
	\begin{lemma}\label{lem:KKtilde}
		Let $(G_t)_{0 \leq t \leq T}$ be a centered Gaussian process with correlation function $K$ under the law $\nu$. Under the law $\tilde{\nu}^t$ defined
		by
		\[
		d\tilde{\nu}^t(\theta) = \frac{\exp\{-\frac{1}{2\lambda^2}\int_0^t G_s^2(\theta) ds\}}{\int \exp\{-\frac{1}{2\lambda^2}\int_0^t G_s^2(\theta) ds\}\,d\nu}\,d\nu(\theta)
		\]
		$G_t$ is a centered Gaussian process with correlation function $\tilde{K}^{t,\lambda}$ defined by \eqref{eq:KQtildeopb}. In other words,
		\(\forall \ 0 \leq s,\,u  \leq t \leq T\)
		\[
		\int_\Theta G_s(\theta) G_u(\theta) \, d\nu(\theta) = K(s,u) \text{   and   } \int_\Theta G_s(\theta) G_u(\theta) d\tilde{\nu}^t(\theta) = \tilde{K}^{t,\lambda}(s,u) .
		\]
	\end{lemma}

We collect some useful results about the norms of the operators $\bar{K}_Q^t$ 
and $\bar{\tilde{K}}_Q^{t,\lambda}$ in the following lemma.
\begin{lemma}\label{eq:KKtildenorms}
	We have
		\begin{equation}\label{eq:normineq}
			\| \bar{K}_Q^t \|_{t,1},\, \Vert \bar{\tilde{K}}_Q^{t, \lambda} \Vert_{t,1}\leq \left\{
			\begin{array}{ll}
				\sigma^2 A^2 T & \text{{\rm S-model}}\\
				\sigma^2 C^2 T & \text{{\rm H-model}}
			\end{array} \ 0 \leq t \leq T
			\right.,
		\end{equation}
	where $A$ is the parameter in the S-model and $C$ is the parameter in the H-model defined by \eqref{eq:sigbound}.
\end{lemma}
	\begin{proof}
		According to Remark~\ref{rem:lambda_changement_variable}, we assume \(\lambda = 1\).
		The first inequality in \eqref{eq:normineq} is obtained from the definition of the Trace-Class norm:
		\[
		\| \bar{K}_Q^t \|_{t,1} = \sigma^2 \int_0^t \mathbb{E}^Q[ f(X_s)^2 ]\,ds
		\]
		and the facts that $f = {\rm Id}$ and $X^2_t \leq A^2$ for the S-model, and $|f|$ is bounded by $C$ for the H-model.
		
		The second inequality in \eqref{eq:normineq} is a direct consequence of Lemma A2 in \cite{ben-arous-guionnet:95}.
\end{proof}

%%%%%%%%%%%%%%%
%%%%%%%%%%%%%%%
\subsection{A Volterra equation}\label{subsec:operators}

On the road to relating the processes $B_t$ and $W_t$ we have to spend some time on characterising the solutions to a Volterra equation of the second kind.

Our main result is the following Proposition.
\begin{proposition}\label{prop:BtWt}
	Consider  a function \(c\in L^2([0,T])\), a kernel \((\kappa(t,s))_{0\leq s , t \leq T}\) such that \(\kappa(t,s) = 0\) if
	\(t < s\) and 
	\begin{equation}\label{eq:hypkappa}
\forall t\in [0,T],\quad 	\int_0^t \kappa^2(t,s)\,ds \left(=	\int_0^T \kappa^2(t,s)\,ds\right)  \leq M,
\end{equation}
 and a stochastic process \(W\).

	The equation 
	\begin{equation}\label{eq:voltc}
		B_t =  W_t + \int_0^t c(s)\,ds + \int_0^t \left( \int_0^s \kappa(s,u)\, dB_u \right) \,ds\ 0 \leq t \leq T,
	\end{equation}
	 has a unique solution given by
	\begin{multline}\label{eq:Btfinalb}
		B_t =  W_t+ \int_0^t \widetilde{W}_s\,ds +  \int_0^t \left( \int_0^s H(s,u) \widetilde{W}_u\,du \right)ds \\+ \int_0^t c(s)\,ds + \int_0^t \left( \int_0^s H(s,u) c(u)\,du\right)ds
	\end{multline}
	where
	\begin{equation}\label{eq:tildeWb}
		\widetilde{W}_t = \int_0^t \kappa(t,s) \,dW_s.
	\end{equation}
	The function $H$ is defined from $\kappa$ by
	\begin{equation}\label{eq:H}
		\bar{H} = ({\rm Id} - \bar{\kappa})^{-1}-{\rm Id}.
	\end{equation}
	Moreover we have
	\begin{equation}\label{eq:boundnormHb}
		| H(s,u) | \leq M h(M T)\quad 0 \leq s,\,u \leq t,
	\end{equation}
	where  $h(x),\,x \geq 0$, is the sum of the absolutely convergent series
	\[
	h(x) = \sum_{n \geq 0} \frac{x^n}{\sqrt{n!}}
	\]
\end{proposition}
\begin{proof} 
	
	\textbf{Necessary condition:} \\
		Consider  a solution \((B_t)\) to \eqref{eq:voltc}, one denotes $\tilde{B}_t =  \int_0^t \kappa(t,s)\,dB_s$.
		Thus, \eqref{eq:voltc} can be written
	\begin{equation}\label{eq:Btb}
		B_t =  W_t + \int_0^t c(s)\,ds + \int_0^t \tilde{B}_s \,ds.
	\end{equation}
	Differentiating both sides of \eqref{eq:Btb}, multiplying by $\kappa(t,s)$, and integrating between 0 and $t$ yields
	\begin{align}
		\nonumber \tilde{B}_t &=  \int_0^t \kappa(t,s) \,dW_s + \int_0^t \kappa(t,s) c(s)\,ds + \int_0^t \kappa(t,s)\tilde{B}_s\,ds \\
		&=:  \widetilde{W}_t +  \tilde{c}(t) + \int_0^t \kappa(t,s)\tilde{B}_s\,ds,
		\label{eq:Ctb}
	\end{align}
	Now \eqref{eq:Ctb} is a classical linear Volterra equation of the second kind in the unknown $\tilde{B}_t$, except for the fact that $\widetilde{W}_t $ is a random 
	process instead of a deterministic function in $L^2([0,T])$. A close look at the proof in this case (which is found, e.g. in \cite[pp.10-15]{tricomi:57}) shows that it 
	carries over to our case and that the unique solution to \eqref{eq:Ctb}  can be expressed as
	\begin{equation}\label{eq:Ctsolb}
		\tilde{B}_t =  \widetilde{W}_t + \tilde{c}(t) +  \int_0^t H(t,s) \, ( \widetilde{W}_s\ + \tilde{c}(s))ds,
	\end{equation}
	where the resolvent kernel $H$ is the sum of the absolutely convergent series
	\begin{equation}\label{eq:HQt}
		H(t,s) = \sum_{n \geq 0} \left( \kappa(t,s)\right)^{(n+1)}.
	\end{equation}
	The function $\kappa(t,s)^{(n)}$ is the kernel of the operator $\bar{\kappa}$ iterated $n$ times. 
	We notice  that
	\[
	\bar{H} = ({\rm Id}- \bar{\kappa})^{-1}-{\rm Id}
	\]
	\begin{remark}
	Note that under the assumption on \(\kappa\) in Proposition~\ref{prop:BtWt}, one can simplify the composition 
		defined by \eqref{eq:iterate1} into
		\[
		\kappa^{(2)}(s,u) = \int_u^s \kappa(s,v) \kappa(v,u)\,dv\quad \forall \, u \leq s \in [0,t].
		\]
		In other words, the composition does not depend on \(t\) for such kernels.
	\end{remark}

	Combining \eqref{eq:Btb} and \eqref{eq:Ctsolb} we have obtained the necessary condition
	\begin{multline*}
		B_t =  W_t+  \int_0^t \widetilde{W}_s\,ds + \int_0^t c(s)\,ds\\+  \int_0^t \tilde{c}(s)\,ds 
		+  \int_0^t \left( \int_0^s H(s,u) ( \widetilde{W}_u+\tilde{c}(u))\,du\right)\,ds.
	\end{multline*}
	Now it is known from the theory of Volterra equations that the resolvent kernel $H_Q$ and the kernel $\kappa$ satisfy the relation
	\begin{align*}
		H(s,v) & = \kappa(s,v) + \int_v^s \kappa(s,u) H(u,v)\,du \\ 
		& =  \kappa(s,v) + \int_v^s  H(s,u)\kappa(u,v)\,du,
	\end{align*}
	or in operator form
	\begin{equation}\label{eq:int3b}
		\bar{H} =   \bar{\kappa} +   \bar{\kappa} \circ \bar{H}=   \bar{\kappa} + \bar{H} \circ   \bar{\kappa}.
	\end{equation}
	This implies that, according to \eqref{eq:Ctb},
	\[
	\int_0^t \tilde{c}(s)\,ds +  \int_0^t \left( \int_0^s H(s,u) \tilde{c}(u)\,du \right)\,ds = \int_0^t \left( \int_0^s H(s,u) c(u)\,du\right)\,ds
	\]
	and therefore
	\begin{multline*}
		B_t =   W_t+   \int_0^t \widetilde{W}_s\,ds +    \int_0^t \left( \int_0^s H(s,u) \widetilde{W}_u\,du \right)\,ds \\+ \int_0^t c(s)\,ds + \int_0^t \left( \int_0^s 
		H(s,u) c(u)\,du\right)\,ds.
	\end{multline*}
	The last identity and  \eqref{eq:tildeWb} prove that the process $B_t$ is adapted to \(\mathbb{F}^W\).

	\noindent \textbf{Sufficient condition:} \\
	We show that, conversely, if $B_t$ is given by \eqref{eq:Btfinalb} it satisfies \eqref{eq:voltc}. Indeed, \eqref{eq:Btfinalb} commands that
	\[
	dB_u =   dW_u +    (\widetilde{W}_u + c(u)) du+\left( \int_0^u H(u,v) (  \widetilde{W}_v  + c(v))\,dv \right)\,du.
	\]
	Multiplying both sides of this equation by $\kappa(s,u)$ and integrating w.r.t. $u$ between 0 and $s$, we obtain, using \eqref{eq:Ctb} and \eqref{eq:tildeWb}
	\begin{multline}\label{eq:int1}
		\int_0^s \kappa(s,u)\,dB_u = 
		\widetilde{W}_s  +  \int_0^s \kappa(s,u)\,( \widetilde{W}_u  + c(u))\,du \\
		+  \int_0^s \kappa(s,u) \left( \int_0^u H(u,v) (\widetilde{W}_v + c(v))\,dv \right)\,du,
	\end{multline}
	or, in operator form,
	\begin{align*}
		\int_0^s \kappa(s,u)\,dB_u & = \widetilde{W}_s  + ( \bar{\kappa} + \bar{\kappa} \circ \bar{H}) \cdot (\widetilde{W} + c)(s) \\
		& =  \widetilde{W}_s  + \bar{H} \cdot (\widetilde{W}+ c)(s),
	\end{align*}
	by \eqref{eq:int3b}.
	We have obtained
	\[
	\int_0^s \kappa(s,u)\,dB_u = \widetilde{W}_s +  \int_0^s H(s,u) \widetilde{W}_u \,du +  \int_0^s H(s,u) c(u) \,du.
	\]
	Integrating both sides of this equation between 0 and $t$ and adding $ W_t + \int_0^t c(s)\,ds$  to both sides  yields
	\begin{multline*}
		W_t + \int_0^t c(s)\,ds + \int_0^t \left( \int_0^s \kappa(s,u)\,dB_u \right)\,ds = 
		W_t +   \int_0^t \widetilde{W}_s\,ds\\
		+  \int_0^t \left( \int_0^s H(s,u) \widetilde{W}_u \,du \right) \,ds + \int_0^t c(s)\,ds +  \int_0^t \left( \int_0^s H(s,u)  c(u)\,du \right) \,ds,
	\end{multline*}
	and since the right hand side of this equation is equal to $B_t$ by  \eqref{eq:Btfinalb}  this indicates that \eqref{eq:voltc} is satisfied, as claimed.
	
	To finish the proof we must show \eqref{eq:boundnormHb}. It is an easy consequence of the proof found in \cite[pp. 10-15]{tricomi:57}, of \eqref{eq:tchs} and assumption \eqref{eq:hypkappa}.
\end{proof}

\subsection{Proof of Theorem~\ref{theo:dynmean}}

We now prove Theorem~\ref{theo:dynmean}. Using Remark~\ref{rem:lambda_changement_variable} we write the proof in the case $\lambda=1$.
\begin{proof}	
	We start from \eqref{eq:RN} which we reproduce here in the case $\lambda=1$. Note that we have dropped for simplicity the upper index $Q$.
	\begin{equation}\label{eq:girsanovlimit1}
		\frac{dQ}{dP} = \int \exp \left\{ \int_0^TG_t\,dB_t - \int_0^T G_t^2 \, dt   \right\}\,d\gamma_Q 
	\end{equation}
	The plan is to rewrite $dQ/dP$ in such a way that we can apply Girsanov's Theorem, Theorem~\ref{theo:girsanov},  and read off the dynamics from the new expression.
	
	We write
	\[
	G_t = m_Q(t) +  R_t,
	\]
	where $R_t$ is, under $\gamma_Q$, a zero mean Gaussian process with correlation function $  K_Q(t,s)$ and (see Lemma~\ref{lem:KKtilde}), under $\tilde{\gamma}_Q^t$, a  zero mean Gaussian process with correlation function $\tilde{K}_Q^t(t,s)$, where $\tilde{\gamma}_Q^t$ is defined by
	\begin{equation}\label{eq:LambdaT}
		d \tilde{\gamma}_Q^t = \frac{\exp \left\{ -  \int_0^t R_s^2\,ds \right\}}{\int \exp \left\{ -  \int_0^t R_s^2\,ds \right\} \,d\gamma_Q}\ d \gamma_Q := \Lambda_t \ d \gamma_Q
	\end{equation}
	We write
	\begin{multline}\label{eq:dQdP1}
		\frac{dQ}{dP} = \exp \left\{ \int_0^Tm_Q(t) \,dB_t -  \int_0^T  m_Q^2(t)\,dt\right\} \times \\
		\int \exp \left\{ \int_0^T R_t\,dB_t - \int_0^T R_t m_Q(t) \,dt   - \frac{1}{2} \int_0^T R_t^2 \,dt   \right\}\,d\gamma_Q.
	\end{multline}
	The first factor in the right hand side does not pose any problem. The second does because of the presence of $R_t$. The following parts of the proof are meant to get rid of $R_t$.
	
	We rewrite the second factor of the right hand side as follows, using \eqref{eq:LambdaT},
	\begin{multline*}
		\int \exp \left\{ \int_0^T R_t\,dB_t -  \int_0^T R_t m_Q(t) \,dt   - \frac{1}{2} \int_0^T R_t^2 \,dt   \right\}\,d\gamma_Q = \\
		\int \exp \left\{ - \frac{1}{2} \int_0^T R_t^2\,dt \right\} \,d\gamma_Q \times 
		\int \exp \left\{  \int_0^T R_t\,dB_t - \int_0^T R_t m_Q(t) \,dt \right\}\,d\tilde{\gamma}_Q^T
	\end{multline*}
	\textbf{Dealing with the term $\int \exp \left\{ - \dfrac{1}{2} \int_0^T R_t^2\,dt \right\} \,d\gamma_Q$:}\,\\
	We define
	\[
	h(t) = \log \int \exp \left\{ - \frac{1}{2} \int_0^t R_s^2\,ds \right\} \,d\gamma_Q.
	\]
	We have
	\[
	h'(t) = -\frac{1}{2} \int R_t^2 \frac{\exp \left\{ - \frac{1}{2} \int_0^t R_s^2\,ds \right\} \,d\gamma_Q}{\int \exp \left\{ - \frac{1}{2} \int_0^t R_s^2\,ds \right\} \,d\gamma_Q},
	\]
	and by \eqref{eq:LambdaT} and the definition of $\tilde{K}_Q^t$
	\[
	h'(t) = -\frac{1}{2} \int R_t^2 \,d\tilde{\gamma}_Q = -\frac{1}{2} \tilde{K}_Q^t(t,t)
	\]
	Since $h(0)=0$ we conclude that
	\[
	h(t) = -\frac{1}{2} \int_0^t \tilde{K}_Q^s(s,s)\,ds
	\]
	and hence that
	\begin{equation}\label{eq:dQdP2}
		\int \exp \left\{ - \frac{1}{2} \int_0^T R_t^2\,dt \right\} \,d\gamma_Q = \exp \left\{  - \frac{1}{2} \int_0^T \tilde{K}_Q^t(t,t)\,dt \right\} 
	\end{equation}
	\textbf{Dealing with the term $\int \exp \left\{  \int_0^T R_t\,dB_t -  \int_0^T R_t m_Q(t) \,dt \right\}\,d\tilde{\gamma}_Q^T$:}\,\\
	Conditionally to $B_t$, $\int_0^T R_t\,dB_t $ is, under $\tilde{\gamma}_Q^T$, a centered Gaussian process with variance $\int_0^T \int_0^T \tilde{K}_Q^T(t,s)\,dB_s\,dB_t$. Similarly $\int_0^T R_t m_Q(t) \,dt$ is, under $\tilde{\gamma}_Q^T$, a centered Gaussian process with variance
	$\int_0^T \int_0^T \tilde{K}_Q^T(t,s) m_Q(t) m_Q(s) \,ds\,dt$. Their sum is also, conditionally to $B_t$, a centered Gaussian process and therefore, by Gaussian calculus
	\begin{multline}\label{eq:dQdP3}
		\int \exp \left\{\int_0^T R_t\,dB_t -  \int_0^T R_t m_Q(t) \,dt \right\}\,d\tilde{\gamma}_Q^T = \\
		\exp \Big\{ \frac{1}{2} \Big(  \int_0^T \int_0^T \tilde{K}_Q^T(t,s)\,dB_s\,dB_t - 2  \mathbb{E}_{\tilde{\gamma}_Q^T}\left[ \int_0^T R_t\,dB_t \times  \int_0^T R_t m_Q(t) \,dt \right]\\
		+\int_0^T \int_0^T \tilde{K}_Q^T(t,s) m_Q(t) m_Q(s) \,ds\,dt
		\Big)   \Big\} = \\
		\exp \Big\{ \frac{1}{2}   \int_0^T \int_0^T \tilde{K}_Q^T(t,s)\,dB_s\,dB_t -   \int_0^T \int_0^T \tilde{K}_Q^T(t,s) m_Q(s) \,ds\,dB_t\\ 
		+\frac{1}{2} \int_0^T \int_0^T \tilde{K}_Q^T(t,s) m_Q(t) m_Q(s) \,ds\,dt
		\Big\}.
	\end{multline}
	Combining \eqref{eq:dQdP1}-\eqref{eq:dQdP3} we end up with the following formula for $dQ/dP$:
	\begin{multline}\label{eq:bout0}
		\frac{dQ}{dP} = \exp \left\{ \int_0^T m_Q(t) \,dB_t - \frac{1}{2} \int_0^T  m_Q(t)^2\,dt\right\} \times \exp \left\{  - \frac{1}{2} \int_0^T \tilde{K}_Q^t(t,t)\,dt \right\}  \times\\
		\exp \Big\{ \frac{1}{2}   \int_0^T \int_0^T \tilde{K}_Q^T(t,s)\,dB_s\,dB_t -    \int_0^T \int_0^T \tilde{K}_Q^T(t,s) m_Q(s) \,ds\,dB_t\\ 
		+\frac{1}{2} \int_0^T \int_0^T \tilde{K}_Q^T(t,s) m_Q(t) m_Q(s) \,ds\,dt
		\Big\}
	\end{multline}
	This formula does not allow us yet to use Girsanov Theorem, Theorem~\ref{theo:girsanov}. Therefore we need to massage it in order to make this possible.
	
	We apply It\^o formula three times, one for each of the last three terms of the product in the right hand side of \eqref{eq:bout0}. We first note that, according to \eqref{eq:LambdaT},
	\begin{equation}\label{eq:derLambda}
		\frac{d\Lambda_t}{dt} = \frac{1}{2} \Lambda_t \left( \tilde{K}_Q^t(t,t) - R_t^2 \right).
	\end{equation}
	Indeed,
	\begin{align*}
		\frac{d\Lambda_t}{dt} &= -\frac{1}{2} \left[ R_t^2 \Lambda_t -  \exp \left\{ - \frac{1}{2} \int_0^t R_s^2\,ds \right\} \frac{\int R_t^2 \exp \left\{ - \frac{1}{2} \int_0^t R_s^2\,ds \right\} \,d\gamma_Q}{\left(\int \exp \left\{ - \frac{1}{2} \int_0^t R_s^2\,ds \right\} \,d\gamma_Q\right)^2}  \right] \\
		&= \frac{1}{2} \Lambda_t \left[ \int R_t^2 \frac{\exp \left\{ - \frac{1}{2} \int_0^t R_s^2\,ds \right\} }{\int \exp \left\{ - \frac{1}{2} \int_0^t R_s^2\,ds \right\} \,d\gamma_Q} \,d\gamma_Q- R_t^2   \right]\\
		&= \frac{1}{2} \Lambda_t \left(\int R_t^2\,d\tilde{\gamma}_Q  - R_t^2  \right)\\
		&= \frac{1}{2} \Lambda_t \left( \tilde{K}_Q^t(t,t) - R_t^2  \right).
	\end{align*}
	Second we note that, by definition,
	\begin{equation}\label{eq:my1}
		\tilde{K}_Q^T(t,s) = \int R_t R_s \Lambda_T \, d\gamma_Q.
	\end{equation}
	\textbf{Dealing with the term $\exp(\frac{1}{2}\int_0^T \int_0^T \tilde{K}_Q^T(t,s)\,dB_s\,dB_t)$:}\,\\
	We prove that 
	\begin{multline} \label{eq:boutun}
		\exp \left\{ \frac{1}{2}   \int_0^T \int_0^T \tilde{K}_Q^T(t,s)\,dB_s\,dB_t \right\} = \exp \left\{ \frac{1}{2} \int_0^T \tilde{K}_Q^t(t,t)\,dt \right\} \times \\
		\exp \left\{ \int_0^T \int_0^t \tilde{K}_Q^t(t,s)\,dB_s\,dB_t\right\} \times \exp \left\{ -\frac{1}{2} \int_0^T \left( \int_0^t \tilde{K}_Q^t(t,s)\,dB_s   \right)^2\,dt \right\}.
	\end{multline}
	Indeed, it follows from \eqref{eq:my1} that
	\begin{align}\nonumber
		\int_0^T \int_0^T \tilde{K}_Q^T(t,s)\,dB_s\,dB_t & = \int \left(\int_0^T R_t  dB_t\right)^2 \Lambda_T \, d\gamma_Q\\
		& = \int \int_0^T d\left(  \left(\int_0^t R_s  dB_s\right)^2 \Lambda_t \right)\,d\gamma_Q\label{eq:my2}
	\end{align}
	Because $\Lambda_t$ is deterministic, Itô product rule indicates that
	\[
	d\left(  \left(\int_0^t R_s  dB_s\right)^2 \Lambda_t \right) = \left(\int_0^t R_s  dB_s\right)^2\, d\Lambda_t + \Lambda_t\,d\left(\int_0^t R_s  dB_s\right)^2
	\]
	There remains to write It\^o's formula for $\left(\int_0^t R_s  dB_s\right)^2$
	\begin{equation}\label{eq:my3}
		\left(\int_0^t R_s dB_s\right)^2 = 2 \int_0^t R_s \left(\int_0^s R_u\,dB_u\right)\,dB_s + \int_0^t R_s^2 \, ds,
	\end{equation}
	to remember Isserli's theorem
	\begin{equation}\label{eq:my4}
		\int R_t^2 R_s R_u \, d\tilde{\gamma}_Q^t = \tilde{K}_Q^t(t,t) \tilde{K}_Q^t (s,u) + 2 \tilde{K}_Q^t(t,s) \tilde{K}_Q^t (t,u) \quad s,\,u\,\leq t
	\end{equation}
	and to combine \eqref{eq:my1}, \eqref{eq:derLambda}, \eqref{eq:my2}, \eqref{eq:my3} and \eqref{eq:my4} to obtain
	\begin{multline*}
		\int_0^T \int_0^T \tilde{K}_Q^T(t,s)\,dB_s\,dB_t = 2  \int_0^T \int_0^t \tilde{K}_Q^t(t,s)\,dB_s\,dB_t + \int_0^T \tilde{K}_Q^t(t,t)\,dt \\
		-\int_0^T \left( \int_0^t \tilde{K}_Q^t(t,s)\,dB_s   \right)^2\,dt.
	\end{multline*}
	\textbf{Dealing with the term $\exp \left\{ \frac{1}{2}   \int_0^T \int_0^T \tilde{K}_Q^T(t,s) m_Q(s) m_Q(t)\,ds\,dt \right\}$:}\,\\
	Similar arguments show that
	\begin{multline} \label{eq:boutdeux}
		\exp \left\{ \frac{1}{2}   \int_0^T \int_0^T \tilde{K}_Q^T(t,s) m_Q(s) m_Q(t)\,ds\,dt \right\} = \\
		\exp \left\{  \int_0^T m_Q(t) \left(\int_0^t \tilde{K}_Q^t(t,s) m_Q(s)\,ds \right)\,dt\right\} \\
		\times \exp \left\{ -\frac{1}{2} \int_0^T \left( \int_0^t \tilde{K}_Q^t(t,s) m_Q(s)\,ds   \right)^2\,dt \right\}.
	\end{multline}
	Indeed, by \eqref{eq:my1}, we have
	\begin{align*}
		\int_0^T \int_0^T \tilde{K}_Q^T(t,s) m_Q(s) m_Q(t)\,ds\,dt &= \int \left( \int_0^T \int_0^T R_t R_s m_Q(t) m_Q(s)\,ds\,dt \right) \Lambda_T d\gamma_Q \\
		&= \int \left( \int_0^T R_t m_Q(t) \,dt \right)^2 \Lambda_T d\gamma_Q
	\end{align*}
	We apply It\^o's formula to $\left( \int_0^T R_t m_Q(t) \,dt \right)^2 \Lambda_T$. By \eqref{eq:derLambda} we have:
	\begin{multline*}
		\left( \int_0^T R_t m_Q(t) \,dt \right)^2 \Lambda_T = 2 \int_0^T R_t m_Q(t) \left( \int_0^t R_s m_Q(s)\,ds \right) \Lambda_t\,dt  \\
		+\frac{1}{2} \int_0^T \left( \int_0^t R_s m_Q(s) \,ds \right)^2 \Lambda_t \left( \tilde{K}_Q^t(t,t)-R_t^2   \right)\,dt.
	\end{multline*}
	Integrating both sides of the previous equation w.r.t. to $\gamma_Q$ and making use of \eqref{eq:my1} yields
	\begin{multline*}
		\int \left( \int_0^T R_t m_Q(t) \,dt \right)^2 \Lambda_T \,d\gamma_Q = 2\int_0^T m_Q(t) \left( \int_0^t \tilde{K}_Q^t(t,s) m_Q(s)\,ds \right)\,dt\\
		+  \frac{1}{2} \int_0^T \tilde{K}_Q^t(t,t) \left( \int_0^t \int_0^t \tilde{K}_Q^t(s,u) m_Q(s) m_Q(u)\,ds\,du   \right)\,dt\\
		- \frac{1}{2} \int_0^T \left(  R_t^2  \left( \int_0^t R_s m_Q(s) \,ds \right)^2 \Lambda_t   \right)\,d\gamma_Q dt.
	\end{multline*}
	Using  \eqref{eq:my4} again we obtain
	\begin{multline*}   
		\int \left( \int_0^T R_t m_Q(t) \,dt \right)^2 \Lambda_T \,d\gamma_Q = 2 \int_0^T m_Q(t) \left( \int_0^t \tilde{K}_Q^t(t,s) m_Q(s)\,ds \right)\,dt \\
		-\frac{1}{2}\int_0^T \left( \int_0^t \tilde{K}_Q^t(t,s) m_Q(s)\,ds   \right)^2\,dt ,
	\end{multline*}
	i.e. \eqref{eq:boutdeux}.
	
	\noindent
	\textbf{Dealing with the term $\exp \left\{ -  \int_0^T \int_0^T \tilde{K}_Q^T(t,s) m_Q(s) \,ds\,dB_t  \right\}$:}\,\\
	Similar arguments show that
	\begin{multline} \label{eq:bouttrois}
		\exp \left\{ -   \int_0^T \int_0^T \tilde{K}_Q^T(t,s) m_Q(s) \,ds\,dB_t  \right\} \\
		= \exp \left\{ -   \int_0^T \left( \int_0^t \tilde{K}_Q^t(t,s) m_Q(s)\,ds   \right)  \,dB_t\right\} \times \\
		\exp \left\{ -   \int_0^T m_Q(t)\left( \int_0^t  \tilde{K}_Q^t(t,s)\,dB_s   \right)  \,dt \right\} \times \\
		\exp \left\{   \int_0^T \left(  \int_0^t  \tilde{K}_Q^t(t,s)\,dB_s \right) \times \left(  \int_0^t \tilde{K}_Q^t(t,s) m_Q(s)\,ds \right)\,dt \right\} 
	\end{multline}
	Indeed, by \eqref{eq:my1}
	\begin{align*}
		\int_0^T \int_0^T \tilde{K}_Q^T(t,s) m_Q(s) \,ds\,dB_t &= \int \left( \int_0^T \int_0^T R_t R_s m_Q(s) \,ds\,dB_t    \right) \Lambda_T\,d\gamma_Q\\
		&= \int \left( \int_0^T R_t\,dB_t  \right) \left( \int_0^T R_s m_Q(s)\,ds  \right) \Lambda_T\,d\gamma_Q
	\end{align*}
	We apply It\^o's formula to the product $\left( \int_0^T R_t\,dB_t  \right) \left( \int_0^T R_s m_Q(s)\,ds  \right) \Lambda_T$ using \eqref{eq:derLambda}:
	\begin{multline*}
		\left( \int_0^T R_t\,dB_t  \right) \left( \int_0^T R_s m_Q(s)\,ds  \right) \Lambda_T = \int_0^T \left(\int_0^t R_s m_Q(s)\,ds   \right)R_t \Lambda_t \,dB_t\\
		+ \int_0^T \left( \int_0^t R_s \,dB_s   \right) R_t m_Q(t) \Lambda_t \,dt\\
		+ \frac{1}{2} \int_0^T \left( \int_0^t R_s \,dB_s   \right) \left(\int_0^t R_u m_Q(u)\,du   \right) \left( \tilde{K}_Q^t(t,t) - R_t^2  \right)\Lambda_t\,dt
	\end{multline*}
	We integrate both sides of the previous equation w.r.t. $\gamma_Q$ and use \eqref{eq:my1} and  \eqref{eq:my4} to obtain
	\begin{multline*}
		\int \left( \int_0^T R_t\,dB_t  \right) \left( \int_0^T R_s m_Q(s)ds  \right) \Lambda_Td\gamma_Q = \int_0^T \left(\int_0^t \tilde{K}_Q^t(t,s) m_Q(s)\,ds \right)\,dB_t\\
		+ \int_0^T m_Q(t) \left( \int_0^t \tilde{K}_Q(t,s)\,dB_s   \right)\,dt\\
		- \int_0^T \left( \int_0^t \tilde{K}_Q(t,s)\,dB_s   \right) \left(  \int_0^t \tilde{K}_Q^t(t,s) m_Q(s)\,ds  \right)dt
	\end{multline*}
	
	Collecting \eqref{eq:bout0}, \eqref{eq:boutun}, \eqref{eq:boutdeux} and \eqref{eq:bouttrois}, and defining
	\[
	\tilde{B}_t = \int_0^t  \tilde{K}_Q^{t}(t,s)\,dB_s \quad \quad \tilde{m}_Q(t) = \int_0^t  \tilde{K}_Q^{t}(t,s) m_Q(s)\,ds,
	\]
	we obtain
	\begin{equation}\label{eq:fin}
		\frac{dQ}{dP} = \exp \left\{ \int_0^T (m_Q(t) -  \tilde{m}_Q(t) + \tilde{B}_t)\,dB_t 
		- \frac{1}{2} \int_0^T (m_Q(t) - \tilde{m}_Q(t) + \tilde{B}_t )^2\,dt   \right\}
	\end{equation}
	We are almost done.
	We next define
	\begin{equation}\label{eq:cQ}
		c_Q(t) := m_Q(t) -  \tilde{m}_Q(t) = \left( {\rm Id} - \bar{\tilde{K}}_Q^{t} \right) \cdot m_Q(t)
	\end{equation}
	By Girsanov's Theorem, Theorem~\ref{theo:girsanov}, we obtain that
	\[
	W_t = B_t -\int_0^t (c_Q(s) + \tilde{B}_s) ds
	\]
	is a Brownian motion under $Q$.
	Therefore we have
	\[
	 B_t =   W_t + \int_0^t c_Q(s)\,ds +\int_0^t \tilde{B}_s \,ds.
	\]
	Using the definition of \(\tilde{B}\), we obtain the Volterra equation
		\[
		 B_t =   W_t + \int_0^t c_Q(s)\,ds + \int_0^t \int_0^s\tilde{K}^{s}(s,u)d B_u \,ds.
		\]
	and we apply Proposition~\ref{prop:BtWt} with \(\kappa(t,s) = \tilde{K}^{s}(s,u) \mathbbm{1}_{u\leq s}\) to solve it.
	\begin{multline}\label{eq:BtH}
		B_t = W_t + \int_0^t \widetilde{W}_s \, ds+
		 \int_0^t \left( \int_0^s H_Q
		(s,u) \widetilde{W}_u\,du  \right)\,ds\\
		+  \int_0^t c_Q(s)\,ds + \int_0^t \left( \int_0^s H_Q
		(s,u) c_Q(u)\,du  \right)\,ds .
	\end{multline}
	We now use \eqref{eq:cQ} and \eqref{eq:int3b} and obtain
	\begin{multline*}
		 B_t = W_t +\int_0^t \widetilde{W}_s \, ds
		+ \int_0^t \left( \int_0^s H_Q
		(s,u) \widetilde{W}_u\,du  \right)\,ds\\
		+  \int_0^t m_Q(s)\,ds + \int_0^t \left( \int_0^s H_Q
		(s,u) m_Q(u)\,du  \right)\,ds   \\
		-  \int_0^t \int_0^s\tilde{K}^{s}
		(s,u)m_Q(u)\,du \, ds \\ \nonumber
		-  \int_0^t \left( \int_0^s H_Q
		(s,u) \int_0^u \tilde{K}^{u}
		(u,v)m_Q(v)\,dv \, du  \right)\,ds.
		\end{multline*}
	and hence
	\begin{equation}\label{eq:browniansrelation}
	B_t = W_t + \int_0^t \widetilde{W}_s \, ds
		+  \int_0^t \left( \int_0^s H_Q
		(s,u) \widetilde{W}_u\,du  \right)\,ds
		+  \int_0^t m_Q(s)\,ds .
	\end{equation}
We apply Proposition~\ref{prop:BtWt} with \(c\equiv 0\) and observe that the process 
		$ W_t + \int_0^t \widetilde{W}_s \, ds + 
		\int_0^t \left( \int_0^s H_Q(s,u) \widetilde{W}_u\,du  \right)\,ds$ is the solution $C_t$ to the Volterra equation
		\[
		C_t =   W_t + \int_0^t \int_0^s \tilde{K}_Q^{s}(s,u) \, dC_u \,ds.
		\]
		By Girsanov's Theorem~\ref{theo:girsanov}, and the fact $m_Q(s)= J \mathbb{E}^Q[f(X_s)]$,  this yields the equations \eqref{eq:dynmean} of the dynamics.
\end{proof}

\begin{remark}
Note that as a side result we have proved \eqref{eq:Pbrownian} in Remark~\ref{rem:Pbrownian}.
	Indeed,  for a general \(\lambda\), we recall the Volterra type equation
	\[
					C_t  =  \lambda W_t + \displaystyle \int_0^t \int_0^s \dfrac{1}{\lambda^2}\tilde{K}_Q^{s,\lambda}(s,u) \,dC_u\,ds.
	\]
The process \(\frac{C_t}{\lambda}\) solves a Volterra equation of type \eqref{eq:voltc} with \(c\equiv 0\) and \(\kappa(t,s) = \frac{1}{\lambda^2}\tilde{K}^{t,\lambda}(t,s)\mathbbm{1}_{0\leq s \leq t}\).
\end{remark}

\begin{remark}\label{rem:effective}
This result is a precise description of the somewhat mysterious ``effective interaction term'' used in the papers  \cite{meegen-kuhn-etal:21} and \cite{cabana-touboul:13}\footnote{This term was derived by an incorrect application of Girsanov's Theorem, Theorem~\ref{theo:girsanov}, in \cite{cabana-touboul:13}}. Indeed, the equation \eqref{eq:browniansrelation}
\[
B_t = \frac{1}{\lambda}\left(  \lambda W_t +  \int_0^t \widetilde{W}_s \, ds + 
		\int_0^t \left( \int_0^s H_Q(s,u) \widetilde{W}_u\,du  \right) \,ds + \int_0^t m_Q(s)\,ds\right),
\]
allows us to rewrite the meanfield dynamics \eqref{eq:dynmean} as
\[
\begin{cases}
X_t  &= X_0 + \int_0^t g(X_s) \, ds +  J \int_0^t \mathbb{E}^Q[f(X_s)]\,ds +\lambda  \int_0^t \widetilde{W}_s \, ds  \\
		&  \quad \quad + \lambda \int_0^t \left( \int_0^s H_Q(s,u) \widetilde{W}_u\,du  \right)\,ds + \lambda W_t \\
\widetilde{W}_t &= \frac{1}{\lambda^2} \int_0^t \tilde{K}_Q^{t,\lambda}(t,s)\,dW_s \\
H_Q(t,s) &= \sum_{n \geq 0} \left( \frac{1}{\lambda^2} \tilde{K}_Q^{t,\lambda}(t,s)\mathbbm{1}_{s\leq t}\right)^{(n+1)}\\
K^{t,\lambda}_Q(t,s) &= \sigma^2 \mathbb{E}^Q\big[ f(X_t)f(X_s)\big]\\
				\text{{\rm Law of }} X & =  Q, \quad Q_{|\mathcal{F}_0} = \mu_0
\end{cases}
\]
where $W$ is a $Q$-Brownian motion and $\tilde{K}_Q^{t,\lambda}$ is given by \eqref{eq:KQtildeopb} as a function of $K^{t,\lambda}_Q$. The dependency of the effective interaction term 
$J \mathbb{E}^Q[f(X_t)] + \lambda\widetilde{W}_t  +  \lambda \int_0^t H_Q(t,s) \widetilde{W}_s\,ds $ upon the Brownian $W$ and the law $Q$ is clear. It is a Gaussian process with mean $J  \mathbb{E}^Q[f(X_t)]$. Its covariance can be computed using standard Itô calculus as a function of $\tilde{K}_Q^{t,\lambda}$.
\end{remark}

%%%%%%%%%%%
%%%%%%%%%%
%%%%%%%%
%%%%%%%%%%%%%
%%%%%%%%%%%%%
%%%%%%%%%%%%%
%\section{The special case \texorpdfstring{$f=1$}{f = 1}}\label{sect:f1}
\section{The special case $f=1$}\label{sect:f1}
We now study the special case where the function $f$ in \eqref{eq:network} is constant and 
equal for example to 1. This means that each neuron $j$ contributes exactly $\jij$ to the 
activity of neuron $i$, independently of the value $X^j_t$ of its membrane potential. It can be seen as a special case of the H-model when the sigmoidal function $f$ is flat. 
It may appear as a toy example but it is useful as it shows the abstract concepts of the previous sections at work. We note 
$Q_1$ the corresponding limit law of the network.
%%%%%%%%%%%%%%%
\subsection{Annealed mean-field dynamics}
The third equation in \eqref{eq:dynmean} reads in this case
\[
K_{Q_1}^t(u,s) = \sigma^2, \quad \forall u,\,s \leq t \leq T
\]
We prove the following Lemma:
\begin{lemma}\label{lem:KQ1tilde}
We have:
\begin{equation}\label{eq:KQ1tilde}
\tilde{K}^{t,\lambda}_{Q_1}(u,s)  =\frac{\sigma^2 \lambda^2}{\lambda^2+\sigma^2 t}, \quad \forall u,\,s \leq t \leq T
\end{equation}
\end{lemma}
\begin{proof}
Given $\varphi \in L^2([0,T])$ we have by \eqref{eq:barK}
\begin{equation}\label{eq:KQ1phi}
\bar{K}_{Q_1}^t \cdot \varphi(s) =  \sigma^2 \int_0^t \varphi(u)\,du
\end{equation}
Next we identify the kernel of 
$({\rm Id}+\frac{1}{\lambda^2}\bar{K}_{Q_1}^t)^{-1}$. 
Given $\varphi \in L^2([0,T])$, denote by $\psi$ its image by 
$({\rm Id}+\frac{1}{\lambda^2}\bar{K}_{Q_1}^t)^{-1}$. By \eqref{eq:KQ1phi} we have
\[
\varphi(s) = ({\rm Id}+\frac{1}{\lambda^2}\bar{K}_{Q_1}^t) \cdot \psi (s) = 
\psi(s) + \dfrac{\sigma^2}{\lambda^2} \int_0^t \psi(u)\,du\ \forall s \leq t.
\]
Integrating both sides between 0 and $t$ yields
\begin{equation}\label{eq:varpsi}
\int_0^t \varphi(s)\,ds = \int_0^t \psi(s)\,ds +  \dfrac{\sigma^2}{\lambda^2} t \int_0^t \psi(u)\,du = \left(1 +  \dfrac{\sigma^2}{\lambda^2}t \right) \int_0^t \psi(s)\,ds.
\end{equation}
We conclude that for all $\varphi \in L^2([0,T])$, by \eqref{eq:KQ1phi} and \eqref{eq:varpsi}
\[
 \left( \bar{K}_{Q_1}^t \circ ({\rm Id}+\frac{1}{\lambda^2}\bar{K}_{Q_1}^t)^{-1} \right) \cdot \varphi(t) = 
 \bar{K}_{Q_1}^t \cdot \psi(t) =  \sigma^2 \int_0^t \psi(s)\,ds = \sigma^2
  \dfrac{1}{1+ \dfrac{\sigma^2}{\lambda^2}
 	 t} \int_0^t \varphi(s)\,ds,
\]
which proves the Lemma.
\end{proof}
\begin{remark}
This is an illustration of the fact, mentioned in Remark~\ref{rem:Ktildetspecial}, that $\bar{\tilde{K}}^{s,\lambda}$ is not the restriction of $\bar{\tilde{K}}^{t,\lambda}$, $s < t \leq T$, to $L^2([0,s])$. Indeed, the kernel $\tilde{K}^{t,\lambda}(u,v)$ is constant and equal to $\frac{\sigma^2 \lambda^2}{\lambda^2+\sigma^2 t}$ on $[0,s] \times [0,s]$. while $\tilde{K}^{s,\lambda}(u,v)$ is also constant but equal to $\frac{\sigma^2 \lambda^2}{\lambda^2+\sigma^2 s}$.
\end{remark}

\begin{lemma}\label{lem:HQ1}
We have:
\begin{equation}\label{eq:HQ1}
H_{Q_1}(t,s) = \dfrac{\sigma^2 }{\lambda^2+\sigma^2 s}\mathbbm{1}_{s \leq t} 
\end{equation}
\end{lemma}
\begin{proof}
By definition
\[
H_{Q_1}(t,s) = \sum_{n \geq 0} \left( \frac{1}{\lambda^2}\tilde{K}_{Q_1}^{t, \lambda}(t,s)\mathbbm{1}_{s \leq t}\right)^{(n+1)}
\]
We obtain successively
\begin{align*}
\left( \frac{1}{\lambda^2}\tilde{K}_{Q_1}^{t,\lambda}
(t,s)\mathbbm{1}_{s \leq t}\right)^{(1)} &= \frac{1}{\lambda^2}\tilde{K}_{Q_1}^{t, \lambda}
(t,s)\mathbbm{1}_{s \leq t} = \frac{\sigma^2 }{\lambda^2+\sigma^2 t}\mathbbm{1}_{s \leq t}\\
\left(\frac{1}{\lambda^2}\tilde{K}_{Q_1}^{t, \lambda}
(t,s)\mathbbm{1}_{s \leq t}\right)^{(2)} &= \frac{1}{\lambda^4} \int_s^t  \tilde{K}_{Q_1}^{t, \lambda}
(t,u)  \tilde{K}_{Q_1}^{u, \lambda}
(u,s)\,du = 
\frac{\sigma^2 }{\lambda^2+\sigma^2 t} \log \frac{\lambda^2+\sigma^2 t}{\lambda^2+\sigma^2 s}\mathbbm{1}_{s \leq t}\\
\left(\frac{1}{\lambda^2}\tilde{K}_{Q_1}^{t, \lambda}
(t,s)\mathbbm{1}_{s \leq t}\right)^{(3)} &= \int_s^t  \left( \frac{1}{\lambda^2}\tilde{K}_{Q_1}^{t, \lambda}
(t,u)\right)^{(1)}
 \left(\frac{1}{\lambda^2}\tilde{K}_{Q_1}^{t, \lambda}
(u,s)\right)^{(2)}du 
= \\
&\ \quad \frac{\sigma^2}{\lambda^2 + \sigma^2 t} \frac{1}{2}\left(\log \frac{\lambda^2+\sigma^2 t}{\lambda^2+\sigma^2 s} \right)^2\mathbbm{1}_{s \leq t}\\
\vdots &\\
\left(\frac{1}{\lambda^2}\tilde{K}_{Q_1}^{t, \lambda}
(t,s)\mathbbm{1}_{s \leq t}\right)^{(n)} &= \frac{\sigma^2}{\lambda^2+\sigma^2 t} \frac{1}{(n-1)!}\left(\log \frac{\lambda^2 + \sigma^2 t}{\lambda^2 + \sigma^2 s} \right)^{(n-1)}\mathbbm{1}_{s \leq t},
\end{align*}
and therefore, by \eqref{eq:HQt}, we obtain the Lemma.
\end{proof}
\begin{remark}
One finally verifies \eqref{eq:int3b} in the case $\kappa(t,s) = \frac{1}{\lambda^2} \tilde{K}_{Q_1}^{t,\lambda}(t,s)\mathbbm{1}_{s \leq t}$:
\begin{align*}
\frac{1}{\lambda^2} \tilde{K}_{Q_1}^{t, \lambda}(t,s)\mathbbm{1}_{s \leq t} - H_{Q_1}(t,s) & =\dfrac{\sigma^2}{\lambda^2+ \sigma^2 t} -  \dfrac{\sigma^2}{\lambda^2+\sigma^2 s}  \\
                                                  & = - \int_s^t \frac{1}{\lambda^2}\tilde{K}_{Q_1}^{t, \lambda}(t,u) H_{Q_1}(u,s)\,du,
\end{align*}
i.e.
\[
\frac{1}{\lambda^2} \bar{\tilde{K}}_{Q_1}^{t, \lambda} + \frac{1}{\lambda^2}\bar{\tilde{K}}_{Q_1}^{t, \lambda} \circ \bar{H}_{Q_1} = \bar{H}_{Q_1}.
\]
\end{remark}
Let us now write the equations \eqref{eq:dynmean} of the dynamics is this case. We clearly have 
\[
m_Q(t) = J,
\]
and by the second equation in \eqref{eq:dynmean}
\[
C_t = \lambda W_t + \sigma^2\int_0^t \frac{ C_s}{\lambda^2 + \sigma^2 s} \,ds
\]
We can then compute $C_t$ as a function of the $Q_1$ Brownian motion $W_t$ by solving the previous linear SDE obtaining
\begin{equation}\label{eq:B0}
C_t =\lambda\left(\lambda^2+\sigma^2 t\right) \int_0^t \frac{dW_s}{\lambda^2 +\sigma^2 s}
\end{equation}
and write the corresponding dynamics as
\begin{equation}\label{eq:dynQ1W}
X_t = X_0 + \int_0^t g(X_s)\,ds + J t+ \lambda\left(\lambda^2+\sigma^2 t\right) \int_0^t \frac{dW_s}{\lambda^2+\sigma^2 s}
\end{equation}
or equivalently in  differential form
\begin{equation}\label{eq:dynQ1WSDE}
dX_t = \left( g(X_t) + J +  \lambda  \sigma^2   \int_0^t \frac{dW_s}{\lambda^2 + \sigma^2 s}   \right)\,dt + \lambda dW_t
\end{equation}
\begin{remark}
The ``effective interaction term'' of Remark~\ref{rem:effective}, $J +  \lambda  \sigma^2   \int_0^t \frac{dW_s}{\lambda^2 + \sigma^2 s} $, 
has mean $J$ and variance $\frac{t \sigma^4}{\lambda^2+\sigma^2 t}$ at time $t$ as can be easily verified. Note that this is different from 
the variance reported in previous work such as \cite{faugeras_constructive_2009}[Equations (25)] or \cite{cabana-touboul:13}[Section 3.3].
 It shows the advantage of the approach developed in this article of working out the exact dependency of the effective interaction term with the Brownian.
\end{remark}

\subsection{Quenched mean-field dynamics}
The first equation in \eqref{eq:network} writes in this case
\begin{equation}\label{eq:xit}
dX^i_t =  \left( -X^i_t+S_i\right)\,dt+\lambda dB^i_t \quad 1 \leq i \leq N,
\end{equation}
where 
\[
S_i = \sum_{j=1}^N \jij.
\]
Since the $\jij$s are i.i.d. as $\mathcal{N}(J/N,\sigma^2/N)$, the $S_i$s are i.i.d. as $\mathcal{N}(J,\sigma^2)$. The propagation of chaos is obvious in this case since the $N$ equations \eqref{eq:xit} are independent and the quenched mean-field dynamics is also obvious.
\begin{align}\label{eq:quencheddynamics}
dX_t & = \left( -X_t+G\right)\,dt+\lambda dB_t, \\
G &\simeq \mathcal{N}(J,\sigma^2),\, \quad \text{independent of  } B \nonumber
\end{align}
%%%%
\subsection{Comparison between the quenched and annealed dynamics}
We know from the results in \cite{ben-arous-guionnet:95,guionnet:97}
 that these two dynamics should be identical. It is unclear though that the solutions to \eqref{eq:dynQ1WSDE} and \eqref{eq:quencheddynamics} have the same laws. In effect they do, as shown in the next Proposition.

\begin{proposition}
The solutions to \eqref{eq:dynQ1WSDE} and \eqref{eq:quencheddynamics} are two Gaussian processes with identical means and covariance functions, assuming identical Gaussian distributed initial conditions.
\end{proposition}
\begin{proof}
Gaussianity is clear. We show that the two processes have the same mean and variance. 
The solution to the quenched mean-field dynamics \eqref{eq:quencheddynamics} is readily computed as
\begin{equation}\label{eq:limitepourfegal1}
X_t = X_0 e^{-t} + G(1- e^{-t}) + e^{-t}\lambda\int_0^te^{s} dW_s
\end{equation}
where \(G\simeq \mathcal{N}(J,\sigma^2)\), \(X_0\) et \(W\) are assumed to be  independent. Hence, taking the expected value 
\begin{equation}\label{eq:meanf1direct}
\mathbb{E}(X_t) = J + (\mathbb{E}(X_0) - J)e^{-t},
\end{equation}
and
\[
\var(X_t) = \var(G)(1-e^{-t})^2 + \var(X_0)e^{-2t} + \lambda^2\dfrac{1-e^{-2t}}{2},
\]
and therefore
\begin{equation}\label{eq:varf1direct}
\var(X_t) = \var(G) + \dfrac{\lambda^2}{2} - 2\var(G)e^{-t} + \left(\var(G) + \var(X_0) - \dfrac{\lambda^2}{2}\right)e^{-2t}
\end{equation}
Assuming $g(x)=-x$ we write \eqref{eq:dynQ1WSDE} as (using the notation $Y_t$  instead of $X_t$ for sake of clarity)
\begin{equation}\label{eq:dynQ1WSDE1}
dY_t = \left( - Y_t + J +  \lambda  \sigma^2   \int_0^t \frac{dW_s}{\lambda^2 +\sigma^2 s}   \right)\,dt + \lambda dW_t
\end{equation}
so that, taking the expected value of both sides
\begin{equation}\label{eq:mfmean}
d \mathbb{E}(Y_t) = (-\mathbb{E}(Y_t)  + J) dt,
\end{equation}
and
\[
d(Y_t-  \mathbb{E}(Y_t)) = \left(-(Y_t-  \mathbb{E}(Y_t))
 +   \lambda  \sigma^2  \int_0^t \frac{dW_s}{\lambda^2+\sigma^2 s}   \right) dt
 + \lambda dW_t.
\]
Equation \eqref{eq:mfmean} is readily integrated yielding
\begin{equation}\label{eq:mfmeanvalue}
 \mathbb{E}(Y_t)= J + (\mathbb{E}(X_0) - J) e^{-t}
\end{equation}
This equation is identical to \eqref{eq:meanf1direct}: the two processes $X$ and $Y$ have the same mean.

An application of It\^o's formula to $(Y_t-  \mathbb{E}(Y_t))^2$ yields
\begin{align*}
d(Y_t -  \mathbb{E}(Y_t))^2 =& 
	  -2 (Y_t -  \mathbb{E}(Y_t))^2 dt + 
	2  \lambda\sigma^2   (Y_t -  \mathbb{E}(Y_t))  \int_0^t \frac{ dW_s}{\lambda^2+\sigma^2 s} dt  \\
	& +2 \lambda (Y_t -  \mathbb{E}(Y_t)) dW_t + \lambda^2 dt.
\end{align*}
Define $v(t) := \mathbb{E}[(Y_t - \mathbb{E}(Y_t))^2]$. Taking the expected value of both sides of the previous equation we obtain
\begin{equation}\label{eq:vt}
v(t) = v(0) - 2 \int_0^t v(s)\,ds +  \lambda^2 t + 2  \lambda\sigma^2 \int_0^t \mathbb{E}\left[ (Y_s - \mathbb{E}(Y_s))  \int_0^s \frac{dW_\theta}{\lambda^2 +  \sigma^2 \theta} \right]ds 
\end{equation}
Define \(Z_t\) by
\begin{equation}\label{eq:Zt}
dZ_t =  \dfrac{dW_t}{\lambda^2+\sigma^2 t}.
\end{equation}
So, we have
\[
d(Y_t - \mathbb{E}(Y_t)) = (- (Y_t - \mathbb{E}(Y_t)) + \lambda\sigma^2 Z_t ) dt + \lambda dW_t.
\]
We apply Itô's product formula to $Z_t(Y_t - \mathbb{E}(Y_t)) $:
\begin{multline*}
d (Z_t (Y_t - \mathbb{E}(Y_t)) ) = \\
(- Z_t (Y_t - \mathbb{E}(Y_t)) + \lambda\sigma^2 Z_t^2 ) dt + \lambda Z_t dW_t + \frac{\lambda}{\lambda^2 + \sigma^2 t} dt + \frac{Y_t - \mathbb{E}(Y_t)}{\lambda^2 + \sigma^2 t} dW_t.
\end{multline*}
According to \eqref{eq:Zt} we have
\begin{align*}
	\mathbb{E}[Z_t^2]  = \dfrac{t}{\lambda^2(\lambda^2 + \sigma^2 t)}.
\end{align*}
Integrate between 0 and $t$, use $Z_0 = 0$, and take the expected value to obtain
\[
\mathbb{E}[Z_t (Y_t - \mathbb{E}(Y_t))] = \int_0^t  - \mathbb{E}[Z_s (Y_s - \mathbb{E}(Y_s))]  + \dfrac{1}{\lambda} ds.
\]
That is 
\[
\mathbb{E}[Z_t (Y_t - \mathbb{E}(Y_t))] = \dfrac{1 - e^{-t}}{\lambda}.
\]
Thus \eqref{eq:vt} writes
\[
v(t) = v(0) -2 \int_0^t v(s) \, ds + \lambda^2 t + 2 \sigma^2 \int_0^t (1 - e^{-s})\,ds.
\]
This yields
\[
v(t) = \sigma^2 + \frac{\lambda^2}{2} -2 \sigma ^2 e^{-t} + (\sigma^2 + v(0) - \frac{\lambda^2}{2}) e^{-2t},
\]
which is identical to \eqref{eq:varf1direct}.
 \end{proof}

%%%%%%%%%%%%%
\section{Uniqueness of Q}\label{sect:uniqueness}
We prove the uniqueness of the measure $Q$ in \eqref{eq:dynmean}. Such a proof has been provided in \cite{ben-arous-guionnet:95} for the S-model in the case $J=0$ and is not constructive. Our proof is valid for the H-model as well as the S-model for any value of $J$. It is based upon a fixed-point approach and is constructive. In effect it is the basis for a numerical algorithm described in Section~\ref{sect:numerics}. 

To simplify notations we assume $\lambda=1$ and rewrite the system \eqref{eq:dynmean} in this case using the results of the previous sections. We will be using several times in our proof Proposition~\ref{prop:BtWt} for different values of the kernel $\kappa$.

According to \eqref{eq:Btfinalb} in the case $c\equiv 0$ and $\kappa = \tilde{K}^t_Q$, \eqref{eq:dynmean} is equivalent to
\begin{multline}\label{eq:newdyn}
X_t  = X_0 + \int_0^t g(X_s) \, ds + J \int_0^t \mathbb{E}^Q[f(X_s)]\,ds + W_t\\
+ \int_0^t \widetilde{W}_s\,ds + \int_0^t \left( \int_0^s H_Q(s,u) \widetilde{W}_u\,du\right)\,ds 
\end{multline}
$W_t$ is $Q$-Brownian motion,  $\widetilde{W}_t$ is given by \eqref{eq:tildeWb} where the correlation function 
\(\tilde{K}_Q^t\) is linked to \(K_Q(t,s)= \sigma^2 \mathbb{E}_Q[f(X_t)f(X_s)]\) by 
\eqref{eq:KQtildeopb}.
Finally the function $H_Q$ is given by \eqref{eq:H}.

We consider the family of Polish spaces, Cartesian products of $L^2([0,t])$ with $\mathcal{K}_t$, the set of Trace-Class operators on $L^2([0,t])$, for $0 < t \leq T$. The space $L^2([0,t]) \times \mathcal{K}_t$ is endowed with the  norm:
\begin{equation}\label{eq:norm}
\|(m,\bar{K})\|_t^2 = \| m \|_{L^2([0,t])}^2 + \| \bar{K} \|_{t,2}^2\quad \forall m \in L^2([0,t]),\, \bar{K} \in \mathcal{K}_t
\end{equation}
Given $\bar{K}$ in $\mathcal{K}_t$ we associate with it the elements $\bar{\tilde{K}}^t$ of $\mathcal{K}_t$
 by \eqref{eq:KQtildeopb} (in the case $\lambda=1$) and $\bar{H}^t$ by \eqref{eq:H} with $\kappa(t,s)=\tilde{K}^t(t,s)\mathbbm{1}_{t\geq s}$. 
For $(m,\bar{K})$ in $L^2([0,t]) \times \mathcal{K}_t$ we consider the process \(X^{m,\bar{K}}\)
solution of the SDE
\begin{multline}\label{eq:limitK}
X^{m,\bar{K}}_t  = X_0 + \int_0^t g(X^{m,\bar{K}}_s) \, ds + J \int_0^t m(s)\,ds + W_t\\
+\int_0^t \widetilde{W}_s\,ds + \int_0^t \left( \int_0^s H(s,u) \widetilde{W}_u\,du\right)\,ds 
\end{multline}
where $\widetilde{W}$ is given by \eqref{eq:tildeWb} with $\kappa(t,s)=\tilde{K}^t(t,s)\mathbbm{1}_{t\geq s}$. 
We  define the mapping $\Lambda$ from $L^2([0,t]) \times \mathcal{K}_t$ as follows
\begin{equation}\label{eq:pG}
\Lambda(m, \bar{K}) = ( \mathbb{E}\left[ f(X_t^{m,\bar{K}}) \right], 
\sigma^2\mathbb{E}\left[ f(X_t^{m,\bar{K}})f(X_s^{m,\bar{K}}) \right])\\
\end{equation}
where $X_t^{m,\bar{K}}$ is the solution to \eqref{eq:limitK}.

\begin{lemma}\label{lem:wellposed}
The mapping $\Lambda$ maps $L^2([0,t]) \times \mathcal{K}_t$ to itself for all $0 < t \leq T$.
\end{lemma}
\begin{proof}
Clear from the definition \eqref{eq:pG} since $f$ (respectively $X^{m,\bar{K}}$) is bounded for the H-model (respectively for the S-Model).
\end{proof} 
We use the metric induced by the norm \eqref{eq:norm}:
\begin{equation}\label{eq:dist}
 d_t((m_1,\bar{K}_1),(m_2,\bar{K}_2))^2 = \Vert m_1 - m_2 \Vert^2_{L^2([0,t])} + \Vert \bar{K}_1 - \bar{K}_2 \Vert^2_{t,2}
 \end{equation}
Observe that if $X$ is a solution to \eqref{eq:newdyn}, then the corresponding pair $(m_Q,\bar{K}_Q)$ is a fixed point of $\Lambda$ and conversely if $(m,\bar{K})$ is such a fixed point of $\Lambda$, \eqref{eq:limitK} defines a solution to \eqref{eq:newdyn} on $[0,T]$. Our problem is now translated into a fixed point problem for which we prove the following contraction lemma
\begin{lemma}\label{lem:contraction}
For $t \leq T$, \(\bar{K}_1, \bar{K}_2 \in \mathcal{K}_T, m_1,\,m_2 \in L^2([0,T])\), we have
\begin{equation}\label{eq:lemma}
d_t(\Lambda((m_1,\bar{K}_1)),\Lambda((m_2,\bar{K}_2)))^2 \leq c_T \int_0^t d_s((m_1,\bar{K}_1),(m_2,\bar{K}_2))^2\,ds.
\end{equation}
where $d_t$ is defined by \eqref{eq:dist} and $c_T$ is a constant depending only on $T$.
\end{lemma}
\begin{proof}
Let $(m_1,\bar{K}_1)$ and $(m_2,\bar{K}_2)$ be two elements of $L^2([0,T]) \times \mathcal{K}_T$, $(p_1,\bar{G}_1)$ and  $(p_2,\bar{G}_2)$ their images by $\Lambda$, see \eqref{eq:pG}. In order to simplify the reading and w.l.o.g. we assume $J=\sigma=1$. Let  $(X^1_t)_{0 \leq t \leq T}$ and $(X^2_t)_{0 \leq t \leq T}$ be the solutions to the corresponding equations \eqref{eq:limitK}. We write
\begin{align}
\label{eq:work}X^1_t-X^2_t &= \int_0^t (g(X^1_s) - g(X^2_s))\,ds  +  \int_0^t (m_1(s) - m_2(s))\,ds\\
&  + \int_0^t \int_0^s (\tilde{K}^s_1(s,u) - \tilde{K}^s_2(s,u))\,dW_u ds\nonumber\\
 &  + \int_0^t \left(\int_0^s H_1(s,u) \int_0^u \tilde{K}_1^u(u,v) \, dW_v\,du - \int_0^s H_2(s,u) \int_0^u \tilde{K}_2^u(u,v) \, dW_v\,du\right)\,ds\nonumber
\end{align}
By \eqref{eq:pG} and Cauchy-Schwarz we write
\[
(p_1(t) - p_2(t))^2 = \left(\mathbb{E}[f(X^1_t) - f(X^2_t) ]\right)^2 \leq \mathbb{E}[(f(X^1_t) - f(X^2_t) )^2]\\
\]
Using the fact that $f$ is Lipschitz continuous we conclude that
\begin{equation}\label{eq:p1minusp2}
(p_1(t) - p_2(t))^2 \leq C \mathbb{E}[(X^1_t - X^2_t )^2]\
\end{equation}
for some positive constant $C$.

By \eqref{eq:pG} and Cauchy-Schwarz again we write
\begin{align*}
\left( G_1(t,s) - G_2(t,s) \right)^2 &= \left(\mathbb{E}[f(X^1_t)f(X^1_s) - f(X^2_t)f(X^2_s) ]\right)^2\\
& \leq \left(\mathbb{E}[|f(X^1_t)|\, | f(X^1_s) - f(X^2_s)| + | f(X^2_s) | \, | f(X^1_t) - f(X^2_t)|]\right)^2\\
& \leq \left(\mathbb{E}\left[ \sqrt{f(X^1_t)^2+f(X^2_s)^2} \sqrt{(f(X^1_s) - f(X^2_s))^2+(f(X^1_t) - f(X^2_t))^2} \right] \right)^2\\
& \leq \mathbb{E}\left[f(X^1_t)^2+f(X^2_s)^2\right]\times \mathbb{E}\left[(f(X^1_s) - f(X^2_s))^2+(f(X^1_t) - f(X^2_t))^2 \right]
\end{align*}
Using the facts that  $f$ is bounded and Lipschitz continuous for the H-model, or is the identity and $|X^1_t|$, $|X^2_s|$ are a.s. bounded by $A$ for the S-model we can write
\begin{equation}\label{eq:Lambdadif}
\left( G_1(t,s) - G_2(t,s) \right)^2 \leq C \mathbb{E}[(X^1_t - X^2_t)^2] + C 
\mathbb{E}[(X^1_s - X^2_s)^2] \\
\end{equation}
for some other positive constant $C$.

We thus have to upperbound $\mathbb{E}[(X^1_t - X^2_t)^2]$.

The  application of Ito's formula to $(X^1_t - X^2_t)^2$ yields
\begin{align*}
\mathbb{E}[(X^1_t - X^2_t)^2] =& 2 \mathbb{E}\left[ \int_0^t (X^1_s - X^2_s)(g(X^1_s) - g(X^2_s))\,ds \right]+ \\
			& 2 \mathbb{E}\left[ \int_0^t (m_1(s) - m_2(s))(X^1_s - X^2_s)\,ds \right] +\\
			& 2 \mathbb{E}\left[ \int_0^t (\widetilde{W}_1(s) - \widetilde{W}_2(s))\,ds \right]+ \\
			& 2 \mathbb{E}\left[ \int_0^t (X^1_s - X^2_s) \left(  \int_0^s (H_1(s, u)\widetilde{W}_1(s,u) -  H_2(s, u)\widetilde{W}_2(s,u))\,du \right)\,ds \right]
\end{align*}
The right hand side of this equation is the sum of the expectations of four terms. The third term being equal to
\[
\int_0^t \int_0^s (\tilde{K}^t_1(s,u) -  \tilde{K}^t_2(s,u)) dW_u \,ds,
\]
its expectation is equal to 0.

We consider next the first term $2 \mathbb{E}\left[ \int_0^t (X^1_s - X^2_s)(g(X^1_s) - g(X^2_s))\,ds \right]$.

\noindent
In the case of the H-Model, $g(x)=-\alpha x$,  $\alpha > 0$, and $g$ is a decreasing function. $g$ is also a decreasing function in the case of the S-model, as shown in Figure~\ref{fig:Ug}. Therefore in both cases $2 \mathbb{E}\left[ \int_0^t (X^1_s - X^2_s)(g(X^1_s) - g(X^2_s))\,ds \right]  \leq 0$.

We have obtained
\begin{align}\label{eq:e1}
\mathbb{E}[(X^1_t - X^2_t)^2] \leq
			& 2 \mathbb{E}\left[ \int_0^t (m_1(s) - m_2(s))(X^1_s - X^2_s)\,ds \right] +\\
		\nonumber	& 2 \mathbb{E}\left[ \int_0^t (X^1_s - X^2_s) \left(  \int_0^s (H_1(s, u)\widetilde{W}_1(s,u) -  H_2(s, u)\widetilde{W}_2(s,u))\,du \right)\,ds \right]
\end{align}

\noindent
{\bf Term} $2 \mathbb{E}\left[ \int_0^t (m_1(s) - m_2(s))(X^1_s - X^2_s)\,ds \right]$\ \\
By the inequality 
\begin{equation} \label{eq:inequalityb}
ab \leq \frac{1}{2}(a^2 + b^2)
\end{equation}
 true for any two real numbers $a$ and $b$ we obtain
\begin{multline}\label{eq:e2}
2 \mathbb{E}\left[ \int_0^t (m_1(s) - m_2(s))(X^1_s - X^2_s)\,ds \right] \\
\leq \int_0^t (m_1(s) - m_2(s))^2\,ds +  \int_0^t \mathbb{E}\left[ (X^1_s - X^2_s)^2 \right]\,ds \\= 
\Vert m_1 - m_2 \Vert_{L^2([0.t])}^2 +  \int_0^t \mathbb{E}\left[ (X^1_s - X^2_s)^2 \right]\,ds
\end{multline}

\noindent
{\bf Term} $2 \mathbb{E}\left[ \int_0^t (X^1_s - X^2_s) \left(  \int_0^s (H_1(s, u)\widetilde{W}_1(s,u) -  H_2(s, u)\widetilde{W}_2(s,u))\,du \right)\,ds \right]$\ \\
By the same inequality we have
\begin{multline*}
2 \mathbb{E}\left[ \int_0^t (X^1_s - X^2_s) \left(  \int_0^s (H_1(s, u)\widetilde{W}_1(s,u) -  H_2(s, u)\widetilde{W}_2(s,u))\,du \right)\,ds \right] \leq \\
\int_0^t \mathbb{E}\left[ (X^1_s - X^2_s)^2\right]\,ds  + \\
 \int_0^t \mathbb{E}\left[ \left(  \int_0^s (H_1(s, u)\widetilde{W}_1(s,u) -  H_2(s, u)\widetilde{W}_2(s,u))\,du \right)^2 \right]\, ds 
\end{multline*}
We proceed by upperbounding the second term in the right hand side of this inequality.

First note that
\begin{multline*}
H_1(s, u)\widetilde{W}_1(s,u) -  H_2(s, u)\widetilde{W}_2(s,u) = \\
H_1(s,u)\left( \int_0^u \tilde{K}_1^u(u,v)\,dW_v \right) -  H_2(s,u)\left( \int_0^u \tilde{K}_2^u(u,v)\,dW_v \right) = \\
H_1(s,u) \left( \int_0^u ( \tilde{K}_1^u(u,v) -  \tilde{K}_2^u(u,v) )\,dW_v \right) + (H_1(s,u) - H_2(s,u)) \left( \int_0^u \tilde{K}_2^u(u,v)\,dW_v \right),
\end{multline*}
so that, by \eqref{eq:inequalityb} and Cauchy-Schwarz,
\begin{multline*}
\left(   \int_0^s (H_1(s, u)\widetilde{W}_1(s,u) -  H_2(s, u)\widetilde{W}_2(s,u))\,du     \right)^2 \leq \\
2 \int_0^s H_1(s,u)^2\,du \times \int_0^s \left( \int_0^u ( \tilde{K}_1^u(u,v) -  \tilde{K}_2^u(u,v) )\,dW_v \right)^2 \,du  + \\
 2\int_0^s (H_1(s,u) - H_2(s,u))^2\,du \times \int_0^s\left( \int_0^u \tilde{K}_2^u(u,v)\,dW_v \right)^2 \,du .
\end{multline*}
We take the expected value of both sides of this inequality and apply twice Ito's isometry formula so that
\begin{multline*}
\mathbb{E}\left[ \left(   \int_0^s (H_1(s, u)\widetilde{W}_1(s,u) -  H_2(s, u)\widetilde{W}_2(s,u))\,du     \right)^2 \right] \leq \\
 2\int_0^s H_1(s,u)^2\,du \times \int_0^s   \int_0^u ( \tilde{K}_1^u(u,v) -  \tilde{K}_2^u(u,v) )^2\,dv  \,du + \\
 2\int_0^s (H_1(s,u) - H_2(s,u))^2\,du \times \int_0^s  \int_0^u \tilde{K}_2^u(u,v)^2\,dv  \,du
\end{multline*}
By \eqref{eq:boundnormHb} $(H_1)^2$ can be upperbounded by a constant depending only on $T$. Moreover by \eqref{eq:hsnormdef} $\int_0^s  \int_0^u \tilde{K}_2^u(u,v)^2\,dv  \,du \leq \| \bar{\tilde{K}}_2 \|_{t,2}^2$. By definition of $\bar{\tilde{K}}_2^t$, and using the fact that $({\rm Id} + \bar{K}_2^t)^{-1}$ is contracting (see Remark~\ref{rem:contraction}), we write
\[
 \Vert \bar{\tilde{K}}_2^t \Vert_{t,2} = \Vert \bar{K}_2^t \circ ({\rm Id} + \bar{K}_2^t)^{-1} \Vert_{t,2} \leq  \Vert \bar{K}_2^t \Vert_{t,2} \, \Vert ({\rm Id} + \bar{K}_2^t)^{-1} \Vert_{t, 2} \leq \Vert \bar{K}_2^t \Vert_{t,2} \leq C_T
\]
for some positive constant $C_T$ by \eqref{eq:tchs} and \eqref{eq:normineq}. To sum up, we have found that there exists a positive constant $C_T$ such that
\begin{multline}\label{eq:e3}
\int_0^t \mathbb{E}\left[ \left(   \int_0^s (H_1(s, u)\widetilde{W}_1(s,u) -  H_2(s, u)\widetilde{W}_2(s,u))\,du     \right)^2 \right]\,ds \leq \\
C_T \Bigg(  \int_0^t  \int_0^s   \int_0^u ( \tilde{K}_1^u(u,v) -  \tilde{K}_2^u(u,v) )^2\,dv  \,du \,ds + \\
 \int_0^t \int_0^s (H_1(s,u) - H_2(s,u))^2\,du \,ds \Bigg)
\end{multline}
The term $\int_0^s   \int_0^u ( \tilde{K}_1^u(u,v) -  \tilde{K}_2^u(u,v) )^2\,dv\,du$ is less than $\| \bar{\tilde{K}}_1^t- \bar{\tilde{K}}_2^t \|_{t,2}^2$ and
\begin{align*}
\bar{\tilde{K}}_1^t - \bar{\tilde{K}}_2^t =& \bar{K}_1^t \circ ({\rm Id} + \bar{K}_1^t )^{-1} - \bar{K}_2^t \circ ({\rm Id} + \bar{K}_2^t )^{-1}\\
								=& ({\rm Id} + \bar{K}_1^t )^{-1} \circ \bar{K}_1^t - \bar{K}_2^t \circ ({\rm Id} + \bar{K}_2^t )^{-1}\\
								=& ({\rm Id} + \bar{K}_1^t )^{-1} \circ \left(\bar{K}_1^t \circ ({\rm Id} + \bar{K}_2^t ) - ({\rm Id} + \bar{K}_1^t ) \circ \bar{K}_2^t\right) \circ ({\rm Id} + \bar{K}_2^t )^{-1} \\
								=& ({\rm Id} + \bar{K}_1^t )^{-1} \circ ( \bar{K}_1^t - \bar{K}_2^t ) \circ ({\rm Id} + \bar{K}_2^t )^{-1}.
\end{align*}
We take the norm of the two sides of this equality and, using again the fact that the operators $({\rm Id} + \bar{K}_1^t )^{-1}$ and $({\rm Id} + \bar{K}_2^t )^{-1}$ are contracting 
(Remark~\ref{rem:contraction}), we obtain
\begin{equation}\label{eq:K1tildeminusK2tilde}
\| \bar{\tilde{K}}_1^t - \bar{\tilde{K}}_2^t \|_{t,2} \leq \|  \bar{K}_1^t - \bar{K}_2^t  \|_{t,2},
\end{equation}
and therefore
\begin{equation}\label{eq:e4}
\int_0^t \int_0^s   \int_0^u ( \tilde{K}_1^u(u,v) -  \tilde{K}_2^u(u,v) )^2\,dv\,du \,ds \leq T \|  \bar{K}_1 - \bar{K}_2 \|_{t,2}^2
\end{equation}
Consider now the obvious inequality
\[
 \int_0^t  \int_0^s   ( H_1(s,u) - H_2(s,u) )^2 \,du  \,ds \leq C_T \Vert \bar{H}_1 - \bar{H}_2 \Vert^2_{t,2}.
\]
According to \eqref{eq:H} in the case $\kappa(t,s) = \tilde{K}^t(t,s)\mathbbm{1}_{t\geq s}$
\begin{align*}
\bar{H}_1 - \bar{H}_2 &= ({\rm Id} -\bar{\tilde{K}}_1^t)^{-1} -  ({\rm Id} -  \bar{\tilde{K}}_2^t)^{-1} \\
& =
 ({\rm Id} - \bar{\tilde{K}}_1^t)^{-1} \circ  (\bar{\tilde{K}}_1^t - \bar{\tilde{K}}_2^t) \circ ({\rm Id} -  \bar{\tilde{K}}_2^t)^{-1},
\end{align*}
so that we have
\begin{align*}
\Vert \bar{H}_1 - \bar{H}_2 \Vert_{t,2} & \leq  \Vert ({\rm Id} -  \bar{\tilde{K}}_1^t)^{-1} \Vert_{t,2}\, \Vert \bar{\tilde{K}}_1^t - \bar{\tilde{K}}_2^t \Vert_{t,2} \,
\Vert ({\rm Id} - \bar{\tilde{K}}_2^t)^{-1} \Vert_{t,2}\\
& = \Vert \bar{H}_1 + {\rm Id}\Vert_{t,2} \, \Vert \bar{\tilde{K}}_1^t - \bar{\tilde{K}}_2^t \Vert_{t,2} \,\Vert \bar{H}_2  + {\rm Id}\Vert_{t,2}\\
& \leq ( \Vert \bar{H}_1 \Vert_{t,2} + \Vert  {\rm Id}\Vert_{t,2} ) \Vert \bar{\tilde{K}}_1^t - \bar{\tilde{K}}_2^t \Vert_{t,2} ( \Vert \bar{H}_2 \Vert_{t,2} + \Vert  {\rm Id}\Vert_{t,2} )\\
& = ( \Vert \bar{H}_1 \Vert_{t,2} + 1 ) \Vert \bar{\tilde{K}}_1^t - \bar{\tilde{K}}_2^t \Vert_{t,2} ( \Vert \bar{H}_2 \Vert_{t,2} + 1 ).
\end{align*}
 Because of \eqref{eq:boundnormHb}  $\Vert \bar{H}_1 \Vert_{t,2}$ and $\Vert \bar{H}_2 \Vert_{t,2}$ are upper-bounded by a constant depending only on $T$ and thus
\[
\Vert \bar{H}_1 - \bar{H}_2 \Vert_{t,2} \leq C_T \Vert \bar{\tilde{K}}_1^t - \bar{\tilde{K}}_2^t \Vert_{t,2},
\]
and, by \eqref{eq:K1tildeminusK2tilde},
\begin{equation}\label{eq:fourth12}
\Vert \bar{H}_1 - \bar{H}_2 \Vert_{t,2} \leq C_T \Vert \bar{K}_1^t - \bar{K}_2^t \Vert_{t,2}
\end{equation}
for some positive constant $C_T$.
We have obtained
\begin{equation}\label{eq:e5}
\int_0^t  \int_0^s   ( H_1(s,u) - H_2(s,u) )^2 \,du  \,ds \leq C_T \Vert \bar{K}_1 - \bar{K}_2^t \Vert_{t,2}.
\end{equation}
Now by combining \eqref{eq:e1}-\eqref{eq:e5} we obtain
\[
\mathbb{E}[(X^1_t-X^2_t)^2] \leq C_T\left( \int_0^t \mathbb{E}[(X^1_s-X^2_s)^2]\,ds+\Vert m_1 - m_2 \Vert^2_{L^2([0,t])} + \Vert \bar{K}_1 - \bar{K_2} \Vert_{t,2}^2 \right),
\]
so that Gronwall's Lemma commands that
\begin{equation}\label{eq:almost}
\mathbb{E}[(X^1_t-X^2_t)^2] \leq C_T e^{T C_T} \left(  \Vert m_1 - m_2 \Vert^2_{L^2([0,t])} + \Vert \bar{K}_1 - \bar{K_2} \Vert_{t,2}^2   \right) = D_T d_t((m_1,\bar{K}_1),(m_2,\bar{K}_2))
\end{equation}
for some positive constant $D_T$.

To conclude we have
\[
d_t(\Lambda((m_1,\bar{K}_1)),\Lambda((m_2,\bar{K}_2)))^2 = \Vert p_1 - p_2 \Vert^2_{L^2([0,t]} + \Vert \bar{G}_1 - \bar{G}_2 \Vert^2_{t,2},
\]
and, by \eqref{eq:p1minusp2} and \eqref{eq:Lambdadif},
\begin{multline*}
d_t(\Lambda((m_1,\bar{K}_1)),\Lambda((m_2,\bar{K}_2)))^2 \\
\leq C_T \left( \int_0^t \mathbb{E}[(X^1_s-X^2_s)^2] ds+ \int_0^t \int_0^t \mathbb{E}[(X^1_u-X^2_u)^2] \,du\,ds    \right).
\end{multline*}
By \eqref{eq:almost}
\[
d_t(\Lambda((m_1,\bar{K}_1)),\Lambda((m_2,\bar{K}_2)))^2 \leq c_T \int_0^t d_s((m_1,\bar{K}_1),(m_2,\bar{K}_2))^2\,ds,
\]
for some positive constant $c_T$, which finishes the proof.
\end{proof}
We can now state the Theorem:
\begin{theorem}\label{theo:uniqueness}
The map $\Lambda$ has a unique fixed point equal to $(m_Q,\bar{K}_Q)$ in \eqref{eq:dynmean}.
\end{theorem}
\begin{proof}
The existence and uniqueness of the fixed point are classically obtained by iterating Lemma~\ref{lem:contraction} to obtain
\[
d_T(\Lambda^{k+1}((m,\bar{K})), \Lambda^k((m,\bar{K})))^2 \leq \frac{c_T^k T^k}{k!} d_T(\Lambda((m,\bar{K})), (m,\bar{K}))^2.
\]
Existence of a fixed point arises from the fact that $\Lambda^k((m,\bar{K}))$ is a Cauchy sequence, hence convergent. Uniqueness is proved by assuming the existence of two fixed points $(m_1^\infty,\bar{K}_1^\infty)$ and $(m_2^\infty,\bar{K}_2^\infty)$ which must satisfy
\[
d_t((m_1^\infty,\bar{K}_1^\infty),(m_2^\infty,\bar{K}_2^\infty))^2 \leq c_T \int_0^t d_s((m_1^\infty,\bar{K}_1^\infty),(m_2^\infty,\bar{K}_2^\infty))^2\,ds,
\]
and by Grönwall Lemma this commands that
\[
d_t((m_1^\infty,\bar{K}_1^\infty),(m_2^\infty,\bar{K}_2^\infty)) = 0 \quad 0 \leq t \leq T
\]
\end{proof}
This Lemma provides us with the following estimate of the rate of convergence to the fixed point $(m^{\infty},\bar{K}^{\infty})$:
\begin{corollary}
	Rate of convergence to the Fixed Point \(K^\infty\)
	\begin{align*}
	d_T((m^{\infty},\bar{K}^{\infty}),\Lambda^\ell((m,\bar{K})))^2 &\leq \sum_{j\geq 
	\ell}d_T(\Lambda^{j+1}((m,\bar{K})),\Lambda^{j}((m,\bar{K})))^2 \\
	&\simeq \dfrac{c_T^\ell T^\ell}{\ell !} d_T(\Lambda((m,\bar{K})),(m,\bar{K}))^2
	\end{align*}

\end{corollary}
It also implies that $m_Q(t)$ is zero in the S-model case.
\begin{corollary}\label{cor:zeromean}
If $\mu_0$ is symmetric (which implies $\mathbb{E}[X_0]=0$), the S-model satisfies
\[
m_Q(t) = 0\quad  \forall\ 0 \leq t \leq T, \quad \forall J.
\]
\end{corollary}
\begin{proof}
In the S-model case, the function \(f\) is the identity, the function \(g\) is defined by \eqref{eq:g}.
Using the symmetry of the Brownian motion and the  fact that \(f\) and \(g\) are odd,
we immediately deduce that for any process \(X_t\) solution of \eqref{eq:dynmean}, \(- X_t\) is 
also a solution. If in addition \(\mu_0\) is symmetric, the processes \(X_t\) and \(- X_t\) start with the same law.
By uniqueness of the solution, we deduce that for any \(t \geq 0\), the law of \(X_t\) and \(- X_t\) are equal.
As a consequence, we have \(\mathbb{E}[X_t]=0\).
\end{proof}
%%%%%%%%%%%%%%%%%%%%%%%%
%%%%%%%%%%%%%%%%%%%%%%%%
\section{Numerical simulations of the limit dynamics}\label{sect:numerics}
We have implemented the fixed point algorithm $\Lambda$ in the Julia language. 
We choose a time $T$, a time step $\Delta t$ and we estimate a realisation of  the solution on $[0,T]$
by discretising \eqref{eq:dynmean}. For simplicity again we only consider the case $\lambda=1$.

We note $\Delta \Phi_l$, $l=0,\cdots,L-1$,  $L=\lfloor \frac{T}{\Delta t} \rfloor$,  the difference $\Phi_{(l+1)\Delta t}-\Phi_{l \Delta t}$ for $\Phi \in \{W,\,B,\,X\}$.

At iteration $n \geq 1$ of the numerical approximation of $\Lambda$ we are given $m^{(n-1)}(l  \Delta t)$ and $\tilde{K}^{(n-1)}(l \Delta t, j \Delta s)$, $j \leq l=0,\cdots,L - 1$ , 
and we produce $P$ independent solutions of \eqref{eq:dynmean} noted $X^{n,p}_l$, $l=0,\cdots,L$, 
$p=1,\cdots,P$. These are obtained by drawing sequences of  independent Brownian increments $\Delta W_l^{n,p} \simeq \mathcal{N}(0,\Delta t)$, $l = 0,\cdots L-1$, $p=1,\cdots,P$, 
\(n \geq 1\).

Discretising \eqref{eq:dynmean} yields the equations
\begin{align}\label{eq:firstEM}
\Delta X_l^{n,p} &= g(X_l^{n,p}) \Delta t + J m^{(n-1)}(l \Delta t) \Delta t +  \Delta B_l^{n,p}\\
\nonumber \Delta B_l^{n,p} &= \Delta W_l^{n,p} + \Delta t \sum_{j=0}^{l-1} \tilde{K}^{l \Delta t, (n-1)}(l\Delta t,j\Delta t) \Delta B_j^{n,p} \quad l = 0,\cdots L-1,
\end{align}
One then computes
\begin{align*}
m^{(n)} (l \Delta t) & = \frac{1}{P} \sum_p f(X_l^{n,p}) \ l = 0,\cdots, L-1 \\
K^{(n)} (l \Delta t, j \Delta t) & = \frac{1}{P} \sum_p f(X_l^{n,p})f(X_j^{n,p})\ j, l = 0,\cdots, L-1.
\end{align*}
The iteration of the map $\Lambda$ produces a sequence of symmetric functions $K^{(n)}(t,s)$, $0 \leq s,\, t \leq T$ and $\tilde{K}^{t,\,(n)}(s,u)$, $0 \leq u \leq s \leq t$. 
In the discrete time setting these functions are approximated by the $L \times L$ symmetric matrices $K^{(n)}(l \Delta t, j \Delta t)$, $0 \leq j, l \leq L-1$ and the $L$ $i \times i$ symmetric matrices $\tilde{K}^{i \Delta t,\,(n)}(l \Delta t, j\Delta t)$ , $0 \leq j, l \leq i-1$, $i=1,\cdots,L$.

Equation \eqref{eq:KQtildeopb} indicates that the $i$th symmetric $i \times i$ matrix $\tilde{K}^{i\Delta t,\,(n)}$ is obtained from the $i \times i$ sub-matrix $K^{i \Delta t,\,(n)}$ of the $L \times L$ matrix $K^{(n)}$ by
\[
\tilde{K}^{i\Delta t,\,(n)} = K^{i\Delta t,\,(n)} \left( {\rm Id}_i +  \Delta t K^{i\Delta t,\,(n)} \right)^{-1} = \left( {\rm Id}_i + \Delta t  K^{i\Delta t,\,(n)} \right)^{-1} K^{i\Delta t,\,(n)},
\]
where ${\rm Id}_i$ is the $i \times i$ identity matrix.

We then use $m^{(n)}$ and $\tilde{K}^{(n)}$ in \eqref{eq:firstEM} to compute the next approximation of $\Lambda$.
\begin{remark}
Note that since in the second equation of \eqref{eq:firstEM} only the $l$th row of $K^{l\Delta t,\,(n-1)}$ is needed, we do not have to compute $\left( {\rm Id}_l +  \Delta t K^{l\Delta t,\,(n-1)} \right)^{-1}$ but we rather solve for $x$ the linear system
\[
\left( {\rm Id}_l +  \Delta t K^{l\Delta t,\,(n-1)} \right)x =  K^{l\Delta t,\,(n-1)} e_l,
\]
where $e_l$ is the $l$th vector of the standard basis of $\mathbb{R}^l$. This can be achieved very efficiently by, for example, the conjugate gradient algorithm.
\end{remark}
%%%%%%%%%%%
%%%%%%%%%%%
\subsection{Numerical simulation results: the H-model}
We consider the case where the function $f$ is given by
\[
f(x) = \frac{1}{2}(1 + \tanh(x))
\]
As shown in Figure~\ref{fig:sigmoid1} the minimum activity is 0, the maximum is 1.  This is the function used in most works in mathematical neuroscience.
\begin{figure}[htb]
\centerline{
\includegraphics[width=0.5\textwidth]{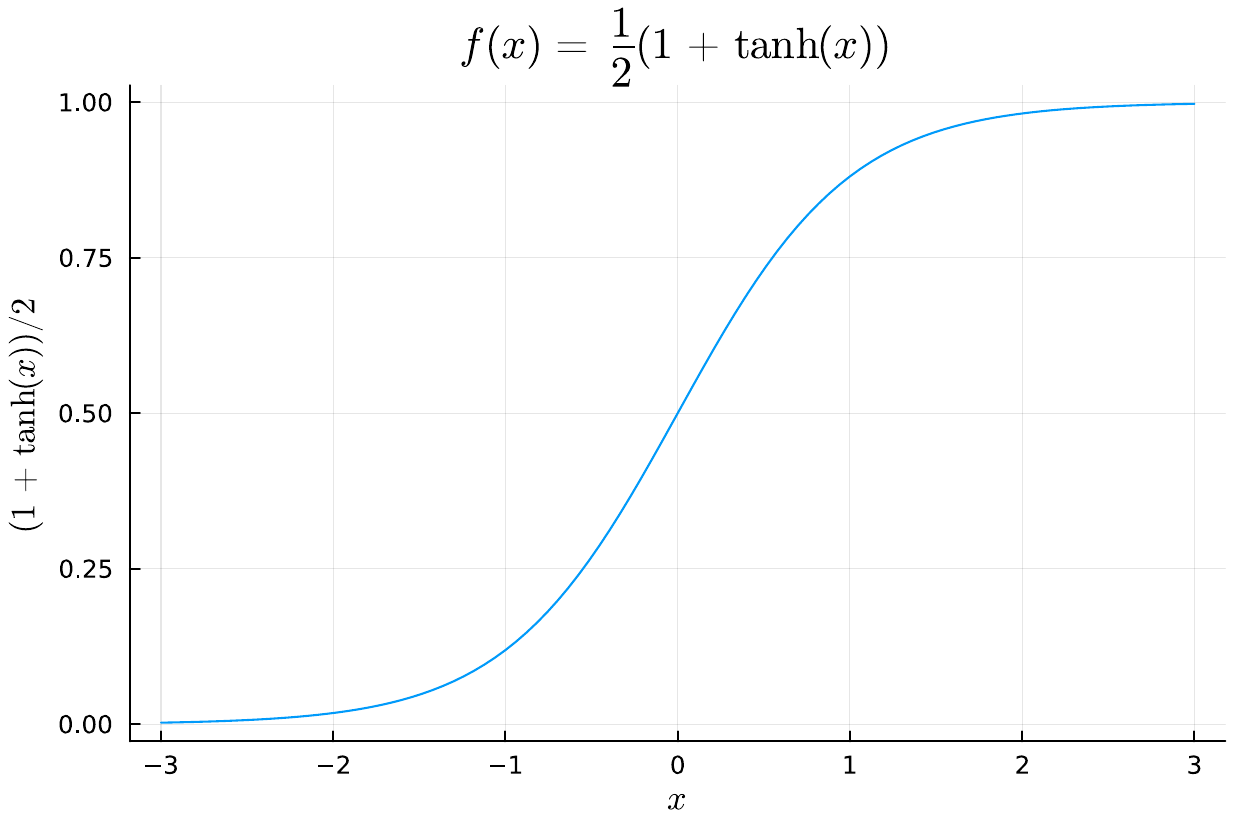}
}
\caption{The sigmoid function used in the numerical experiments for the H-model. Note that it varies between 0 and 1.}
\label{fig:sigmoid1}
\end{figure}
We show some results of the estimations produced by the iterative method $\Lambda$ as a set of composite figures.
Each figure displays, on the first row, $K_Q(t,s)$, $\tilde{K}_Q^t(t,s)$, and the covariance of the solution $X_t$, $\cov(X_t, X_s)$ , on the second row, the functions $K_Q(t,t)$, $\tilde{K}_Q^t(t,t)$, and $\cov(X_t,X_t))$; on the third row we display the function $m_Q(t)$, and two "typical" trajectories $X^i_t$ and $X^j_t$ as well as the corresponding activities $f(X^i_t)$ and $f(X^j_t)$ for various values of $J$ and $\sigma$. 
Together with the estimation of $K_Q$ and $m_Q$ we obtain an estimation of the standard deviation of these estimates.

We start with an example where $J=0.0$. Figure~\ref{fig:1} shows the results as an array of plots. 
The function $m_Q(t)$ is identically 0 by construction and thus not shown. 
The plot of $K_Q(t,t)$ in this figure shows three curves, $K_Q(t,t)$ in blue which falls between a red and a green curve. The red curve is $K_Q(t,t)$ plus twice the estimated standard deviation while the green is $K_Q(t,t)$ minus twice the estimated standard deviation. Both figures are for $\Delta t = 0.04$, the left one is for 100 000 trajectories  while the right one is for 500 000 trajectories, yielding a more accurate estimation. In both cases we have iterated $\Lambda$ ten times.
\begin{figure}[htb]
\centerline{
\includegraphics[width=0.5\textwidth]{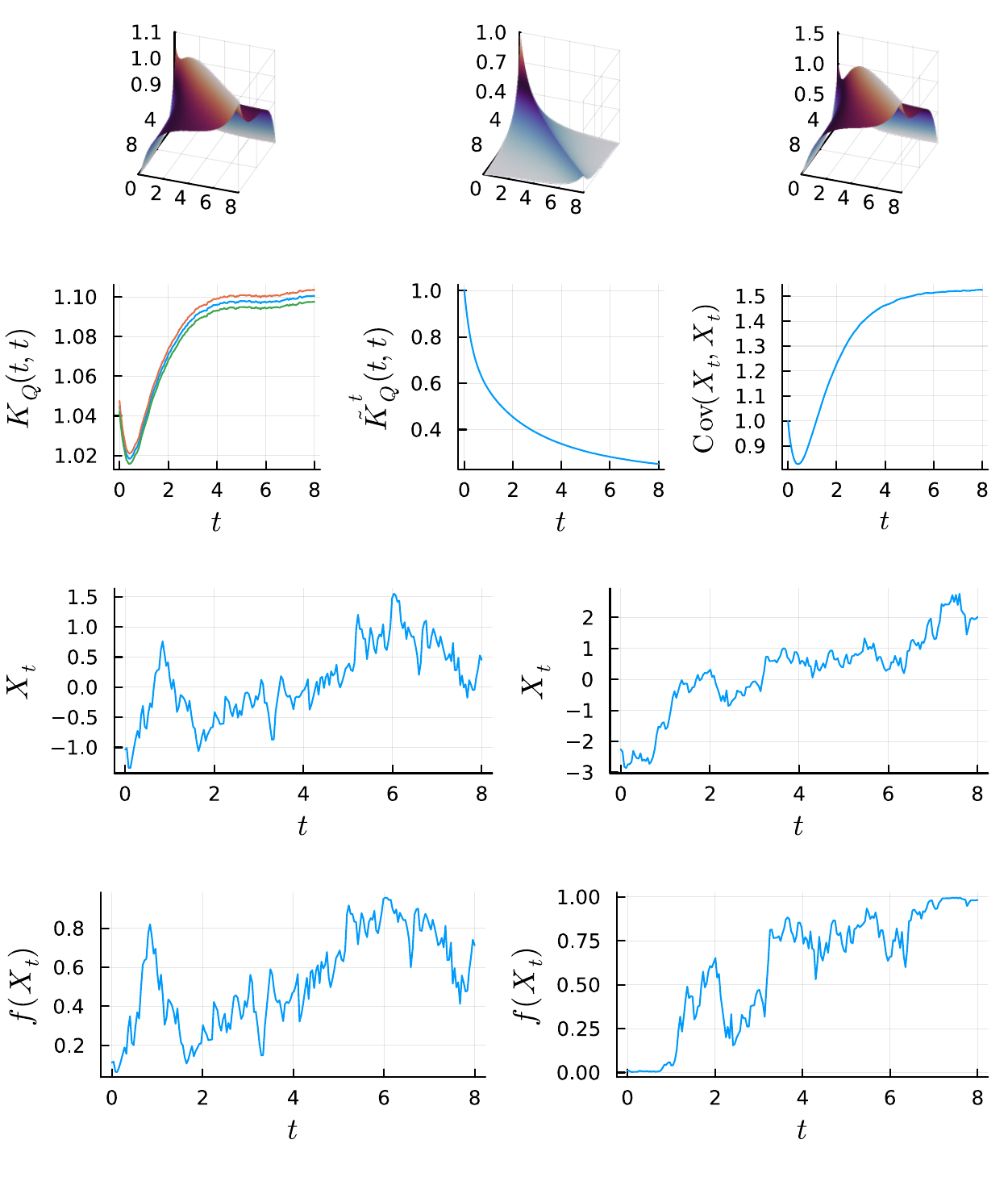}
\hspace{0.5cm}
\includegraphics[width=0.5\textwidth]{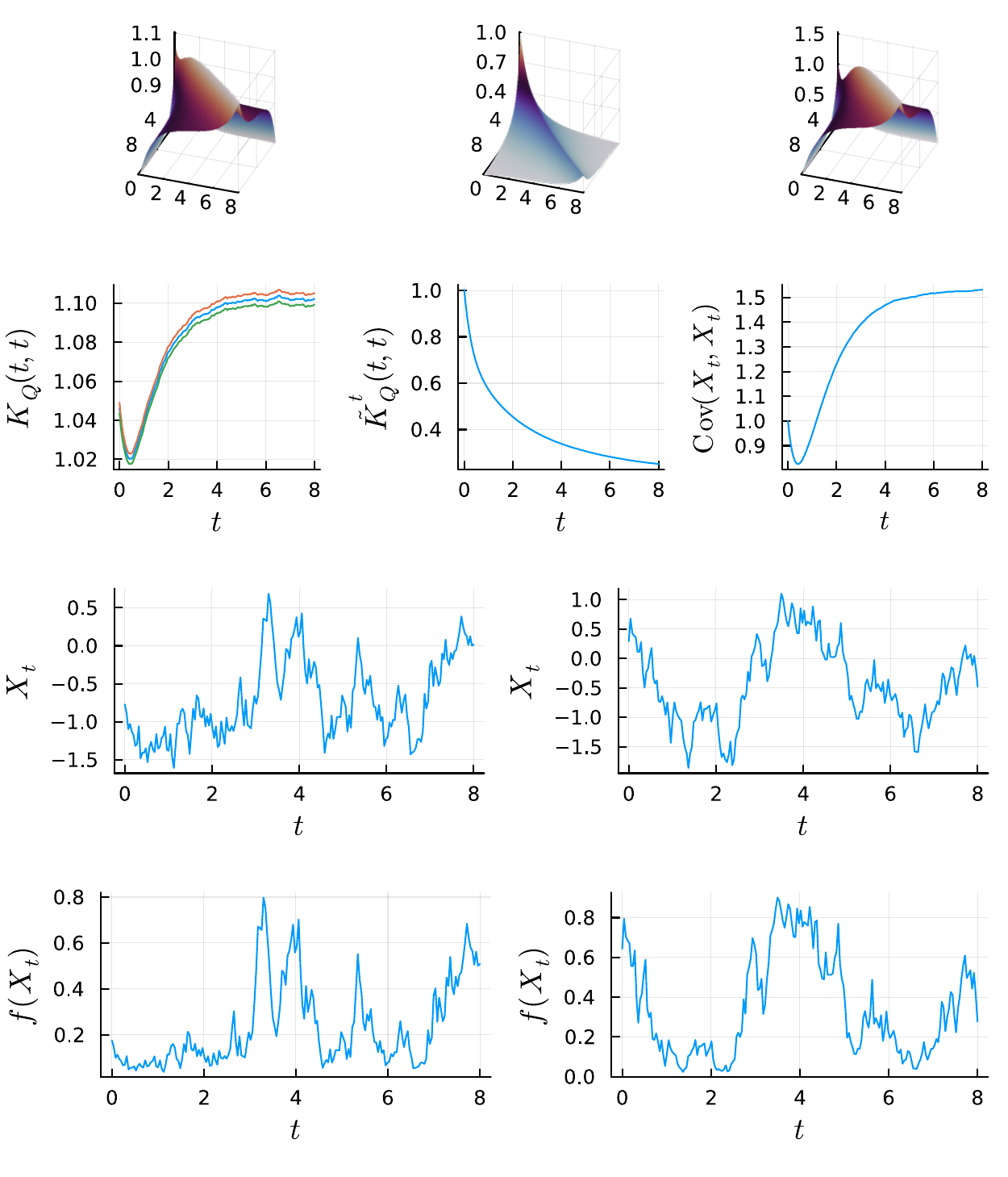}
}
\caption{H-model: $J=0,\,\sigma = 1$. The computation is performed on the time interval $[0,\,8.0]$ with $\Delta t=0.04$. The Monte Carlo method for estimating the fixed point of $\Lambda$ uses 100 000 trajectories (Left) and 500 000 trajectories (Right). $\Lambda$ is iterated 10 times. Note the improvement in the accuracy of the estimation of $K_Q$. }
\label{fig:1}
\end{figure}

We now increase the excitation level and consider the case $J=1$, shown in figure~\ref{fig:2}. The function $m_Q(t)$ is non zero and converges when time increases to a constant value slightly below 0.7 and acts as a constant injected current resulting in increased values of the potential values $X_t$ and activity values $f(X_t)$. 

We next double the excitation level and go to $J=2$. The function $m_Q$ then saturates at 1.8 resulting in a further increase in the activity level shown as $f(X_t)$, its values getting much more often very close to  the maximum value of 1. If we increase the excitation level to $J=5$ we note that, as expected, the function $m_Q(t)$ saturates at an even higher value, close to 5, resulting in activity levels very close to the maximum value of 1.
\begin{figure}[htb]
\centerline{
\includegraphics[width=0.5\textwidth]{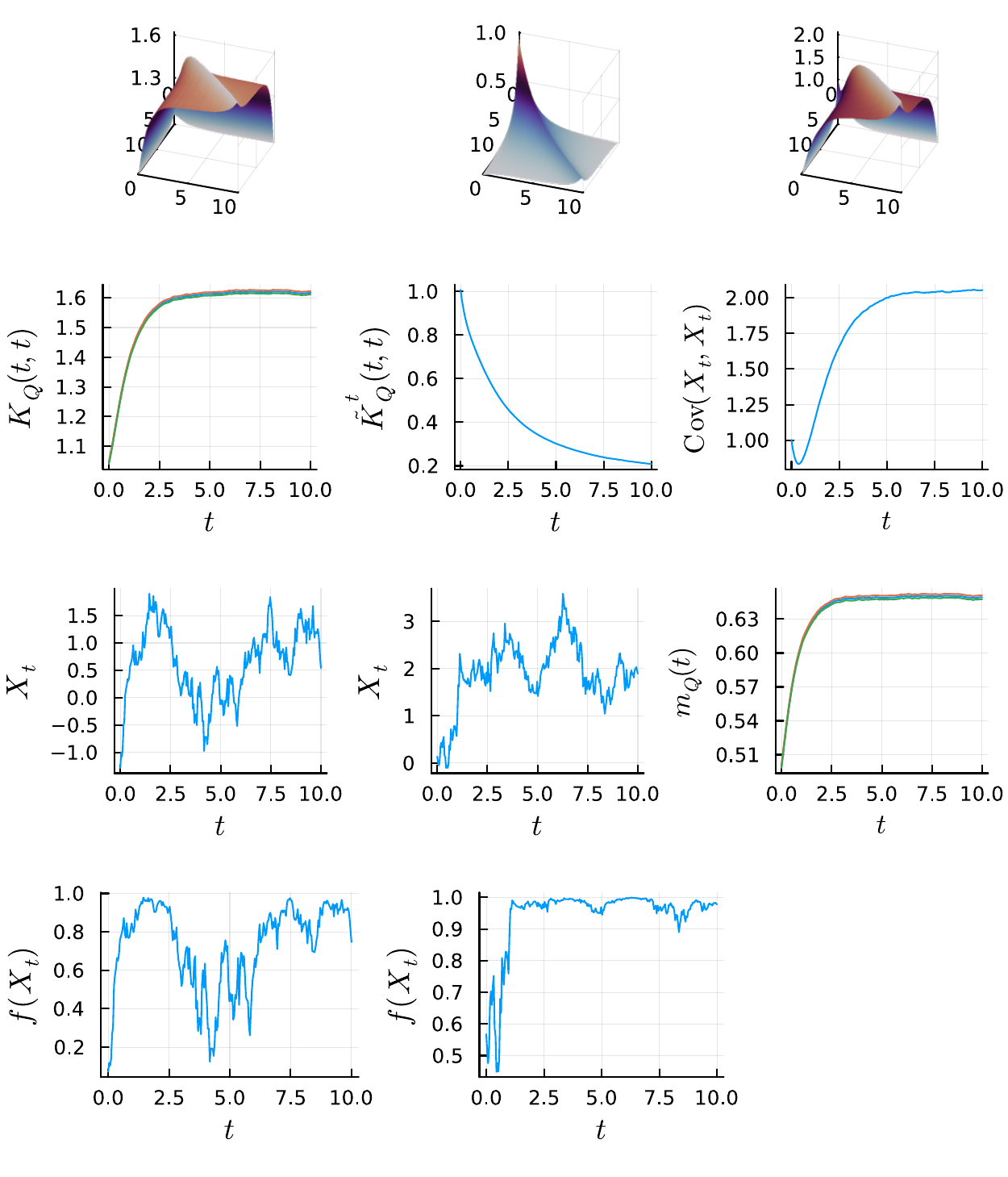}\hspace{0.1cm } 
\includegraphics[width=0.5\textwidth]{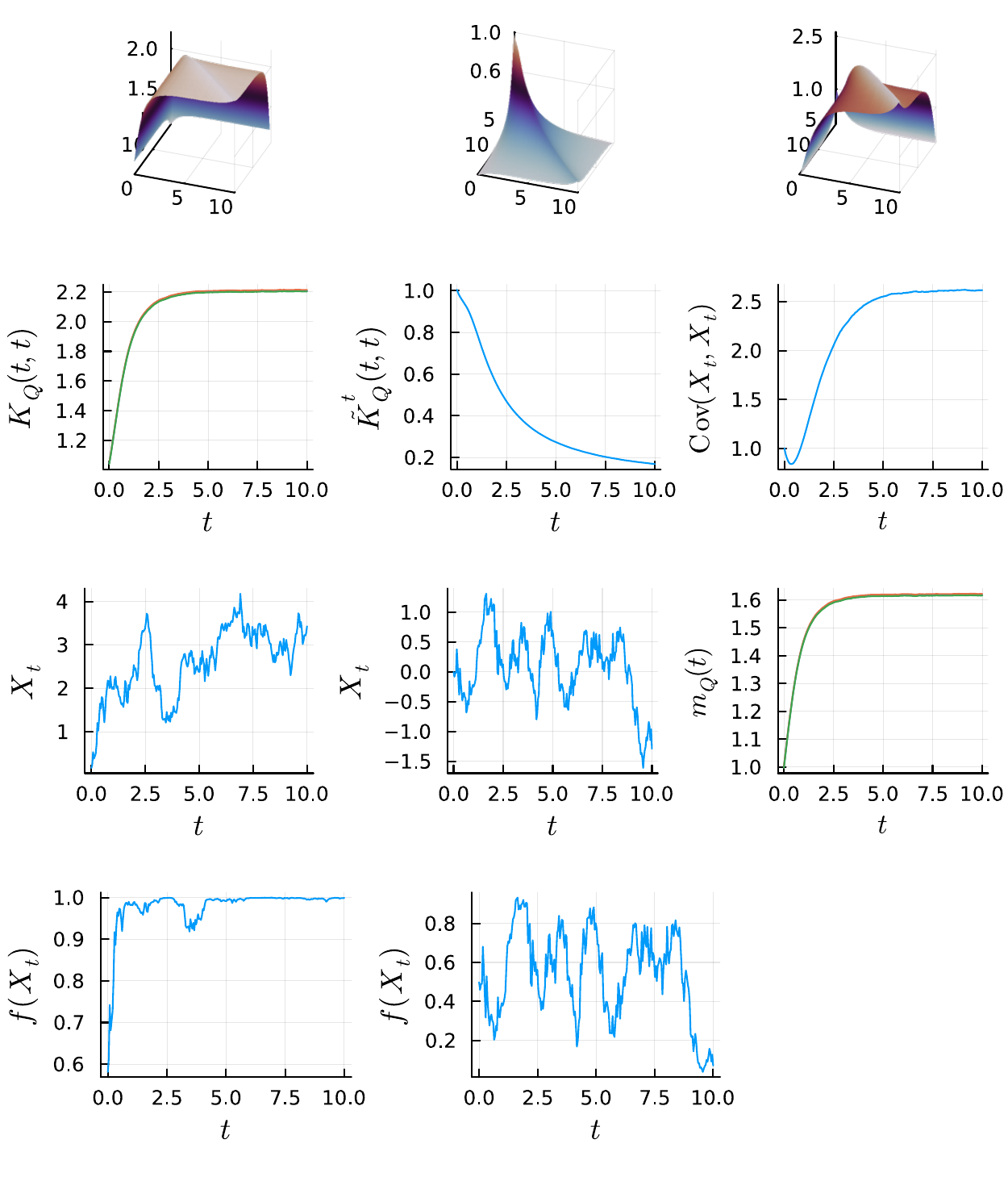}
}
\centerline{
\includegraphics[width=0.5\textwidth]{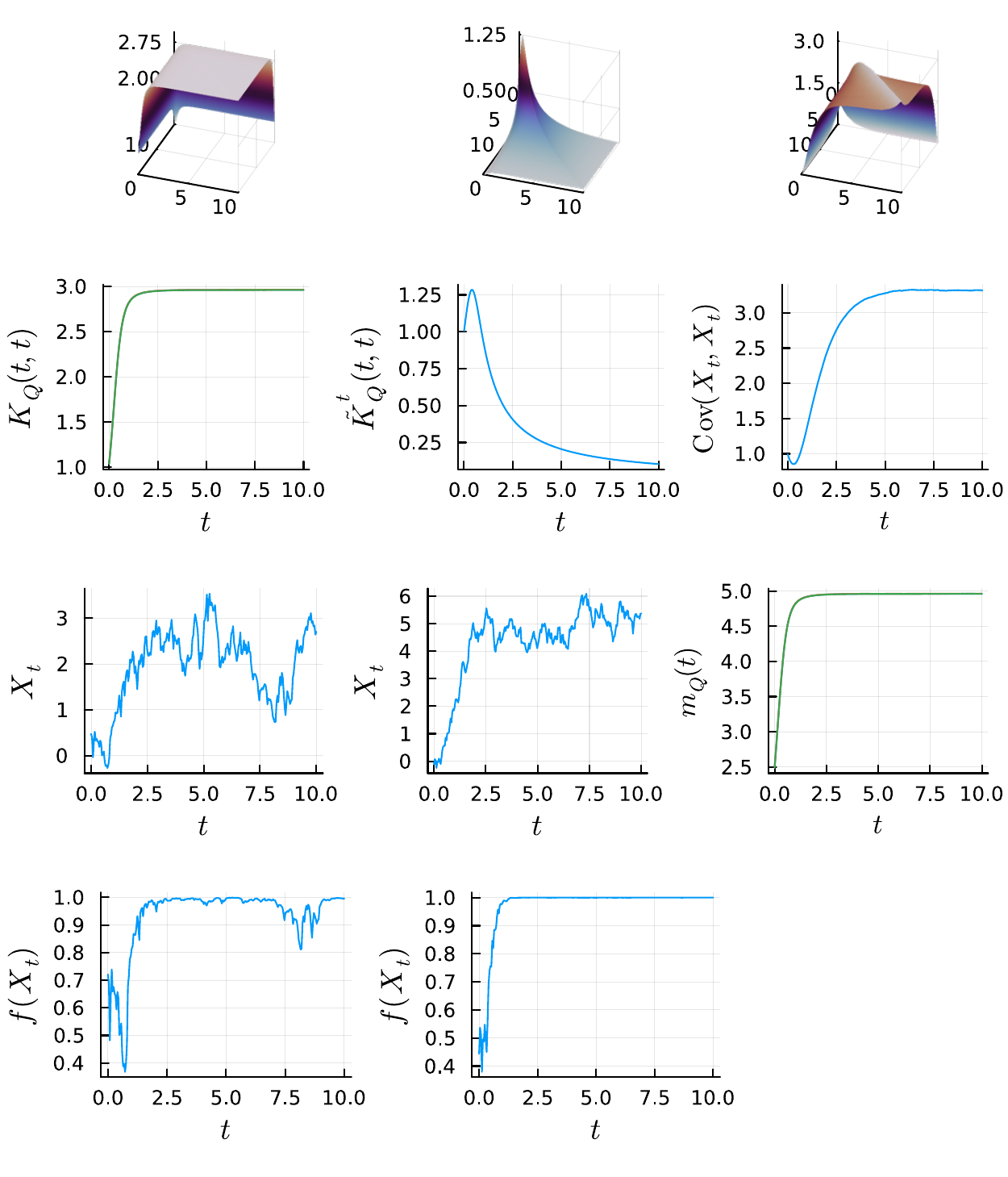}}
\caption{H-model: From left to right and top to bottom, $J=1,\,2,\,5$ and $\sigma = 1$. The computation is performed on the time interval $[0,\,10.0]$ with $\Delta t=0.04$. The Monte Carlo method for the fixed point computation uses 160 000 trajectories.}
\label{fig:2}
\end{figure}

We now turn to negative values of $J$. Figure~\ref{fig:3} has a similar structure to~\ref{fig:2}. Decreasing the values of $J$ from -1 to -2.5 and to -5 decreases the injected current $m_Q(t)$ resulting in a decrease of the neuronal acivity $f(X_t)$.
\begin{figure}[htb]
\centerline{
\includegraphics[width=0.5\textwidth]{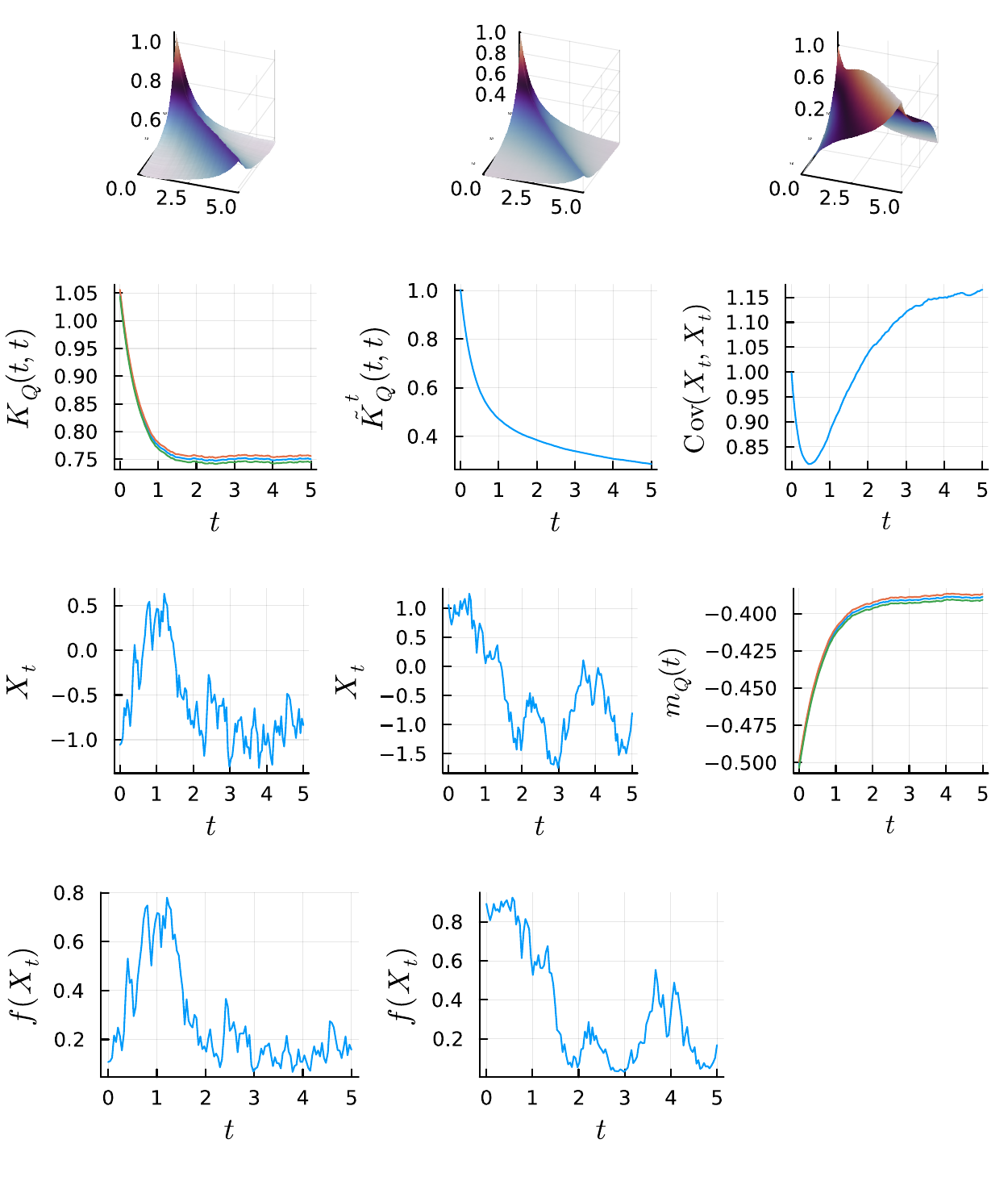}\hspace{0.3cm }
\includegraphics[width=0.5\textwidth]{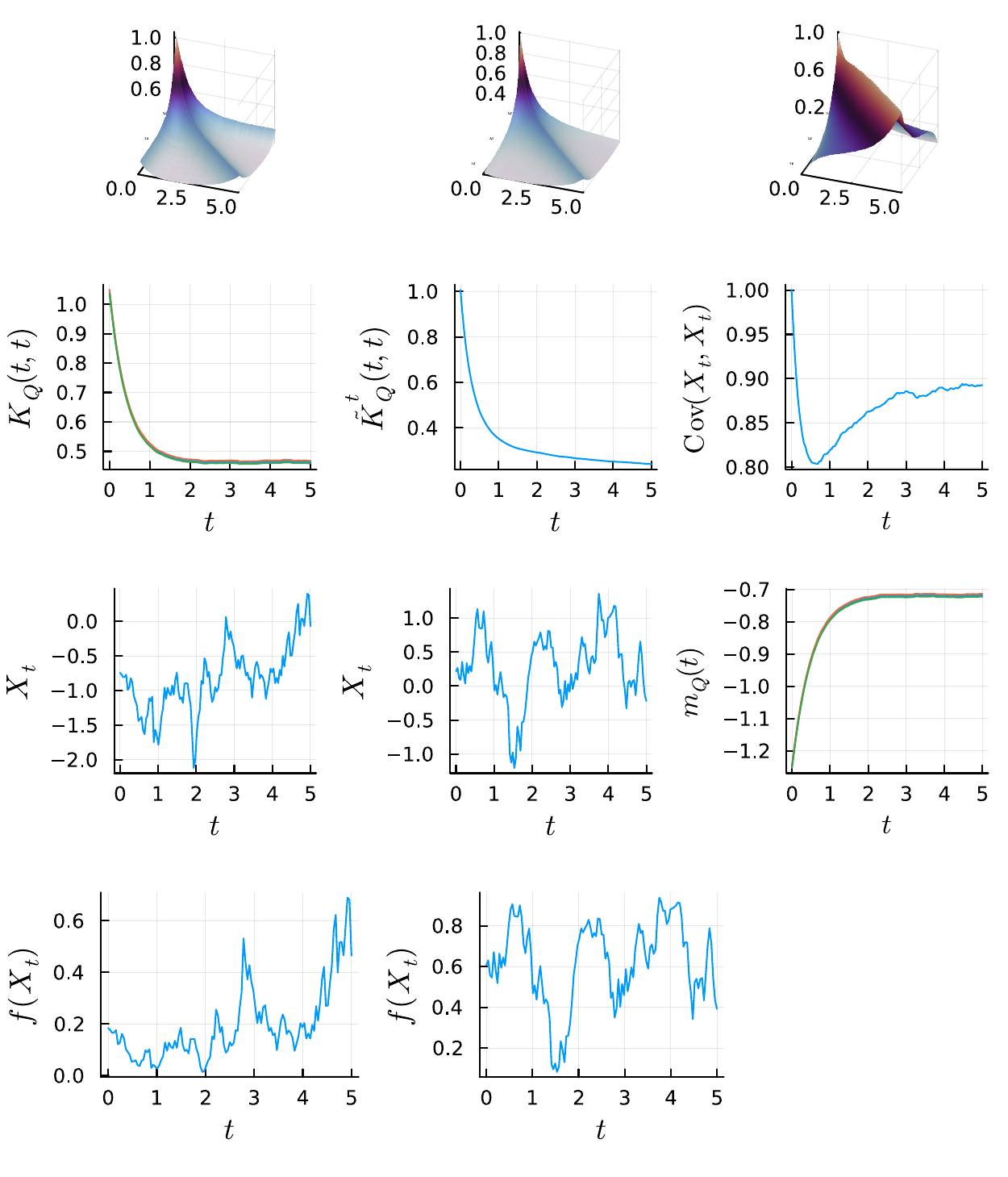}
}
\centerline{
\includegraphics[width=0.5\textwidth]{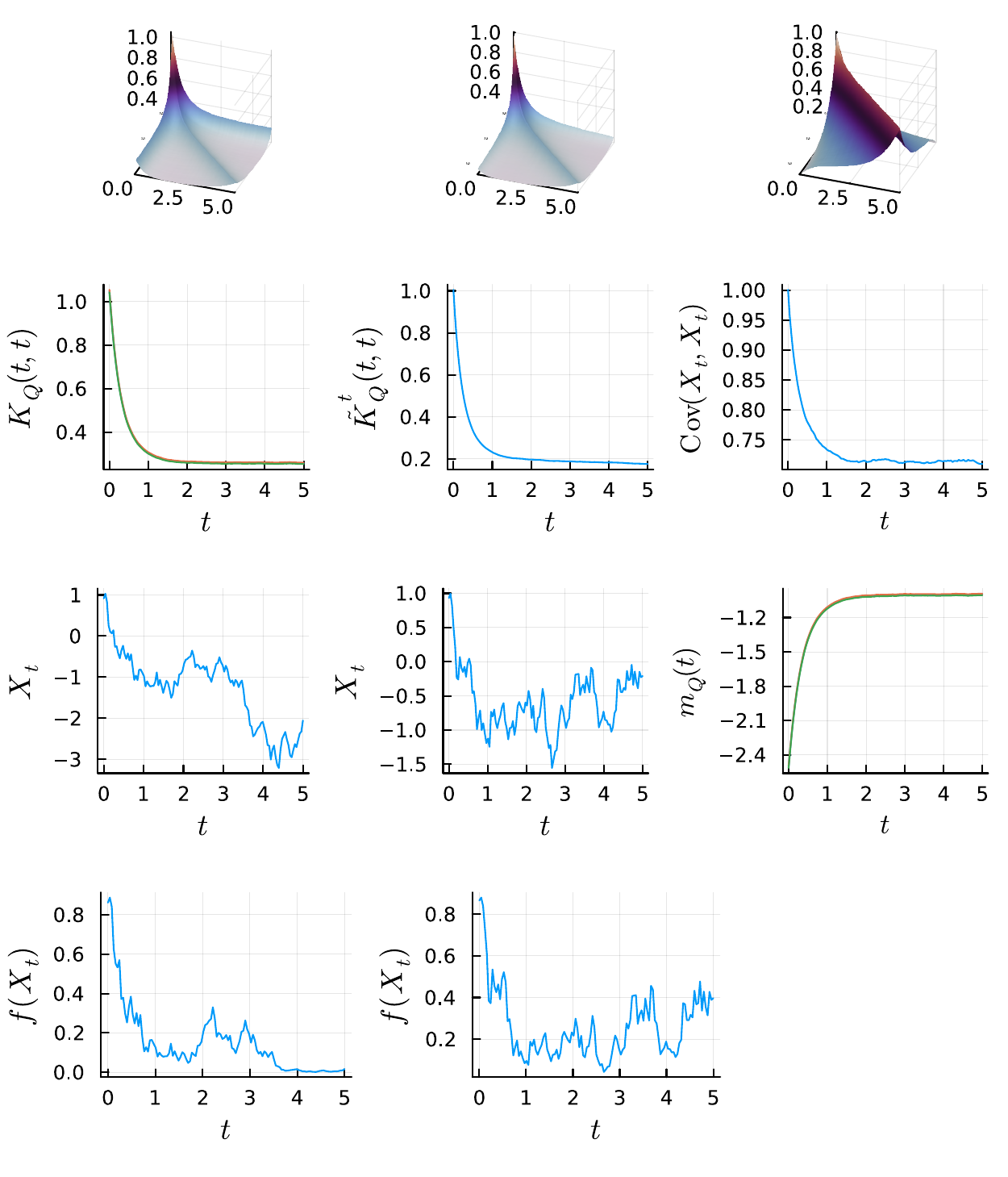}}
\caption{H-model: From left to right and top to bottom, $J=-1,\,-2.5,\,-5$ and $\sigma = 1$. The computation is performed on the time interval $[0,\,5.0]$ with $\Delta t=0.04$. The Monte Carlo method for the fixed point computation uses 100 000 trajectories.}
\label{fig:3}
\end{figure}

%%%%%%%%%%%
%%%%%%%%%%%
\subsection{Numerical simulation results: the S-model}
We next show some results of the estimations produced by the iterative method $\Lambda$ for the S-model, also as a set of composite figures. Since we know from Corollary~\ref{cor:zeromean} that $m_Q(t)$ is zero if the initial condition is zero mean, we do not show it. Moreover, changing the value of $J$ has no effect on $K_Q$ hence we only show the case $J=0$.
Each figure displays, on the first row, from left to right two 
”typical” trajectories $X^i(t)$ and $X^j(t)$.  
The second row shows $K_Q(t, s)$ and $\tilde{K}_Q^t (t, s)$. Note that for the S-model $K_Q(t, s)$ is the correlation 
$\mathbb{E}_Q(X_t X_s)$ and thus we do not show
$\cov(X_t, X_s)$ in this case. The third row shows  the two functions $K_Q(t, t)$ and $\tilde{K}_Q^t (t, t)$.

As a general comment, this case is more difficult numerically than the H-model: 
If $\Delta t$ is chosen too large, some of the generated trajectories exit the interval $]-A\ A[$. To remedy this problem, following Ethier and Kurtz, \cite{ethier_markov_2009} we change the definition \eqref{eq:U} to 
\[
U(x) = -k\log(A^2 - x^2),
\]
where in our work $k=2,\,4$. This change has the effect of having much less exiting trajectories. 

Figure~\ref{fig:6} shows the results when $\sigma=\lambda=1$ and the initial condition has 0 mean, hence $m_Q(t)=0$ (left part of the figure), or when the initial condition has mean 0.5 (right part of the figure where we 
observe that  $m_Q(t)$ decreases to 0).
\begin{figure}[htb]
\centerline{
\includegraphics[width=0.5\textwidth]{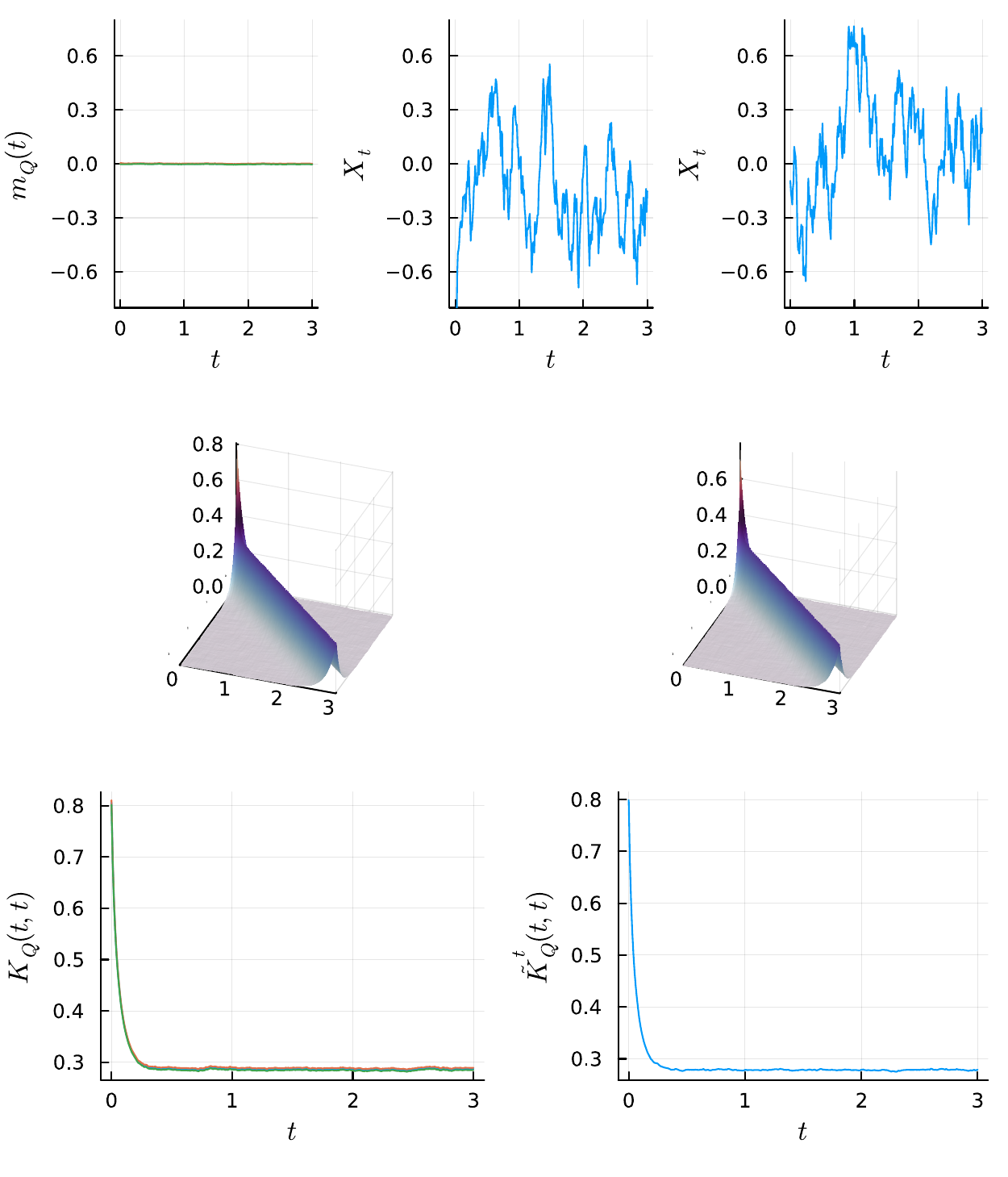} \hspace{0.1cm}
\includegraphics[width=0.5\textwidth]{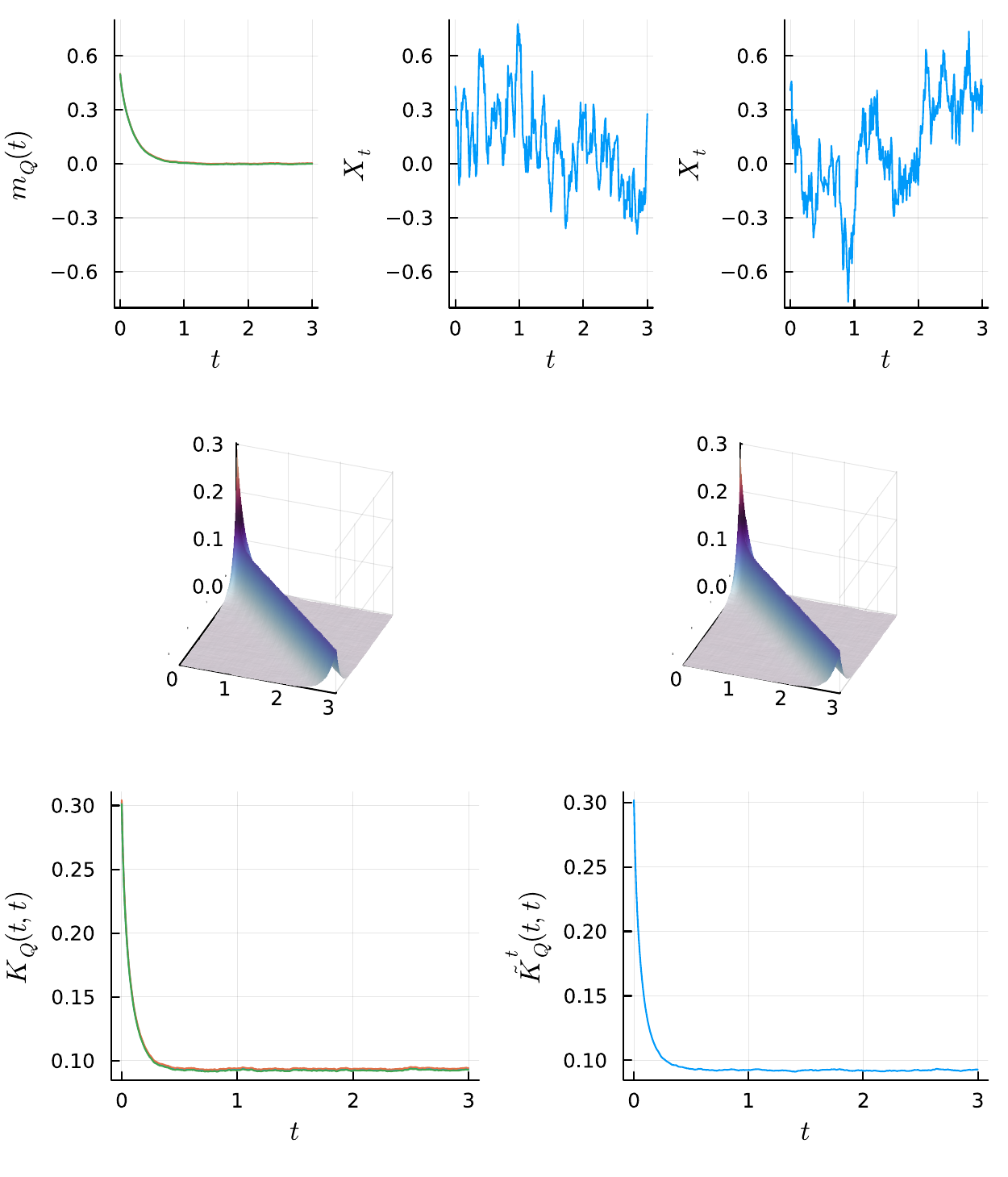}
}
\caption{S-model: $\sigma = 1$. Left: zero mean initial condition. $m_Q(t)=0$ at all times. Right: initial condition has mean 0.5. and $J=1$. $m_Q(t)$ decreases to 0. The computation is performed on the time interval $[0,\,3.0]$ with $\Delta t=0.01$. The Monte Carlo method for the fixed point computation uses 100 000 trajectories.
}
\label{fig:6}
\end{figure}

Figure~\ref{fig:7} shows the results when $\sigma=2.0$ and the initial condition has 0 mean, hence $m_Q(t)=0$ (left part of the figure), or when the initial condition has mean 0.5 (right part of the figure and $J=1$). We 
observe again that  $m_Q(t)$ decreases to 0. The bottom image shows the same case with a zero-mean initial condition. The estimation of the mean which we know to be equal to 0 is quite noisy.
\begin{figure}[htb]
\centerline{
\includegraphics[width=0.5\textwidth]{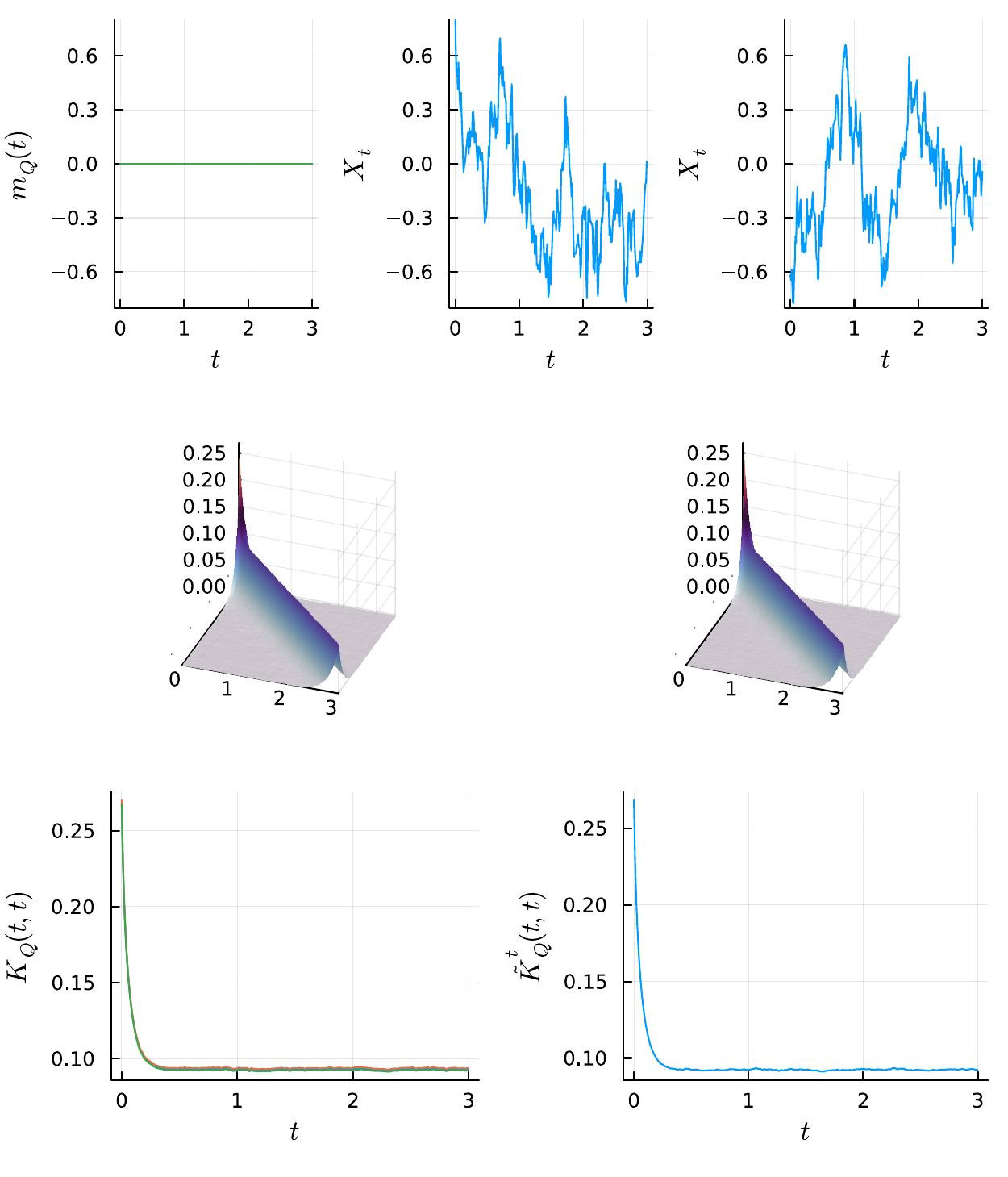} \hspace{0.1cm}
\includegraphics[width=0.5\textwidth]{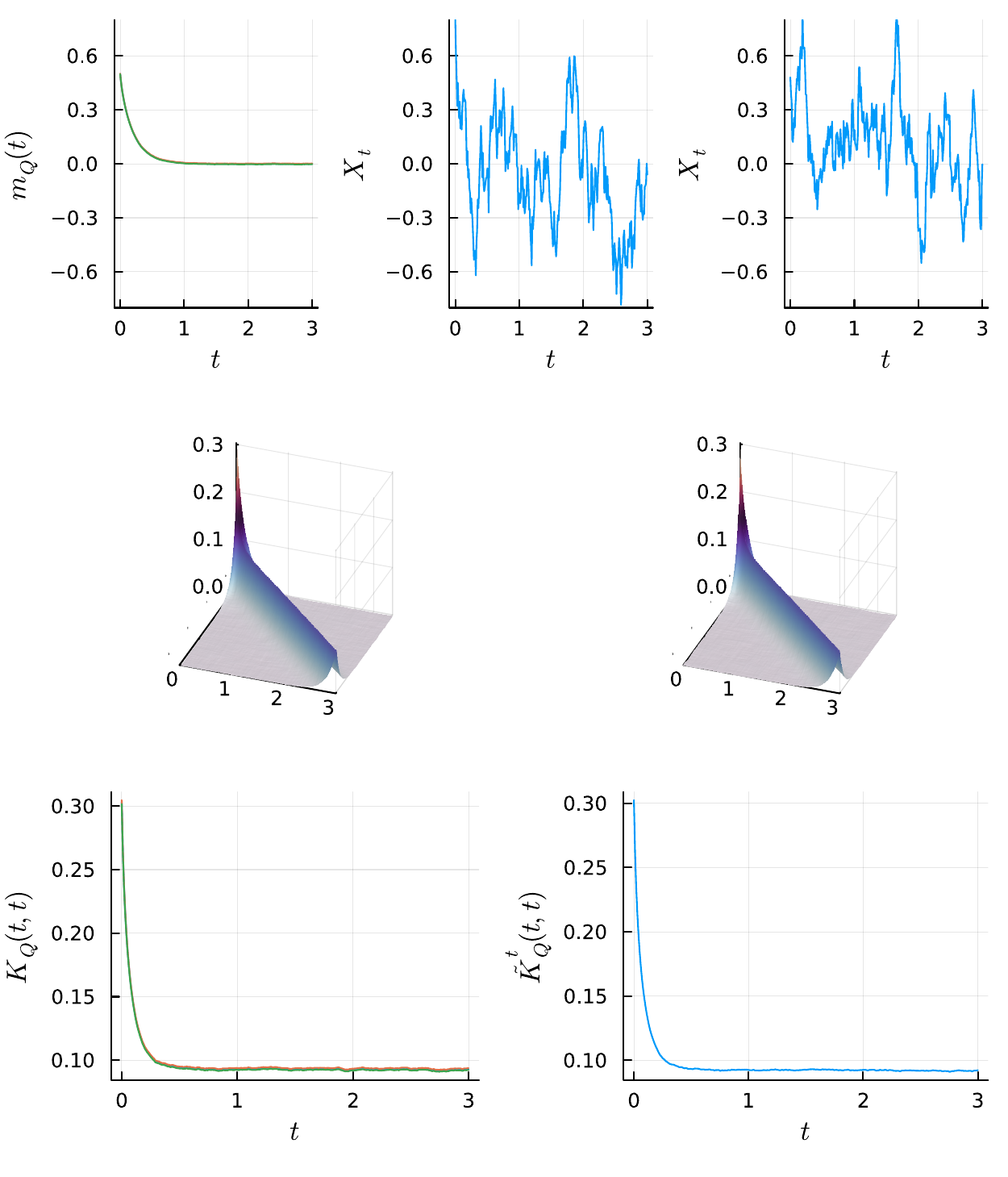}
}
\centerline{
\includegraphics[width=0.5\textwidth]{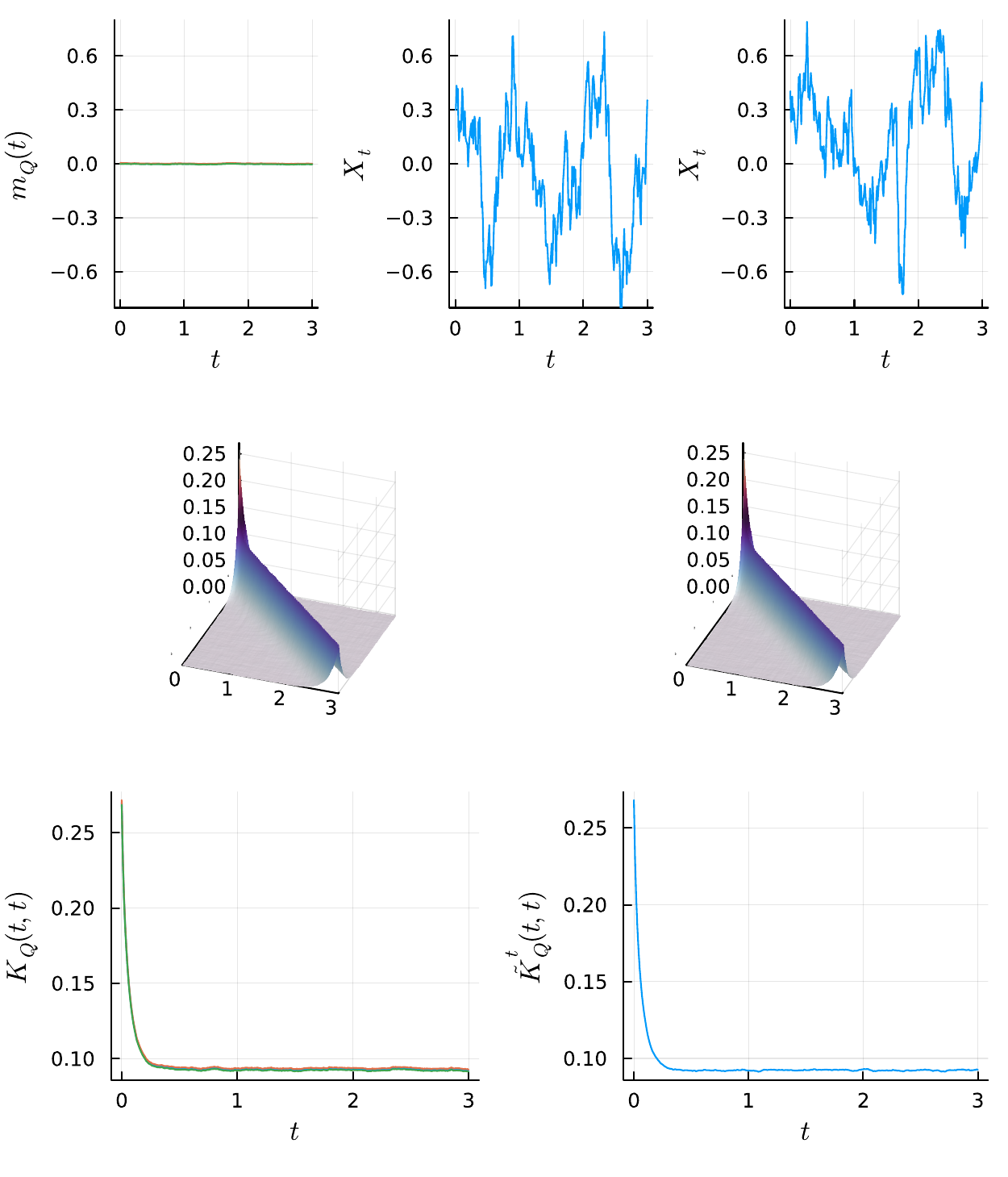}
}
\caption{S-model: $\sigma =2$, $ \lambda=1$. Left: $J=0$. $m_Q(t)=0$ at all times. Right: initial condition has mean 0.5 and $J=1$. $m_Q(t)$ decreases to 0. 
Bottom: $J=1$, zero-mean initial condition. The computation is performed on the time interval $[0,\,3.0]$ with $\Delta t=0.01$. The Monte Carlo method for the fixed point computation uses 100 000 trajectories.
}
\label{fig:7}
\end{figure}

%\subsection{Numerical verification of the universality \texorpdfstring{of equations \eqref{eq:dynmean}}{}}
\subsection{Numerical verification of the universality of equations \eqref{eq:dynmean}}
Thanks to \cite{dembo_universality_2021} we know that the limit measure $Q$ is independent of the actual distribution of the weights $\jij$, subject to the conditions \eqref{eq:dembo-conds}. As a consequence, the equations \eqref{eq:dynmean} that describe the limit dynamics are also independent of this distribution. We illustrate this with the following example. Consider the case of the H-model with Gaussian weights where the $\jij$s are i.i.d. as Gaussians $\mathcal{N}(J/N, \sigma/\sqrt{N})$, and $\lambda = 1$. The results of the estimation of $K_Q$ and $\tilde{K}_Q$ are shown in Figure~\ref{fig:2}. According to the universality property the same limit $Q$ is obtained when the $\jij$s are distributed, e.g., as
\begin{equation}\label{eq:bernoulli}
\jij \simeq \frac{J}{N} + \frac{\sigma}{\sqrt{N}}  \left(\frac{{\rm Ber}_{ij}(p)}{p} - 1 \right) \sqrt{\frac{p}{1-p}},
\end{equation}
where the ${\rm Ber}_{ij}(p)$s are i.i.d. as Bernoulli random variables with parameter $p$, $0 <  p < 1$. It is clear that $\mathbb{E}[\jij] = J/N$ and $var(\jij) = \sigma^2/N$. The limit measures are therefore the same in the two cases \emph{independently} of the value of $p$. We have simulated a fully connected network of $N$ neurons defined by \eqref{eq:network} where the $\jij$s are i.i.d. as \eqref{eq:bernoulli} and estimated $K_Q$ as
\[
K_Q(t,s) = \frac{1}{N} \sum_{i=1}^N f(X^i_t) f(X^i_s)
\]
We show in Figure~\ref{fig:KQcomparison} the covariance functions $K_Q$ estimated from the network equations \eqref{eq:network} and the mean-field dynamics \eqref{eq:dynmean}. The agreement is qualitatively very good.
\begin{figure}[htb]
\centerline{
\includegraphics[width=0.5\textwidth]{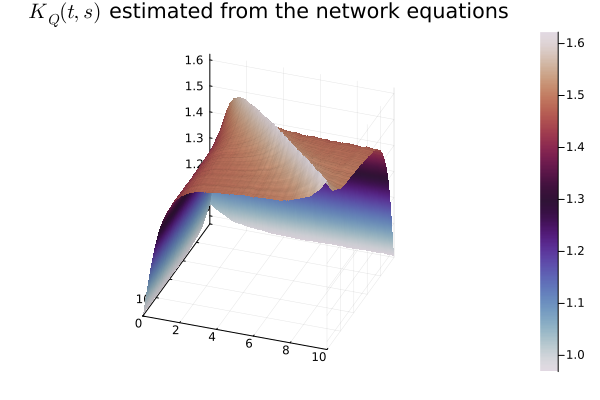} \hspace{0.1cm} 
\includegraphics[width=0.5\textwidth]{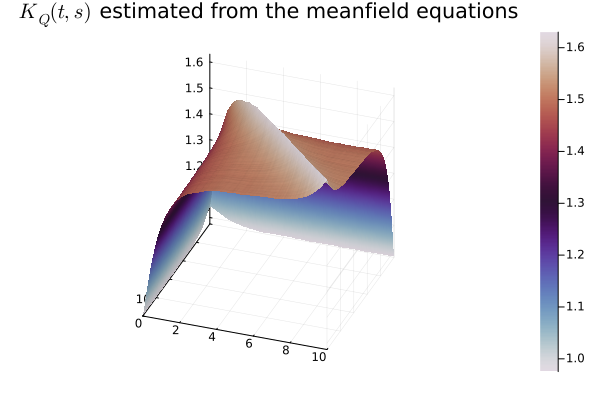}
}
\caption{Comparison of the covariance functions $K_Q(t, s)$ estimated from the network equations \eqref{eq:network} (Left) and from the mean-field dynamics \eqref{eq:dynmean} (Right, same as Figure~\ref{fig:2}, top left). The network size is 30,000. The Bernoulli parameter $p$ is equal to 0.25.}
\label{fig:KQcomparison}
\end{figure}
In order to be slightly more precise in our comparison, we show in Figure~\ref{fig:KQdiagandmeancomparison} the variances $K_Q(t,t)$ (Left) and the means $m_Q(t)$ estimated from the network equations \eqref{eq:network} (Left) and from the mean-field dynamics \eqref{eq:dynmean} (Right) together with the estimated standard deviations.
\begin{figure}[htb]
\centerline{
\includegraphics[width=0.5\textwidth]{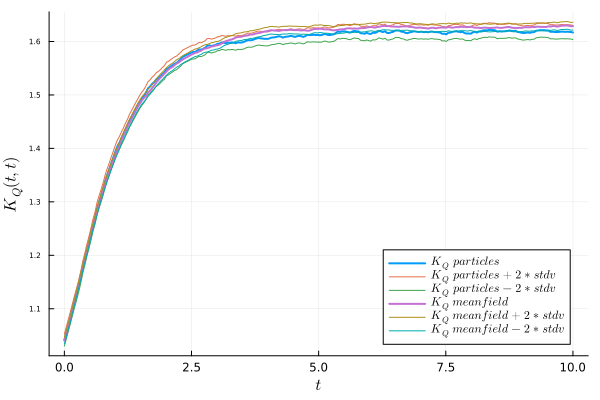} \hspace{0.1cm} 
\includegraphics[width=0.5\textwidth]{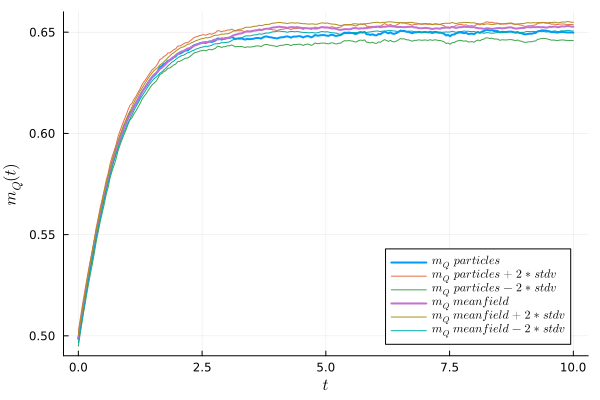}
}
\caption{Left: Comparison of the variance functions $K_Q(t, t)$ estimated from the network equations \eqref{eq:network} and from the mean-field dynamics \eqref{eq:dynmean}. Right:  Comparison of the mean functions $m_Q(t)$. The Bernoulli parameter $p$ is equal to 0.25.}\label{fig:KQdiagandmeancomparison}
\end{figure}

\subsection{Miscellaneous results}
To finish this section of numerical experiments, we show results for two variations of the H-model. The first one is the model of Sompolinsky et al. \cite{sompolinsky-zippelius:82,crisanti-sompolinsky:87,sompolinsky-crisanti-etal:88,crisanti-sompolinsky:18} where the function $f$ is the odd sigmoid $f(x)=\tanh(x)$ shown in Figure~\ref{fig:sigmoid2}. 
\begin{figure}[htb]
\centerline{
\includegraphics[width=0.45\textwidth]{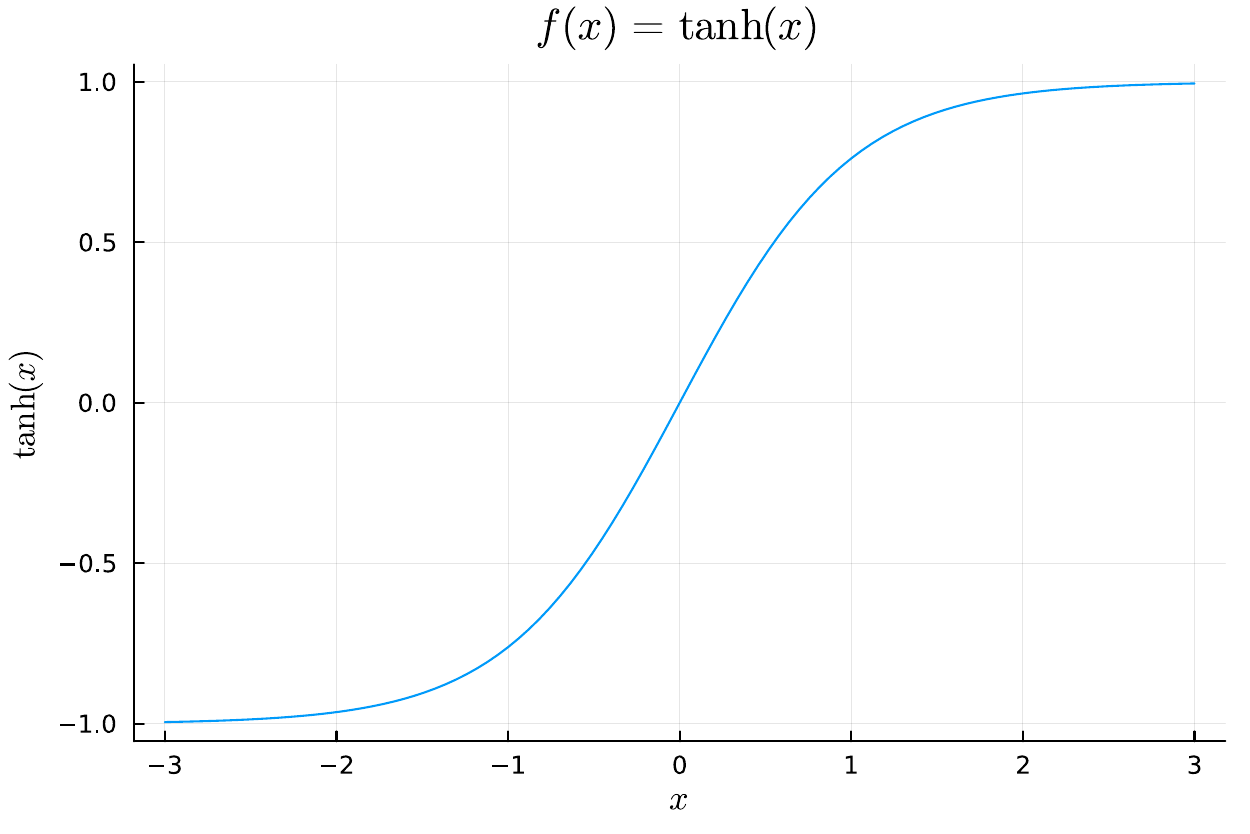}
}
\caption{The sigmoid function used in the work of Sompolinsky et al.. Note that it varies between -1 and 1 and is an odd function, as opposed to that of Figure~\ref{fig:sigmoid1}.}
\label{fig:sigmoid2}
\end{figure}

Because this is an odd function, Corollary~\ref{cor:zeromean} applies and, given a zero-mean initial value, $m_Q(t)=0$ for $0 \leq t \leq T$. This is shown in Figure~\ref{fig:8}.
\begin{figure}[htb]
\centerline{
\includegraphics[width=0.45\textwidth]{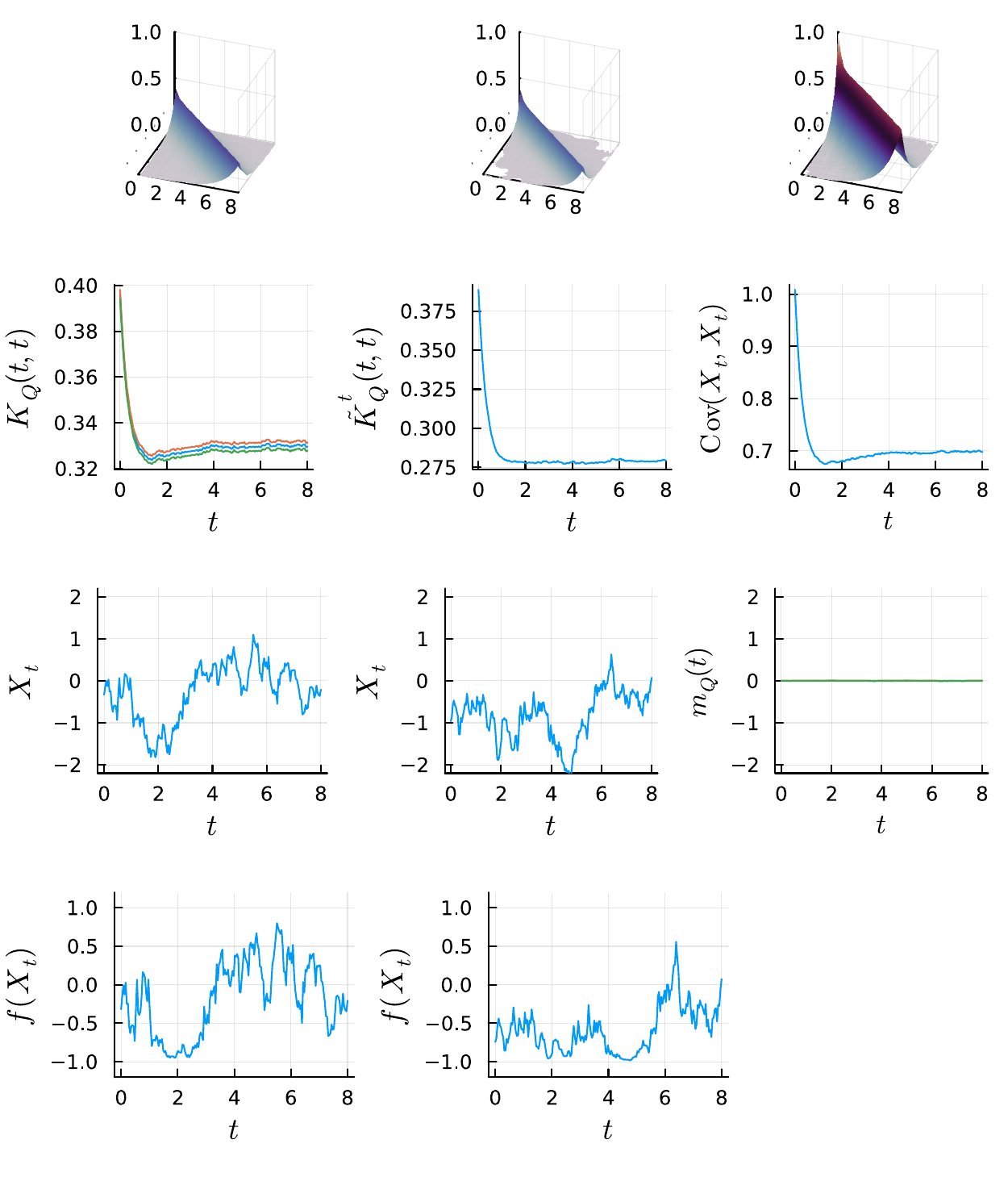}
}
\caption{H-model in the case of an odd sigmoïd function: $J = \sigma = \lambda=1$. Zero mean initial value. $m_Q(t)=0$ at all times. The computation is performed on the time interval $[0,\,8.0]$ with $\Delta t=0.04$. The Monte Carlo method for the fixed point computation uses 100 000 trajectories.
}
\label{fig:8}
\end{figure}

We next show two results when the function $f$ is the ReLu function, $f(x)=\max(x,0)$. It does not satisfy the boundedness hypothesis hence our results do not apply. The reason why we show some results is that this function is used in the field of AI as the main nonlinearity in deep recursive neural networks, see e.g. \cite{lecun_deep_2015,oostwal_hidden_2021}. Figure~\ref{fig:9} shows two results. On the left the excitatory case when $J=1$. The reader will appreciate the fact that $m_Q(t)$ seems to be growing without limit (exploding), probably due to the fact that $f$ is not upbounded, On the right we show the inhibitory case when $J=-1$. The situation seems very different, $m_Q(t)$ levels off, probably due to the fact that $f$ is bounded below. \red{Note that the cases of the ReLu function and of non zero mean synaptic weights have been investigated in the context of the transition to chaos in \cite{kadmon_transition_2015,mastrogiuseppe_intrinsically-generated_2017}.}
\begin{figure}[htb]
\centerline{
\includegraphics[width=0.5\textwidth]{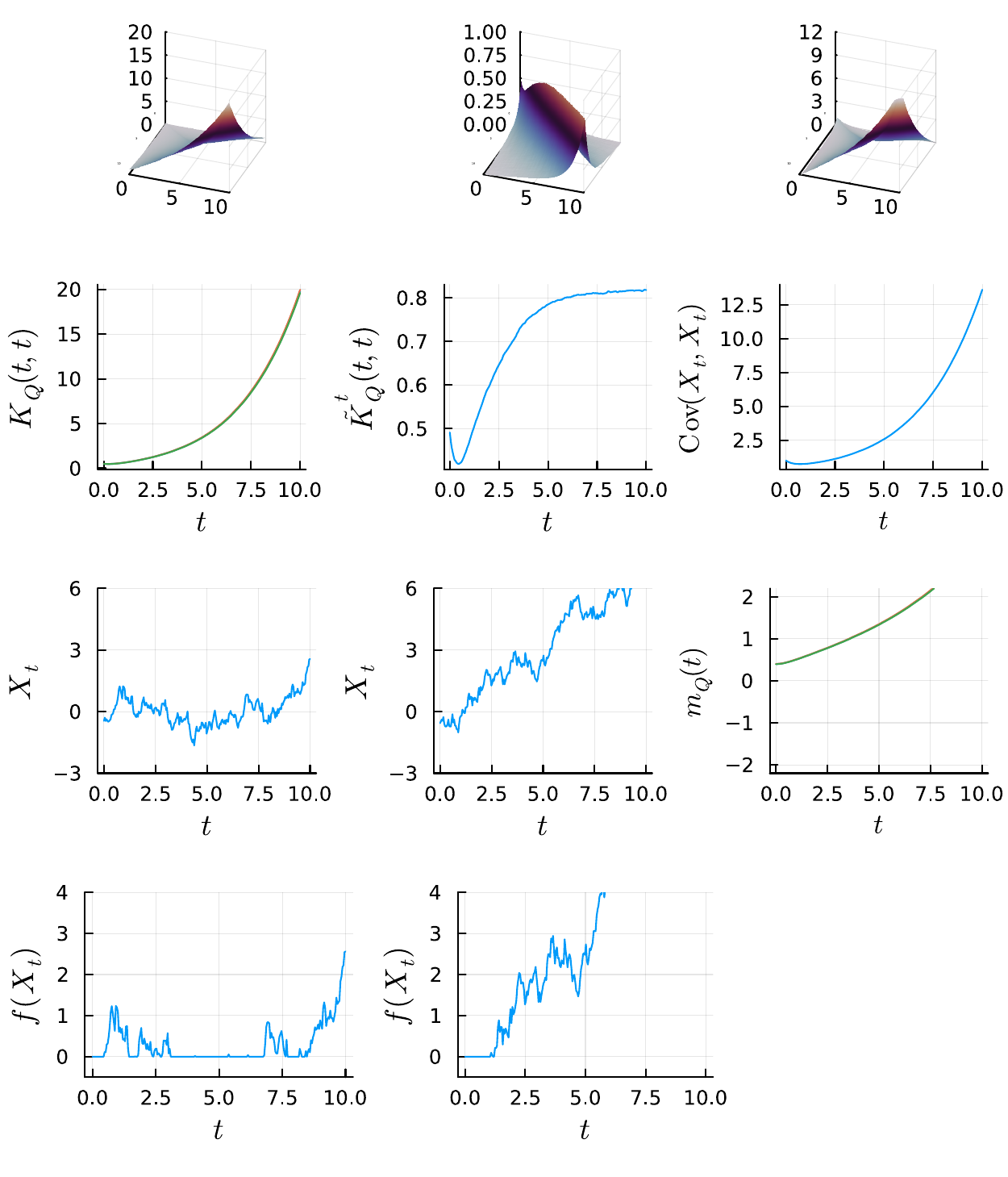}\hspace{0.4cm}
\includegraphics[width=0.5\textwidth]{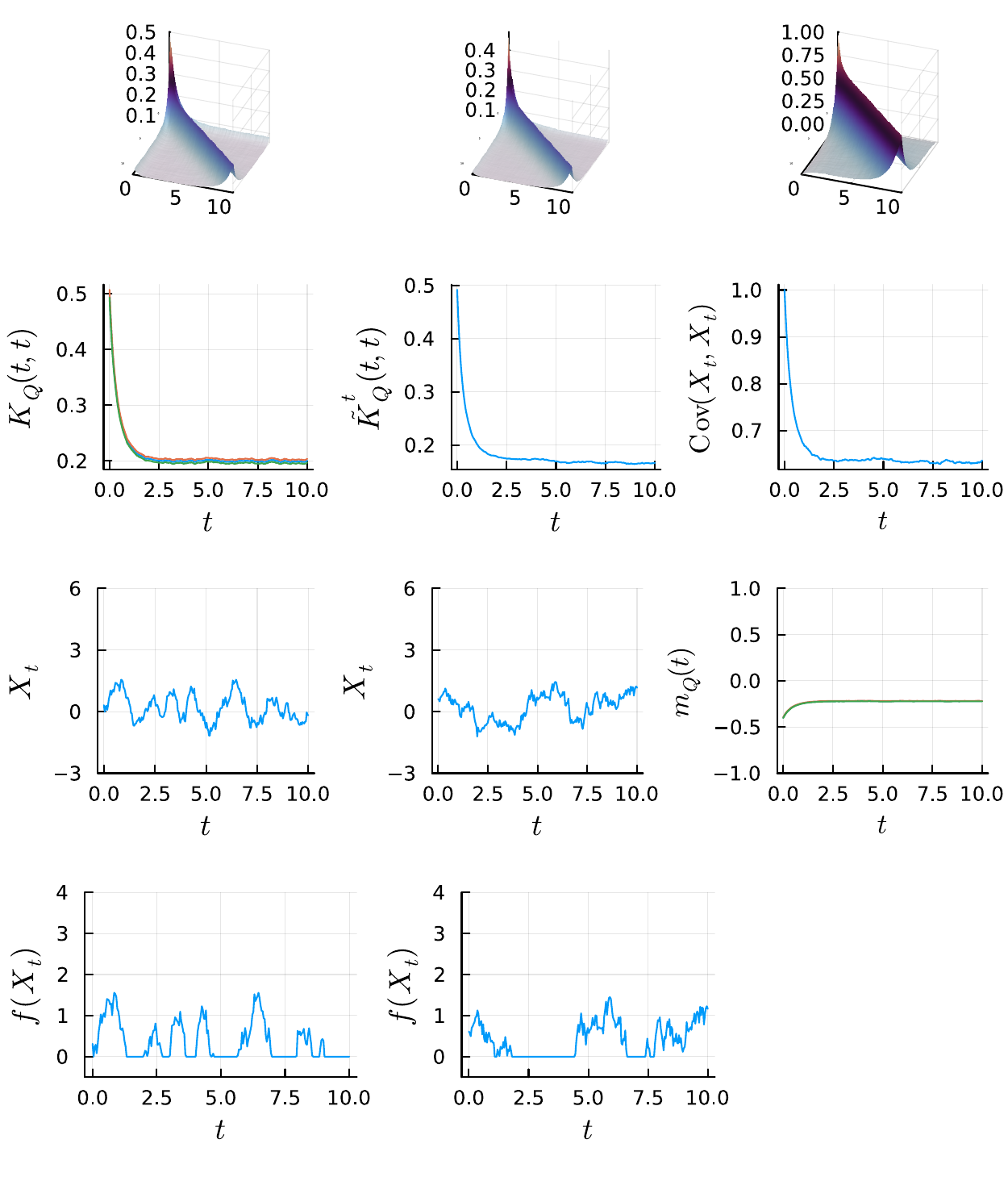}
}
\caption{H-model in the case of the ReLu function. Left $J = \sigma = \lambda=1$. Zero mean initial value. $m_Q(t)$ appears to be growing without limit. Right: $-J = \sigma = \lambda=1$. The computation is performed on the time interval $[0,\,10.0]$ with $\Delta t=0.04$. The Monte Carlo method for the fixed point computation uses 100 000 trajectories.
}
\label{fig:9}
\end{figure}
\vspace{0.5cm}

\noindent
\textbf{Acknowledgements} The authors thank Romain Veltz for his help in developing the Julia code implementing the $\Lambda$ algorithm.
E.T. thanks the support of the ANR project ChaMaNe (ANR-19-CE40-0024, CHAllenges in MAthematical NEuroscience)

\noindent
\textbf{Statements and declarations}
The authors have no competing interests to declare that are relevant to the content of this article.

%\bibliographystyle{plain}
%\bibliography{biblio_paper_ofet}

\begin{thebibliography}{10}
	
	\bibitem{ben-arous-guionnet:95}
	G.~Ben~Arous and A.~Guionnet.
	\newblock Large deviations for {L}angevin spin glass dynamics.
	\newblock {\em Probab. Theory Related Fields}, 102(4):455--509, 1995.
	
	\bibitem{borodin_stochastic_2017}
	A.~N. Borodin.
	\newblock {\em Stochastic processes}.
	\newblock Probability and its Applications. Birkh\"auser/Springer, Cham, 2017.
	
	\bibitem{cabana-touboul:18b}
	T.~Cabana and .~D. Touboul.
	\newblock Large deviations for randomly connected neural networks: {II}.
	{S}tate-dependent interactions.
	\newblock {\em Adv. in Appl. Probab.}, 50(3):983--1004, 2018.
	
	\bibitem{cabana-touboul:13}
	T.~Cabana and J.~D. Touboul.
	\newblock Large deviations, dynamics and phase transitions in large stochastic
	and disordered neural networks.
	\newblock {\em J. Stat. Phys.}, 153(2):211--269, 2013.
	
	\bibitem{cabana-touboul:18}
	T.~Cabana and J.~D. Touboul.
	\newblock Large deviations for randomly connected neural networks: {I}.
	{S}patially extended systems.
	\newblock {\em Adv. in Appl. Probab.}, 50(3):944--982, 2018.
	
	\bibitem{cameron-martin:44}
	R.~H. Cameron and W.~T. Martin.
	\newblock Transformations of {W}iener integrals under translations.
	\newblock {\em Ann. of Math. (2)}, 45:386--396, 1944.
	
	\bibitem{crisanti-sompolinsky:87}
	A.~Crisanti and H.~Sompolinsky.
	\newblock Dynamics of spin systems with randomly asymmetric bonds: {L}angevin
	dynamics and a spherical model.
	\newblock {\em Phys. Rev. A (3)}, 36(10):4922--4939, 1987.
	
	\bibitem{crisanti-sompolinsky:18}
	A.~Crisanti and H.~Sompolinsky.
	\newblock Path integral approach to random neural networks.
	\newblock {\em Phys. Rev. E}, 98:062120, Dec 2018.
	
	\bibitem{dembo_universality_2021}
	A.~Dembo, E.~Lubetzky, and O.~Zeitouni.
	\newblock Universality for {L}angevin-like spin glass dynamics.
	\newblock {\em Ann. Appl. Probab.}, 31(6):2864--2880, 2021.
	
	\bibitem{ethier_markov_2009}
	S.~N. Ethier and T.~G. Kurtz.
	\newblock {\em Markov processes}.
	\newblock Wiley Series in Probability and Mathematical Statistics: Probability
	and Mathematical Statistics. John Wiley \& Sons, Inc., New York, 1986.
	\newblock Characterization and convergence.
	
	\bibitem{faugeras_constructive_2009}
	O.~Faugeras, J.~Touboul, and B.~Cessac.
	\newblock A constructive mean-field analysis of multi population neural
	networks with random synaptic weights and stochastic inputs.
	\newblock {\em Front. Comput. Neurosci.}, 3, 2009.
	
	\bibitem{girsanov:60}
	I.~V. Girsanov.
	\newblock On {T}ransforming a {C}ertain {C}lass of {S}tochastic {P}rocesses by
	{A}bsolutely {C}ontinuous {S}ubstitution of {M}easures.
	\newblock {\em Theory of Probability \& Its Applications}, 5(3):285--301, 1960.
	
	\bibitem{guionnet:95}
	A.~Guionnet.
	\newblock {\em Dynamique de Langevin d'un verre de spins}.
	\newblock PhD thesis, Universit\'e Paris 11, 1995.
	
	\bibitem{guionnet:97}
	A.~Guionnet.
	\newblock Averaged and quenched propagation of chaos for spin glass dynamics.
	\newblock {\em Probab. Theory Related Fields}, 109(2):183--215, 1997.
	
	\bibitem{hopfield:82}
	J.~J. Hopfield.
	\newblock Neural networks and physical systems with emergent collective
	computational abilities.
	\newblock {\em Proc. Nat. Acad. Sci. U.S.A.}, 79(8):2554--2558, 1982.
	
	\bibitem{hopfield:84}
	J.~J. Hopfield.
	\newblock Neurons with graded response have collective computational properties
	like those of two-state neurons.
	\newblock {\em Proceedings of the National Academy of Sciences},
	81(10):3088--3092, 1984.
	
	\bibitem{hopfield-tank:86}
	J.~J. Hopfield and D.~W. Tank.
	\newblock Computing with neural circuits: A model.
	\newblock {\em Science}, 233(4764):625--633, 1986.
	
	\bibitem{kadmon_transition_2015}
	Jonathan Kadmon and Haim Sompolinsky.
	\newblock Transition to {Chaos} in {Random} {Neuronal} {Networks}.
	\newblock {\em Physical Review X}, 5(4):041030, November 2015.
	
	\bibitem{lecun_deep_2015}
	Y.~LeCun, Y.~Bengio, and G.~Hinton.
	\newblock Deep learning.
	\newblock {\em Nature}, 521(7553):436--444, 2015.
	
	\bibitem{mastrogiuseppe_intrinsically-generated_2017}
	F.~Mastrogiuseppe and S.~Ostojic.
	\newblock Intrinsically-generated fluctuating activity in excitatory-inhibitory
	networks.
	\newblock {\em PLoS Comput. Biol.}, 13(4):1--40, 2017.
	
	\bibitem{moynot-samuelides:02}
	O.~Moynot and M.~Samuelides.
	\newblock Large deviations and mean-field theory for asymmetric random
	recurrent neural networks.
	\newblock {\em Probab. Theory Related Fields}, 123(1):41--75, 2002.
	
	\bibitem{oostwal_hidden_2021}
	E.~Oostwal, M.~Straat, and M.~Biehl.
	\newblock Hidden unit specialization in layered neural networks: {R}e{LU} vs.
	sigmoidal activation.
	\newblock {\em Phys. A}, 564:Paper No. 125517, 14, 2021.
	
	\bibitem{peskir_optimal_2006}
	G.~Peskir and A.~Shiryaev.
	\newblock {\em Optimal stopping and free-boundary problems}.
	\newblock Lectures in Mathematics ETH Z\"urich. Birkh\"auser Verlag, Basel,
	2006.
	
	\bibitem{reed_functional_1981}
	M.~Reed and B.~Simon.
	\newblock {\em Functional {Analysis}, ({Methods} of {Modern} {Mathematical}
		{Physics})}, volume~1.
	\newblock Academic Press, 1981.
	
	\bibitem{rogers_diffusions_2000}
	L.~C.~G. Rogers and D.~Williams.
	\newblock {\em Diffusions, {M}arkov processes, and martingales. {V}ol. 2}.
	\newblock Cambridge Mathematical Library. Cambridge University Press,
	Cambridge, 2000.
	
	\bibitem{sompolinsky-crisanti-etal:88}
	H.~Sompolinsky, A.~Crisanti, and H.-J. Sommers.
	\newblock Chaos in random neural networks.
	\newblock {\em Phys. Rev. Lett.}, 61(3):259--262, 1988.
	
	\bibitem{sompolinsky-zippelius:82}
	H.~Sompolinsky and A.~Zippelius.
	\newblock Relaxational dynamics of the {E}dwards-{A}nderson model and the
	mean-field theory of spin-glasses.
	\newblock {\em Phys. Rev. B}, 25:6860--6875, 1982.
	
	\bibitem{tricomi:57}
	F.~G. Tricomi.
	\newblock {\em Integral equations}, volume Vol. V.
	\newblock Interscience Publishers, Inc., New York; Interscience Publishers
	Ltd., London, 1957.
	
	\bibitem{meegen-kuhn-etal:21}
	A.~van Meegen, T.~K\"{u}hn, and M.~Helias.
	\newblock Large-deviation approach to random recurrent neuronal networks:
	parameter inference and fluctuation-induced transitions.
	\newblock {\em Phys. Rev. Lett.}, 127(15):Paper No. 158302, 6, 2021.
	
\end{thebibliography}

\bigskip
\footnotesize

O.~Faugeras, \\
\textsc{
	Centre Inria d'Université Côte d'Azur, EPI CRONOS\\
	2004, route des Lucioles,
	BP 93\\
	06902 Sophia Antipolis  Cedex FRANCE}\par\nopagebreak \texttt{Olivier.Faugeras@inria.fr}

\vspace{5mm}
E.~Tanr\'e, \\
\textsc{Laboratoire J.A.Dieudonné\\
	UMR CNRS-UNS N°7351\\
	Université Côte d'Azur\\
	Parc Valrose\\
	06108 NICE Cedex 2\\
	FRANCE}\par\nopagebreak \texttt{Etienne.Tanre@inria.fr}
}

\end{document}